\documentclass[a4paper]{amsart} 
\usepackage[left = 2.5cm, right = 2.50cm, top = 2.5cm, bottom = 2.5cm]{geometry}

\usepackage[utf8]{inputenc}
\usepackage[english]{babel}
\usepackage[T1]{fontenc}
\usepackage{lmodern}

\usepackage{amsmath, amsfonts, amssymb,  mathabx, mathrsfs, mathtools, dsfont}

\usepackage{graphicx}

\usepackage{tikz,siunitx}
\usetikzlibrary{shapes.geometric,shapes.symbols}

\graphicspath{./}

\usepackage{ulem}
\usepackage{comment}
\usepackage{enumitem}
\usepackage{xcolor}
\usepackage{stmaryrd}
\usepackage{subfig}

\newcommand{\PP}{\mathbb{P}}

\definecolor{darkpastelgreen}{rgb}{0.1, 0.8, 0}

\usepackage{multirow, multicol, hhline}

\usepackage{adjustbox}

\usepackage{hyperref}
\usepackage[noabbrev,capitalize]{cleveref}
\hypersetup{%
colorlinks=true, 								
linkcolor= red,									
anchorcolor=black,
citecolor= red,									
urlcolor = red									
}
\newcommand{\rev}{}
\newcommand{\revv}{}
\newcommand{\kk}{\kappa}

\newcommand{\hht}{{\Delta t}}
\newcommand{\hhx}{{\Delta \x}}
\newcommand{\hhp}{{\Delta \p}}
\newcommand{\hh}{{\triangle}}

\newcommand{\mN}{\mathcal{N}}

\newcommand{\Sym}{\mathbb{S}}
\newcommand{\bp}{\mathbf{p}}

\newcommand{\ind}{\mathds{1}}
\newcommand{\ou}{\bar{u}}
\newcommand{\ov}{\bar{v}}
\newcommand{\tu}{\tilde{u}}

\newcommand{\ie}{i.e.\;}

\newcommand{\mM}{\mathcal{M}}

\newcommand{\hx}{\hat{\x}}

\newcommand{\dt}{\partial_t}
\newcommand{\mL}{{ \left(-\partial_t - \mathcal{L}\right)}}
\newcommand{\DI}{\Delta(I)}
\newcommand{\RR}{\mathbb{R}}

\newcommand{\NN}{\mathbb{N}}
\newcommand{\DD}{\mathrm{D}}
\newcommand{\ot}{\bar{t}}
\newcommand{\ox}{\bar{\x}}
\newcommand{\oX}{\widebar{\X}}
\newcommand{\oY}{\widebar{\Y}}

\newcommand{\op}{\bar{\p}}
\newcommand{\EE}{\mathbb{E}}

\newcommand{\dd}{\mathrm{d}}
\newcommand{\mF}{\mathcal{F}}
\newcommand{\FF}{\mathbb{F}}
\newcommand{\bi}{\mathbf{i}}

\newcommand{\ee}{\mathrm{e}}
\newcommand{\mO}{\mathcal{O}}
\newcommand{\mT}{\mathcal{T}}
\newcommand{\mS}{\mathcal{S}}

\newcommand{\mB}{\mathcal{B}}
\newcommand{\mD}{\mathcal{D}}
\newcommand{\x}{\mathrm{x}}
\newcommand{\y}{\mathrm{y}}
\newcommand{\z}{\mathrm{z}}

\newcommand{\X}{\mathrm{X}}
\newcommand{\Y}{\mathrm{Y}}
\newcommand{\p}{\mathrm{p}}
\newcommand{\q}{\mathrm{q}}

\newcommand{\PPas}{\PP\mbox{-a.s.}}

\newcommand{\T}{\intercal}
\newcommand{\pnx}{\p^{n,\x}}
\newcommand{\Mnx}{M^{n,\x}}
\newcommand{\mMnx}{\mM^{n,\x}}
\newcommand{\mNnx}{\mN^{n,\x}}

\newcommand{\mtNnxk}{\widetilde{\mN}^{n,\x_\ell,\kappa}}
\newcommand{\mtMnxk}{\widetilde{\mM}^{n,\x_\ell,\kappa}}

\newcommand{\psinx}{\psi^{n,\x}}

\newcommand{\pinxl}{\pi^{n,\x_\ell}}
\newcommand{\pinx}{\pi^{n,\x}}
\newcommand{\pinX}{\pi^{n,\oXnx_{n+n'}}}
\newcommand{\varpinx}{\varpi^{n,\x}}
\newcommand{\lambdanx}{\lambda^{n,\x}}
\newcommand{\lambdanxl}{\lambda^{n,\x_\ell}}
\newcommand{\lambdanX}{\lambda^{n,\oXnx_{n+n'}}}

\newcommand{\tpnxk}{\tilde{\p}^{n,\x_\ell,\kk}}
\newcommand{\tpinxk}{\tilde{\pi}^{n,\x,\kk}}
\newcommand{\tlambdanxk}{\tilde{\lambda}^{n,\x,\kk}}

\newcommand{\oXnx}{\oX^{n,\x}}
\newcommand{\oXnxl}{\oX^{n,\x_\ell}}
\newcommand{\oXny}{\oX^{n,\y}}

\newcommand{\mLx}{{\ell ip}_\x}
\newcommand{\mLt}{{ho\ell}_t}
\newcommand{\mLp}{{\ell ip}_\p}

\newcommand{\mNN}{\mathcal{N}^\rho_{d,1,\lambda,\kappa}}
\newcommand{\mGG}{\mathcal{G}^{\rho^s}_{\gamma,d,1,\lambda,\kappa}}
\newcommand{\mGGn}{\mathcal{G}^{\rho^{s,n+1}}_{\gamma^{n+1},d,1,\lambda,\kappa}}
\newcommand{\mSS}{\mathcal{S}^{\rho^s}_{d,\lambda,\kappa}}
\newcommand{\mSSn}{\mathcal{S}^{\rho^{s,n+1}}_{d,\lambda,\kappa}}
\newcommand{\mNNs}{\mathcal{N}^{\rho^s}_{d,1,\lambda,\kappa}}
\newcommand{\ouhkRL}{\ouh_{\kk,\mD}}
\newcommand{\ovhkRL}{\ovh_{\kk,\mD}}
\newcommand{\tuhRL}{\tu^\hh_{\mD}}

\newcommand{\ouh}{{\ou^\hh}}
\newcommand{\ovh}{{\ov^\hh}}
\newcommand{\tuh}{\tu^\hh}

\newcommand{\epskrL}{\varepsilon_{\kk,\mD}}

\DeclareMathOperator{\Tr}{Tr}
\DeclareMathOperator{\vexp}{Vex_p}
\DeclareMathOperator{\Int}{Int}
\DeclareMathOperator{\diam}{diam}
\DeclareMathOperator*{\argmin}{\arg\min}
\DeclareMathOperator*{\argmax}{\arg\max}
\newtheorem{theorem}{Theorem}[section]
\newtheorem{defn}[theorem]{Definition}
\newtheorem{lem}[theorem]{Lemma}

\newtheorem{prop}[theorem]{Proposition}
\newtheorem{cor}[theorem]{Corollary}

\newtheorem{algo}{Algorithm}[section]
\newtheorem{rem}[theorem]{Remark}

\numberwithin{equation}{section} 
\numberwithin{table}{section} 
\Crefname{lem}{Lemma}{Lemmas}
\Crefname{figure}{Figure}{Figures}
\Crefname{equation}{}{}
\Crefname{defn}{Definition}{Definitions}
\Crefname{prop}{Proposition}{Propositions}
\Crefname{algpseudocode}{Algorithm}{Algorithms}
\title{Numerical approximation of Dynkin games with asymmetric information}

\author{\v{L}ubom\'{i}r Ba\v{n}as}
\address{Department of Mathematics, Bielefeld University, Germany}
\email{banas@math.uni-bielefeld.de}

\author{Giorgio Ferrari}
\address{Center for Mathematical Economics, Bielefeld University, Germany}		
\email{giorgio.ferrari@uni-bielefeld.de}

\author{Tsiry Avisoa Randrianasolo}
\address{Chair of Applied Mathematics, Montanuniversität Leoben, Austria}		
\email{tsiry.randrianasolo@outlook.com}

\begin{document}
\begin{abstract}
We propose an implementable, neural network-based structure preserving probabilistic numerical approximation for a generalized obstacle problem describing the value of a zero-sum differential game of optimal stopping with asymmetric information.
The target solution depends on three variables: the time, the spatial (or state) variable, and a variable from a standard $(I-1)$-simplex which represents the probabilities with which the $I$ possible configurations of the game are played.
The proposed numerical approximation preserves the convexity of the continuous solution as well as the lower and upper obstacle bounds.
We show convergence of the fully-discrete scheme to the unique viscosity solution of the continuous problem and present a range of numerical studies to demonstrate its applicability.

\end{abstract}

\maketitle

\section{Introduction}\label{sec:intro}


We consider the generalized obstacle problem

\begin{equation}\label{eq:hjb}
\begin{split}
\max\Big\{\max \Big\{\min\Big\{\mL u ,\,u- \p^\T f\Big\},\,u- \p^\T h\Big\},\,-\lambda(\p,\DD_{\p}^2u)\Big\} & = 0,
\\
u(T,\;\x,\;\p) & =  \p^\T g(\x),
\end{split}
\end{equation}
where $u: [0,T]\times\RR^d\times\DI \rightarrow \RR$, $T>0$ and $\DI = \{\p=(p_1,\ldots,p_I)^\T\in [0,1]^I;\ \sum_{i =1}^{I}p_i = 1 \}$ is the set of $\mathbb{R}^{I}$-valued vectors of probabilities (see below for a detailed explanation).
Furthermore, $\mathcal{L} u :=  \tfrac12 \Tr \big[aa^\T \DD_{\x}^2 u\big] +  b\cdot\DD_{\x} u$, where the coefficient functions { $a:[0,T]\times\RR^d\rightarrow \mathbb{R}^{d\times d}$, $b:[0,T]\times\RR^d\rightarrow \mathbb{R}^d$},
as well as the barrier functions $	f,h:[0,T]\times\RR^d\rightarrow\RR^I$ and the terminal condition $g:\RR^d\rightarrow\RR^I$
are given.  The mapping $\lambda:\DI\times {\Sym^{d\times d}}\rightarrow\RR$ enforces the convexity of the solution $u\equiv u(t,\x,\p)$ with respect to the probability variable $\p$ and is given as
\begin{equation*}
\lambda(\p,A) := \min_{\z\in T_{\DI}(\p)\setminus \{0\}} \frac{\z^\T A\;\z}{\hphantom{^2}\lvert \z\rvert^2},
\end{equation*}
where $\Sym^{d\times d}$ denotes the set of $d\times d$ matrices that are symmetric and  $\mathrm{T}_{\DI}(\p)$ denotes the tangent cone to $\DI$ at $\p$; i.e., $\mathrm{T}_{\DI}(\p):=\overline{\bigcup_{r>0}\tfrac{1}{r}(\DI-\p)}$.

{Under suitable assumptions on the problem's data, 
it is shown in \cite[Theorem 3.4]{gruen2013on} that
\Cref{eq:hjb} admits a unique viscosity solution in the class of bounded, uniformly continuous functions that are Lipschitz-continuous  and convex with respect to $\p$.
The viscosity solution identifies with the value of a zero-sum game of optimal stopping (Dynkin game) with asymmetric information.}
More precisely, given a stochastic basis $(\Omega, \mF, \FF, \PP)$, with the augmented canonical filtration $\FF:= \{\mF_s: t\leq s\leq T\}$ generated by an $\RR^d$-valued Brownian motion $B:= \{B_s: t\leq s\leq T \}$, operator $\mL$ 
is the so-called Dynkin operator associated to the It\^o-process $\X^{t,\x}\coloneqq\{\X^{t,\x}_s:\, 0 \leq t\leq s\leq T\}$ defined on $(\Omega, \mF, \FF, \PP)$ and satisfying
\begin{equation}\label{eq:state}
\X^{t,\x}_s = \x + \int_{t}^{s}b(r, \X^{t,\x}_r)\dd r + \int_{t}^{s}a(r, \X^{t,\x}_r)\dd B_r.
\end{equation}

Consider then the zero-sum game of optimal stopping in which -- under a given scenario $i\in I$ selected randomly with probability  $p_i$ at time $0$ -- two players decide to stop the evolution of the process $\X^{t,\x}$ in order to optimize the performance criterion:
\begin{equation}
\label{eq:functional}
\EE\Big[f_i(\sigma, \X^{t,\x}_{\sigma})\mathds{1}_{\sigma < \tau, \sigma<T} + h_i(\tau, \X^{t,\x}_{\tau})\mathds{1}_{\tau \leq \sigma, \tau<T} + g_i(\X^{t,\x}_{T})\mathds{1}_{\tau =\sigma=T}\Big], \quad i \in \{1, 2, \dots, I\}.
\end{equation}
Exercising a stopping rule $\tau$, Player 1 pays Player 2 the random amount $h_i(\tau, \X^{t,\x}_{\tau})$, while stopping at time $\sigma$, Player 2 receives the random payoff $f_i(\sigma, \X^{t,\x}_{\sigma})$. Finally, if neither of the players has stopped before the final maturity $T$, then {\rev Player 1 pays Player 2} the amount $g_i(\X^{t,\x}_{T})$. Clearly, Player 1 aims at minimizing \eqref{eq:functional}, while Player 2 at maximizing \eqref{eq:functional}.
The asymmetric information feature of the game arises because only one player knows the exact scenario under which the game is played - and therefore the realized values of the game's payoffs $f_i, h_i, g_i$ -- while the other player is informed only about the probability according to which each of the possible scenarios is selected.

Zero-sum games of optimal stopping have been proposed by E.B.\ Dynkin in 1967 (see \cite{Dynkin1967}) as an extension of problems of optimal stopping. Since their introduction a large number of contributions using probabilistic and/or analytic techniques arose, with the main aim of proving existence of a value for the game and a characterization of the (Stackelberg) equilibrium stopping times. We refer to the introduction of \cite{gruen2013on} for a detailed literature review. In particular, Dynkin games have received increasing attention in the mathematical finance literature since the work by Y.\ Kifer \cite{Kifer}, where it is shown that zero-sum games of optimal stopping provide the fair value of the so-called Game or Israeli Options. 

The literature on Dynkin games with asymmetric information is quite recent and still only counts a very limited number of contributions. Other than \cite{gruen2013on}, we like to refer to \cite{LempaMatomaki}, where the value and the equilibrium strategies of a zero-sum game of optimal stopping have been constructed in the case in which both players have different knowledge about the occurrence of a default. In \cite{dynkin19}, it is considered a zero-sum game of optimal stopping where the payoffs depend on two independent continuous-time Markov chains, the first Markov chain being observed only by Player 1, the second Markov chain being observed only by Player 2. This in particular implies asymmetric information, in the sense that players employ stopping rules that are stopping times with respect to different filtrations. More recently, in a very general not necessarily Markovian setting, it is proved in \cite{DeAMePa} that continuous-time zero-sum Dynkin games with partial and asymmetric information admit a value in randomized stopping times. As a byproduct, existence of equilibrium strategies for both players are also shown to exist. Finally, explicit results via free-boundary methods have been obtained in \cite{dynkin22} for a class of games in which the players have asymmetric information regarding the drift of the one-dimensional diffusion underlying the game.

Concerning the numerical approximation of games with asymmetric information we refer to \cite{banas2020numerical}, \cite{gruen2012aprobabilistic} and the references therein.
Furthermore, we mention  \cite{hure2020deep} which studies probabilistic neural network approximations schemes for obstacle problems arising from game theory (with complete and symmetric information).

In the current paper we propose a structure-preserving probabilistic numerical approximation of \eqref{eq:hjb}.
Following \cite{banas2020numerical}, we combine the time-discretization with a convexity-preserving discretization in the probability variable.
Furthermore, we introduce a neural network approximation in the spatial variable to obtain an implementable numerical scheme and show its convergence to the viscosity solution to \eqref{eq:hjb}.
We present numerical studies where we compare the the proposed neural network-based algorithm to a semi-Lagrangian scheme with piecewise linear interpolation in the spatial variable.
Furthermore, we demonstrate the ability of the proposed numerical approximation to capture the expected structure of free boundaries 
arising in the problem of pricing Israeli $\delta$-penalty put option (cf.\ \cite{KuhnKyprianou}, \cite{Kyprianou} among others) with asymmetric information.


The paper is organized as follows. The notation and preliminaries are introduced in \Cref{sec:prelimiaries}.
In \cref{sec:TimeDiscret} we introduce the semi-discrete and the fully-discrete probabilistic numerical approximation schemes and show their convergence to the viscosity solution of \eqref{eq:hjb}.
Sections~\ref{sec:cont}~and~\ref{sec:viscosityPty} contain auxiliary results needed for the convergence of the semi-discrete scheme:
in Section~\ref{sec:cont} we discuss regularity properties of the semi-discrete numerical solution and in \Cref{sec:viscosityPty} we show that the accumulation points of the semi-discrete numerical approximation satisfy the viscosity sub- and supersolution properties.
Regularity properties of the fully-discrete solution are studied in Section~\ref{sec_full}.
Numerical experiments are presented in \Cref{sec:num_res}.


\section{Notation and preliminaries}\label{sec:prelimiaries}
\subsection{Notation}\label{subsec:notation}
We list here some notation in complement to those introduced in \Cref{sec:intro}:
\begin{itemize}
\item Let $m\in\NN$. For any $\x = (x_1,\ldots,x_m)^\T\in\RR^m$, we denote the Euclidean norm of $\x$ by
\[
\lvert \x\rvert := \sqrt{\lvert x_1\rvert^2 + \ldots + \lvert x_m\rvert^2}.
\]
\item For any Lipschitz continuous function $\varphi:[0,T]\times\RR^d \rightarrow \RR^{m}$, we denote the Lipschitz coefficient by 
\[
\llbracket\varphi\rrbracket :=\sup_{(t,\x)\neq (s,\y)}\frac{\lvert \varphi(t,\x) -\varphi(s,\y)\rvert}{\lvert t-s\rvert + \lvert \x-\y\rvert}.
\]
\item A continuous function $\varphi:[0,T]\times\RR^d\times\DI\rightarrow \RR^{m}$ is said to be bounded if the norm
\[
\lVert \varphi\rVert_\infty :=\sup\big\{\lvert \varphi(t,\x,\p)\rvert: (t,\x,\p)\in [0,T]\times\RR^d\times\DI\big\}
\]
is finite.
\item For $t\in E\subseteq[0,T]$, we denote by $\mT_{E}$ the set of $\FF$-stopping times with values in $E$. The subset of stopping rules will be denoted by
\[
\mT_{E}^t :=\big\{\tau\in\mT_{E}:\tau \mbox{ is independent of }\mF_t\big\}.
\]
\end{itemize}
We will frequently use the fact that the maps $(\x,\y)\mapsto \max\{\x,\y\}$ and $(\x,\y)\mapsto \min\{\x,\y\}$ satisfy for every $(\x,\x')$ and $(\y,\y')$ in $\RR^2$ the following inequalities
{
\begin{align}
\label{eq:lipx7}
\lvert \max\{\x,\x'\}-\max\{\y,\y'\}\rvert&\leq \max\{\lvert \x-\y \rvert,\, \lvert \x'-\y' \rvert\},
\\
\label{eq:lipx8}
\lvert \min\{\x,\x'\}-\min\{\y,\y'\}\rvert&\leq \max\{\lvert \x-\y \rvert,\, \lvert \x'-\y' \rvert\}.
\end{align}
}
\subsection{Viscosity solution}\label{subsec:gnlAssumptions}
The following assumptions are required to hold for the data in \Cref{eq:hjb}:
\begin{enumerate}[label=($A_\arabic*)$]
\item{\label{A1}} The functions $a:[0,T]\times\RR^d\rightarrow {\revv \RR^{d\times d}}$, $b:[0,T]\times\RR^d\rightarrow \RR^d$, $f:[0,T]\times\RR^d\rightarrow \RR^I$, $g:\RR^d\rightarrow \RR^I$, and $h:[0,T]\times\RR^d\rightarrow \RR^I$ are uniformly Lipschitz continuous and uniformly bounded.
\item{\label{A2}} Furthermore, for all  $(\x,\p)\in\RR^d\times\DI$, we have that
\begin{equation*} 
\p^\T f(T,\x)\leq \p^\T g(\x)\leq \p^\T h(T,\x).
\end{equation*}
\item{\label{A3}} We assume that, $\PPas$, {\rev for all $\x\in\mD$,} $\X^{t,\x}\in\mD$, for some bounded domain $\mD$ of $\RR^n$.
\end{enumerate}

{\rev
\begin{rem}\label{rem_bndd}
Assumption $(A_3)$ is required for the application of the approximation result for neural networks, which only holds on bounded domain.
For problems where $(A_3)$ does not hold, one may proceed by considering domains $\mD_R = [-R,R]^d$ for increasing sequence of numbers $R >0$
and imposing an artificial boundary condition if $\X^{t,\x}\notin \mD_R$. Then for each fixed $R$ the results below apply and
one may take $R\rightarrow \infty$ to obtain convergence on $\mathbb{R}^d$.
\end{rem}}


We recall the definition of the viscosity solution \Cref{eq:hjb} below, cf. \cite{gruen2013on}.
\begin{defn}[Viscosity solution]\label{def:viscosity}
A function $w:[0,T]\times \RR^d\times \DI\rightarrow\RR$ is a viscosity solution to \Cref{eq:hjb} if the following two properties hold:
\begin{itemize}
\item[(i)]
$w$ is a \textbf{viscosity subsolution} to \Cref{eq:hjb} if for all $(\ot, \ox, \op)\in [0,T)\times\RR^d\times\Int(\DI)$ and for any smooth test function
$\varphi:[0,T]\times\RR^d\times \Int(\DI)\rightarrow\RR$
such that $(w-\varphi)$ has a strict maximum at $(\ot, \ox, \op)$ with $(w-\varphi)(\ot, \ox, \op)= 0$ it holds that
\begin{equation}\label{eq:hjbsub}
\max\Big\{\max\Big\{\min\Big\{\mL\varphi ,\,\varphi-\p^\T f\Big\},\,\varphi-\p^\T h\Big\},\,-\lambda(\p,\DD_{\p}^2\varphi)\Big\} 
\leq 0,
\end{equation}
at $(\ot, \ox, \op)$.
\item [(ii)] $w$ is a \textbf{viscosity supersolution} to \Cref{eq:hjb} if for all $(\ot, \ox, \op)\in [0,T)\times \RR^d\times\DI$ and for any smooth test function
$\varphi:[0,T]\times\RR^d\times\DI\rightarrow\RR$
such that $(w-\varphi)$ has a strict minimum at $(\ot, \ox, \op)$ with $(w-\varphi)(\ot, \ox, \op)= 0$ it holds that
\begin{equation}\label{eq:hjbsup}
\begin{split}
\max\Big\{\max\Big\{\min\Big\{\mL\varphi ,\,\varphi-\p^\T f\Big\},\,\varphi-\p^\T h\Big\},\,-\lambda(\p,\DD_{\p}^2\varphi)\Big\} 
\geq 0,
\end{split} 
\end{equation}
at $(\ot, \ox, \op)$,
\end{itemize}
{\rev and in addition $w(T,\;\x,\;\p)  =  \p^\T g(\x)$ for any $\x,\p\in \mathbb{R}^d\times\DI$.}
\end{defn}

Under the assumptions \ref{A1} and \ref{A2}, the existence of a unique viscosity solution to \Cref{eq:hjb} in the sense of the above definition is guaranteed by the next theorem, cf. \cite[Theorem\,3.4]{gruen2013on}.
\begin{theorem}\label{thm:uniqueSol}
There exists a  unique viscosity solution to \Cref{eq:hjb} in the class of bounded uniformly continuous functions, which are convex and uniformly Lipschitz in $\p$.
\end{theorem}

\section{Numerical approximation}\label{sec:TimeDiscret}
In \Cref{subsec:discrete1} we introduce a probabilistic numerical numerical scheme which is discrete in time and in the probability variable.
We combine the semi-discrete scheme with a neural network approximation in the spatial variable to obtain a fully-discrete numerical approximation in \Cref{sec:nnscheme}.

\subsection{Semi-discrete numerical approximation in \texorpdfstring{$t$}{t}  and \texorpdfstring{$\p$}{p}}\label{subsec:discrete1} 
For $N\in\mathbb{N}$ we consider an equidistant partition $\Pi^\hht= \{t_n = n\hht$, $n=0,\ldots,N\}$
of the time interval $[0,T]$ with time-step $\hht:=T/N$. 
We define the Euler approximation of the stochastic process $\X^{t_n,x}$ \Cref{eq:state} over the time interval $(t_n, t_{n+\ell})$ as
\begin{equation}\label{eq:euler1}
\oX^{n,\x}_{n+\ell} = \x + \sum_{j=n}^{n+\ell-1} \Big(b(t_j,\oX^{n,\x}_j)\hht  + a(t_j,\oX^{n,\x}_j) {\xi^j} \sqrt{\hht}\Big),
\end{equation}
where {$\{\xi^{n}\}_{n\in \mathbb{N}}$} are  $\RR^d$-valued {i.i.d.} random variables with zero mean and unit variance.

For the discretization with respect to the probability variable $\p$ we adopt the approach of \cite{banas2020numerical}.
To this end we consider a family of regular partitions $\{\mM_{\hhp}\}_{\hhp>0}$ of $\Delta(I)$ into open simplices $K$
with mesh-size $\hhp := \max_{K\in\mM_\hhp}\{\diam(K)\}$
such that ${\Delta(I)} = \cup_{K\in\mM_\hhp}\overline{K}$. 
The set of vertices of all $K\in \mM_\hhp$ is denoted by $\mN_{\hhp}:=\{\p_1,\ldots,\p_M\}$. 

Given the step size $\hh := (\hht,\hhp)$ we introduce the following semi-discrete numerical scheme: 
\begin{itemize}
\item For $m=1,\dots,M$ define the map $\x\mapsto \ouh (t_N,\x,\p_m)$ as
\begin{equation}\label{eq:fdis0}
\ouh (t_N,\x,\p_m)= \p_m^\T g(\x).
\end{equation}

\item For $n = N-1,\ldots,0$ proceed as follows:

\begin{enumerate}[label = $\circ$]
\item For $m=1,\dots,M$ define the map $\x\mapsto\oY^{n,\x,\p_m}$ as
\begin{equation}
\label{eq:fdisc1}
\oY^{n,\x,\p_m} := \min\Big\{\max\Big\{\EE\big[\ouh(t_{n+1},\oXnx_{n+1},\p_m)\big],\,\p_m^\T f(t_n,\x)\Big\},\,\p_m^\T h(t_n,\x)\Big\}.
\end{equation}

\item For $m=1,\dots,M$, determine the map $ \p_m \mapsto \ouh (t_n,\x,\p_m)$ as
\begin{equation}
\label{eq:fdisc11}
\ouh(t_n,\x,\p_m) = \vexp\,\big [\oY^{n,\x,\p_1},\ldots,\oY^{n,\x,\p_M}\big ](\p_m).
\end{equation}
\end{enumerate}
\end{itemize}
We note that for $\p\in\DI$, the lower convex envelope in \Cref{eq:fdisc11} is the solution of the
minimization problem (cf.\,\cite{carnicier1992convexity})
\begin{equation}\label{eq:fdisc2} 
\vexp\, [\oY^{n,\x,\p_1},\ldots,\oY^{n,\x,\p_M}](\p)= \min \Big\{\sum_{i = 1}^{M}\lambda_i \oY^{n,\x,\p_i}:\sum_{ i= 1}^M\lambda_i = 1,\,\lambda_i\ge 0, \p =\sum_{ i= 1}^{M}\lambda_i\p_i\Big\}.
\end{equation}

Following \cite{banas2020numerical} we consider the (convexity preserving) data-dependent simplicial partition $\mMnx_{\hhp}$ of $\DI$ with nodes $\mNnx_{\hhp}:=\{\pnx_1,\ldots,\pnx_{\Mnx}\}\subseteq\mN_\hhp$ such that the piecewise linear interpolant of the data values at the nodes $\mNnx_\hhp$ over the partition $\mMnx_{\hhp}$ (for precise definition see \Cref{eq:fdisc5} below) agrees with the discrete data values $\{\p_m, \ouh(t_n,\x,\p_m)\}_{m=1}^M$.
We note that the partition $\mMnx_{\hhp}$ does not necessarily coincides
with the original mesh $\mM_\hhp$, {\rev in particular the mesh size of $\mMnx_{\hhp}$ does not necessarily decrease for $\Delta p \rightarrow 0$ (the mesh size of $\mM_\hhp$)}.

We consider the set of piecewise linear Lagrange basis functions $\{\psinx_i, i=1,\ldots,\Mnx\}$ associated with the set of nodes $\mNnx_{\hhp}$ of the partition $\mMnx_{\hhp}$.
We recall the following properties of the Lagrange basis functions which will be frequently used
throughout the paper: (a) $\psinx_i(\pnx_j) = \delta_{ij}$, where $\delta_{ij}$ is the Kronecker delta, and (b) $\sum_{i=1}^{\Mnx}\psinx_i(\p) = 1$ for $\p \in \Delta(I)$.
We note that (a) implies that at any point $\p\in\DI$ there are at most $I$ basis functions with nonzero value; hence the sum in (b) reduces to $\sum_{i=1}^I$.

For $n=0,\dots,N$, $\x\in \RR^d$, we define the piecewise linear interpolant $\p\mapsto\ouh(t_n,\x,\p)$ of the discrete lower convex envelope   $\{\ouh(t_n,\x,\p_1),\ldots,\ouh(t_n,\x,\p_M)\}$ 
over the convexity preserving partition $\mMnx_{\hhp}$ as
\begin{equation}\label{sumu}
\ouh(t_n,\x,\p) :=\sum_{i=1}^{\Mnx} \ouh(t_n,\x,\p_{m(\pnx_i)})\psi_{m(\pnx_i)}^{n,x}(\p),
\end{equation}
where $m(\pnx_i)\in\NN$ is the index of $\pnx_i$ in $\mN_\hhp$; i.e., $\pnx_i := \p_{m(\pnx_i)}$ for some $\p_{m(\pnx_i)}\in \mN_\hhp$, cf. \cite{banas2020numerical}.

Note that for each $\p\in\DI$ there exist at most $I$ non-zero basis functions (i.e., the basis functions associated with the $I$ vertices $\{\pinx_i\}_{i=1}^I$ of the simplex 
$K=\mathrm{conv}\{\pinx_i, \,\, i =1,\dots,I\} \in \mMnx_{\hhp}$ for which $p\in \overline{K}$) with nonzero value 
which we denote by $\{\lambdanx_1,\ldots,\lambdanx_I\}$ such that \eqref{sumu} is equivalent to
\begin{align}
\label{eq:fdisc5}
\ouh(t_n,\x,\p) =  \sum_{i=1}^{I}\ouh(t_n,\x,\pinx_i(\p))\lambdanx_i(\p),
\end{align}
moreover 
\begin{align}\label{eq:basis}
\sum_{i=1}^{I}\lambdanx_i(\p) = 1 \text{ and } \p =  \sum_{i=1}^{I} \pinx_i(\p)\lambdanx_i(\p).
\end{align}

Next, we define for every $t\in[t_n, t_{n+1}]$
\begin{align}\label{eq:fdisc3}
\ouh(t,\x,\p) =  \ouh(t_n,\x,\p)+ \frac{\ouh(t_{n+1},\x,\p) -   \ouh(t_n,\x,\p)}{\hht}(t-t_n).
\end{align}

%

\begin{theorem}\label{thm:mainRes}
Assume \ref{A1} and \ref{A2}. Then the sequence $\{\ouh\}_{\hh}$ converges uniformly on every compact subsets of $[0,T]\times\RR^d\times \DI$, i.e.,
\begin{equation*}
\lim_{\substack{(s,\y,\q)\to (t,\x,\p)\\ \hh\rightarrow 0}} \ouh(s,\y,\q) = u(t,\x,\p),
\end{equation*}
where $u$ is the unique viscosity solution to \Cref{eq:hjb} in the class of bounded uniformly continuous
functions, which are uniformly Lipschitz and convex in $\p$.
\end{theorem}
\begin{proof}

We verify that the sequence $\{\ouh\}_{\hh}$ satisfies the assumptions of the Arzel\`a--Ascoli theorem, cf. \cite[Section~III.3]{yosida2013functional}.
The equiboundedness follows by \Cref{lem:lipx} and the equicontinuity by \Cref{lem:lipx,lem:lipp,lem:lipt}. 
Hence, up to a subsequence, on every compact subsets of $[0,T]\times\RR^d\times \DI$, the sequence $\{\ouh\}_{\hh}$ converges uniformly to a limit $w$ that is bounded, uniformly continuous and is convex and uniformly Lipschitz continuous in $\p$.

By \Cref{prop:subvis} (resp.\,\Cref{prop:supvis}), $w$ satisfies the viscosity sub- and super-solution property, hence it is a viscosity solution of \Cref{eq:hjb} in the sense of \Cref{def:viscosity}. 

By \Cref{thm:uniqueSol}, the viscosity solution of \Cref{eq:hjb} is unique in the class of bounded, uniformly continuous functions, which are convex and uniformly Lipschitz in $\p$.
Since every limit of $\{\ouh\}_{\hh}$ is bounded, uniformly continuous and convex and uniformly Lipschitz in $\p$, the limit is unique.

\end{proof}															


\subsection{Fully-discrete approximation and convergence}\label{sec:nnscheme}

{Below we describe a modification of the numerical scheme \Cref{eq:fdis0}-\eqref{eq:fdisc11} which employs a Feedforward Neural Network approximation in the state variable $\x$.
The resulting fully-discrete algorithm \Cref{eq:dnn0}-\Cref{eq:fnn3} is practically implementable.}

To describe the neural network approximation scheme we loosely follow the exposition of \cite{bishop2006pattern,hure2020deep}.
Typically a continuous function $\x\mapsto \Phi(\x)$ can be expanded in terms of a linear combination of fixed nonlinear basis functions $\varphi_j(\x)$ and take the form
\begin{equation}\label{eq:typic_expansion}
	\Phi(\x;\theta ) = \sum_jw_j\varphi_j(\x),
\end{equation} 
where $\x=(x_1,\ldots,x_d)^\T\in\RR^d$ are \textit{called input variables}, $\theta = (w_j)_j$ is a set of numbers called \textit{parameters}, and $\Phi$ the \textit{target variables} (the function we want to predict or approximate).
The idea behind neural networks is to extend \Cref{eq:typic_expansion} with basis functions that depend on a linear combination of the inputs, where the coefficients in the linear combination are adaptive parameters.

Depending on the nature of the input and the assumed distribution of the target variables it is possible to generate a number of neural networks expansions of different characteristics.
{We first introduce} the so-called \textit{Feedforward Neural Network} for standard regression problems.
It can be described as a series of layers of functional transformations.
Each subsequent layer has a connection from the previous layer.
In the first layer or \textit{input layer}, we take the image of the input variable $\x$ through an affine transformation that takes the form 
\begin{equation}\label{eq:layer1}
	 w^{(0)}\x + b^{(0)}.
\end{equation}
We shall refer to the parameter $w^{(0)}\in\RR^{\kappa\times d}$ as the \textit{weight matrix} and the parameter  $b^{(1)}\in\RR^\kappa$ as the \textit{bias vector}.
The superscript $(0)$ indicates that the corresponding parameters are in the first layer of the network.
Next, a differentiable, nonlinear\textit{ activation function} $\rho:\RR\mapsto\RR$ acts component-wise on the \textit{activation vector} \Cref{eq:layer1} and transforms it into  a vector $\rho(w^{(0)}\x + b^{(0)})$ of \textit{hidden units} that constitutes the second layer of the neural network also know as \textit{hidden layer}.
Finally, we take the image of the vector of hidden unit vector through another affine transformation to get a scalar-valued \textit{output variable}
\begin{equation}\label{eq:layer2}
	\Phi_\kappa(\x;\theta) = w^{(1)}\cdot \rho(w^{(0)}\x + b^{(0)})  + b^{(1)}
\end{equation}
with $w^{(1)}\in\RR^\kappa$ and $b^{(1)} \in \RR$.
The set of all weight and bias parameters have been grouped together into $\theta\coloneqq\{w^{(0)}, b^{(0)},  w^{(1)}, b^{(1)}\}$.
The function \cref{eq:layer2} is a Feedforward Network with 1 hidden layer, which contains $\kappa$ neurons. 
{\rev
We can generalize the above construction and consider a Feeforward Network with several layers $l=0,\ldots,\lambda-1$, where each hidden layer $l=1,\ldots,\lambda-2$  has $\kappa_i$ neurons, that is described by the following sequence of transformations
	\begin{alignat*}{6}
		\varphi_0(\x;\theta_0) &=\rho_0(w^{(0)}\x + b^{(0)}), &\theta_0&\coloneqq\{w^{(0)}, b^{(0)}\} &\mbox{ with }w^{(0)}&\in\RR^{\kappa_0\times d}, \; b^{(0)}\in\RR^{\kappa_0}
		\\
		\varphi_1(\x;\theta_1) &=\rho_1\big( w^{(1)} \varphi_0(\x)  + b^{(1)}\big), &\theta_1&\coloneqq\{w^{(l)}, b^{(l)}\}_{l=0}^1 &\mbox{ with } w^{(1)}&\in\RR^{\kappa_1\times \kappa_0},\; b^{(1)}\in\RR^{\kappa_1}
		\\
		&\vdots& &\vdots  &&\vdots
		\\
		\varphi_{\lambda-1}(\x;\theta_{\lambda-1}) &=\rho_{\lambda-1}\big( w^{(\lambda-1)} \varphi_{\lambda-2}(\x)  + b^{(\lambda-1)}\big),  &\; \theta_{\lambda-1}&\coloneqq\{w^{(l)}, b^{(l)}\}_{l=0}^{\lambda-1} &\mbox{ with }w^{(\lambda-1)}&\in\RR^{1\times \kappa_{\lambda-2}},\; b^{(\lambda-1)}\in\RR^1.
	\end{alignat*}
Like in \cref{eq:layer2}, we set $\rho_{\lambda-1} = \mathrm{Id}$.
Finally, we can define the space of neural networks with input dimension $d$, output dimension 1, $\lambda$ hidden layers, with a sequence of neurons $\kappa\coloneqq\{\kappa_l\}_{l=0}^{\lambda-1}$, with $\kappa_{\lambda-1} = 1$; and a sequence of activation functions $\rho\coloneqq\{\rho_l\}_{l=0}^{\lambda-1}$, with $\rho_{\lambda-1} = \mathrm{Id}$. It is then represented by the parametric set of functions
\begin{equation}\label{eq:fnn}
	\begin{split}
		&\mNN
		\coloneqq\Big\{\x\in\RR^d\mapsto \Phi_\kappa(\x;\theta) = \varphi_{\lambda-1}(\x;\theta_{\lambda-1}),\,\theta=\theta_{\lambda-1} \Big\}.
	\end{split}
\end{equation}
In general, the Neural Network functions in $\mNN$ have a Lipschitz constant $\gamma_\kappa$ that goes to infinity when $\kappa\to\infty$, see \cite{hure2020deep}. This particularity of the Neural Network  will imply a coupling between $\kappa$ and the parameters $\hhx$, and $\hht$ of the numerical scheme proposed in the current paper. As a remedy we introduce a GroupSort activations, which were used in \cite{anil2019sorting} to preserves the gradient norm of the input variable. In \cite{maximilien2022approximation}, the introduction of GroupSort activations enabled the authors to establish a theoretical analysis of the convergence of their numerical scheme.

Let the integer $s\geq 2$, be a grouping size, dividing the number of neurons $\kappa$, i.e., $\kappa = s \tilde{n}$ for some $\tilde{n}\in \mathbb{N}$.
The GroupSort network functions correspond to classical Feedforward Neural Network functions in $\mNNs$ with a specific activation function $\rho^s\coloneqq\{\rho_l^s\}_{l=0}^{\lambda-1}$.
The nonlinear function $\rho^s$ divides its input into groups of size $s$ and sorts each group in
decreasing order.  By enforcing the parameters $\theta_{\lambda-1}\coloneqq\{w^{(l)}, b^{(l)}\}_{l=0}^{\lambda-1}$ of the networks to satisfy for $i=1,\ldots,\lambda-1$ and $j=0,\ldots,\lambda-1$
\begin{equation*}
	\vert w^{(0)}\vert_{2,\infty}\coloneqq\sup_{\vert\x\vert_2=1}\vert w^{(0)}\x\vert_\infty\leq 1,\quad \vert w^{(i)}\vert_\infty\coloneqq\sup_{\vert\x\vert_\infty=1}\vert w^{(i)}\x\vert_\infty\leq 1,\quad \vert b^{(j)}\vert_\infty\leq \zeta,
\end{equation*}
{\revv for a (fixed) constant $\zeta > 0$}. The related GroupSort neural network functions from $\mNNs$ are Lipschitz continuous with a Lipschitz constant $1$.
We denote the space of GroupSort neural network functions from $\mNNs$ that satisfy the aforementioned constraints by
\begin{equation*}\label{eq:groupsort0}
	\mSS \coloneqq\Big\{\Phi_\kappa(\;\cdot\;;\theta)\in \mNNs: \theta \in \Theta_{\kappa}\Big\},
\end{equation*}
with 
\begin{equation}\label{eq:gs_constraints}
	\begin{split}
		\Theta_{\kappa}\coloneqq\Big\{ &\{w^{(l)}, b^{(l)}\}_{l=0}^{\lambda-1} \mbox{ with }w^{(l)}\in\RR^{\kappa_{l-1}\times \kappa_{l-2}} \mbox{ and } b^{(\lambda-1)}\in\RR^{\kappa_{l-1}},\, l=0,\ldots,\lambda -1:
		\\
		&\vert w^{(0)}\vert_{2,\infty}\leq 1,\; \vert w^{(i)}\vert_\infty\leq 1,\, i=1,\ldots,\lambda-1,\; \vert b^{(j)}\vert_\infty\leq \zeta,\, j=0,\ldots,\lambda-1\Big\}.
	\end{split}
\end{equation}
We consider the following set of neural network functions
\begin{equation}\label{eq:groupsort}
	\begin{split}
		\mGG \coloneqq\Big\{\x\in\RR^d\mapsto\Psi_\kappa(\x;\theta)= \gamma\beta\Phi_\kappa\Big(\frac{\x+\alpha}{\beta};\theta\Big),\Phi_\kappa(\,\cdot\,;\theta)\in\mSS, \mbox{ for some } \alpha\in\RR^d,\, \beta \in \RR^+\Big\} .
	\end{split}
\end{equation}
Neural network functions in $\mGG$ are Lipschitz continuous functions with a Lipschitz constant $\gamma$.
}

In addition to the family of partition $\{\mM_{\hhp}\}_{\hhp>0}$ introduced in \Cref{sec:TimeDiscret}
we consider a family of regular partition $\{\mM_{\hhx}\}_{\hhx>0}$ of the spatial domain {\rev $\mathcal{D}\coloneqq [-R,R]^d\subset\mathbb{R}^d$} (see Assumption~\ref{A3}) into open simplices $\mathcal{T}$
with the mesh size $\hhx := \max_{\mathcal{T}\in \mM_\hhx}\mathrm{diam}(\mathcal{T})$ and denote the set of vertices of all $\mathcal{T}\in \mM_\hhx$
as $\mN_{\hhx} :=\{\x_1,\ldots,\x_{L}\}$.


The fully-discrete numerical solution $\ouhkRL(t_n,\x,\p)\approx\ouh(t_n,\x,\p)$
 is constructed by employing the above {\rev GroupSort Neural Network} approximation in the spatial variable
in the semi-discrete scheme \cref{eq:fdis0}-\cref{eq:fdisc11}.
{\rev To show convergence of the numerical approximation we benefit from the approximation result \cite[Proposition~2.1]{maximilien2022approximation}.
I.e. given a Lipschitz continuous function $x \rightarrow \psi(\x)$  with a Lipschitz constant $\gamma$, we have that
to any $\varepsilon>0$ there exists of a GroupSort Neural Network $\Psi_\kk\in\mGG$ {\revv (note that $\zeta$, $\alpha$, $\beta$ in \cref{eq:gs_constraints}, \cref{eq:groupsort} are independent of $\psi$)} such that
\begin{equation*}
	\vert \psi(\x) - \Psi_\kk(\x;\theta)\vert\leq 2R\gamma\varepsilon \qquad \forall \x\in \mD,
\end{equation*}
with $\Psi_\kk$ of grouping size $\lceil  \frac 1\varepsilon\rceil $, depth $\lambda + 1 = 3$, and width $\sum_{l=0}^{\lambda  -1}\kappa_l = \mO(\frac 1\varepsilon)$.
The above statement is equivalent to saying that for any $\varepsilon$ there exists a (sufficiently large) $\kappa\equiv \kappa(\varepsilon)$ and $\Psi_{\kk(\varepsilon)}\in\mGG$ such that
\begin{equation}\label{eq:univ_approx_groupsort}
	\vert \psi(\x) - \Psi_{\kk(\varepsilon)}(\x;\theta)\vert\leq 2R\gamma\varepsilon \qquad \forall \x\in \mD.
\end{equation}
} 

The resulting algorithm reads as:
\begin{itemize}
	\item For $m=1,\dots,M$ and $\ell=1,\dots,L$ define
	\begin{equation}\label{eq:dnn0} 
		\ouhkRL(t_N,\x_\ell,\p_m)= \p_m^\T g(\x_\ell).
	\end{equation}
	
	\item For $n = N-1,\ldots,0$ proceed as follows:
	
	\begin{enumerate}[label = $\circ$]
		\item For  $m=1,\dots,M$ compute:
		\begin{align}
			\label{eq:fnn1}
			\theta_{n+1}(\p_m) &= {\argmin_{\theta\in\Theta_\kk} \frac 1L \sum_{\ell = 1}^L   \vert {\rev \Psi^{n+1}_\kk}(\x_\ell;  { \theta}) - \ouhkRL(t_{n+1},\x_\ell,\p_m) \vert^2 },
		\end{align}
		and define the map $\x\mapsto\oY^{n,\x,\p_m}_\kappa$ by 
		\begin{equation}
			\label{eq:fnn2}
			\oY^{n,\x, \p_m}_{\kk} = \min\Big\{\max\Big\{ \EE\big[{\rev \Psi^{n+1}_\kk}(\oX^{\x,n}_{n+1};\theta_{n+1}(\p_m))\big],\,\p_m^\T f(t_n,\x)\Big\},\,\p_m^\T h(t_n,\x)\Big\}.
		\end{equation}
		
		\item For $m=1,\dots,M$ and $\ell=1,\dots,L$, determine the map ${ \p_m} \mapsto\ouhkRL(t_n,\x_\ell,\p_m)$ as
		\begin{equation}
			\label{eq:fnn3}
			\ouhkRL(t_n,\x_\ell,\p_m) = \vexp\,\big[\oY^{n,\x_\ell, \p_1}_{\kk},\ldots,\oY^{n,\x_\ell, \p_M}_{\kk}\big ](\p_m).
		\end{equation}
	\end{enumerate}
\end{itemize}
For $n=0, \dots, N$, $\ell=1,\ldots,L$,
we denote the (data dependent) convexity preserving simplicial partition of $\DI$ as $\mtMnxk_\hhp$ and
$\mtNnxk_\hhp=\{\tpnxk_1,\ldots,\tpnxk_{\mtMnxk}\}\subseteq\mN_\hhp$ the set of its nodes such that the piecewise linear interpolant of the data values at the nodes $\mtNnxk_\hhp$ over the partition $\mtMnxk_\hhp $ agrees with the values $\{\p_m,\ouhkRL(t_n,\x,\p_m)\}_{m=1}^M$.
Furthermore, we note the respective counterparts of \eqref{eq:fdisc5}, \eqref{eq:basis} as
\begin{align}
	\label{eq:fdisc5en}
	\ouhkRL(t_n,\x,\p) =  \sum_{i=1}^{I}\ouhkRL(t_n,\x_\ell,\tpinxk_i(\p))\tlambdanxk_i(\p),
\end{align}
and
\begin{align}
	\label{eq:basisen}
	\sum_{i=1}^{I}\tlambdanxk_i(\p) = 1,\quad \p =  \sum_{i=1}^{I} \tpinxk_i(\p)\tlambdanxk_i(\p).
\end{align}

For $n = 0,\ldots,N-1$  we define
\begin{equation}\label{def_eps}
	\epskrL^{n}\coloneqq \max_{m = 1,\ldots,M} {\revv \max_{\x_\ell\in \mathcal{N}_{\hhx}}}\,\,\big\vert \Psi^{n+1}_\kk(\x_{{\revv \ell}};\theta_{n+1}(\p_m))   - \ouhkRL(t_{n+1},\x_{{\revv \ell}},\p_m)\big\vert.
\end{equation}
Throughout the paper, by slight abuse of notation, we define $\x\rightarrow \ouhkRL(t_{n+1},\x,\p_m)$ to be the piecewise linear interpolant over the mesh $\mM_\hhx$ (i.e., in the spatial variable $\x$)
of the discrete solution in \Cref{eq:fnn3}.

\begin{rem}\label{rem:process_D}
i) Assumption \ref{A3} together with the additional assumption that the random variables $\{\xi^n\}_{n\in \mathbb{N}}$ satisfy $|\xi^n|< R < \infty$ $\PPas$ (which is  satisfied
if one considers a random walk approximation where for instance $\xi^n = \pm 1$ with probability $1/2$)
ensure that the Euler approximation
\Cref{eq:euler1} satisfies $\oXnx_{n+1}\in \mathcal{D}^*$ $\PPas$ for some bounded domain $\mathcal{D}^*$ with $\mathcal{D} \subset \mathcal{D}^*$.
In the proofs below we assume without loss of generality that $\mathcal{D}^* = \mathcal{D}$.

{\revv
ii)  The neural network in \cref{eq:fnn1} satisfies $\Psi^{n+1}_\kk\in \mGGn$. 
Since $\ouhkRL$ is uniformly Lipschitz continuous with Lipschitz constant $\gamma=\mLx$ (see \Cref{lem:lipx_nn} below),  we set $\gamma^{n+1} = \gamma$ for $n=N-1, \dots, 0$.

iii) The parameters constants $\alpha$, $\beta$ in \cref{eq:groupsort} depend on the spatial domain $\mathcal{D}$, cf. the proof of \cite[Proposition 2.1]{maximilien2022approximation}.}
\end{rem}

\begin{theorem}\label{thm:error_RL}
	Let \ref{A1}-\ref{A3} hold and let $\ouhkRL$, $\ouh$  be the numerical solutions constructed by \cref{eq:dnn0}-\cref{eq:fnn3} and \cref{eq:fdis0}-\cref{eq:fdisc11}, respectively.
  For all $n=0,\ldots,N$ and all $(\x_\ell,\p_m)\in \mathcal{N}_{\hhx}\times\mN_{\hhp}$ it holds that
	\begin{align}\label{eq:error_RL}
		\lvert (\ouhkRL  - \ouh)(t_n,\x_\ell,\p_m)\rvert \leq   {\rev 2N\mLx \hhx} + \sum_{n'= n}^{N-1}\epskrL^{n'},
	\end{align}
{\revv where the constant $\mLx$ is defined in \Cref{lem:lipx} below.}
\end{theorem}
\begin{proof}
	We fix $(n,{\ell},m) \in \{0,\ldots,N\}\times\{1,\ldots,L\}\times \{1,\ldots,M\}$ and assume that, $\PPas$, $\oXnxl_{n+1}\in\mD$.	
	First we suppose that $0\leq  \ouhkRL(t_n,\x_\ell,\p_m)    - \ouh(t_n,\x_\ell,\p_m) $. 
	By \Cref{eq:fdisc5}, \Cref{eq:basis}, \Cref{eq:fdisc5en}, \Cref{eq:basisen}  and the convexity of the map $\p_m\mapsto  \ouhkRL(t_n,\x_\ell,\p_m)$ we have
	\begin{align}
		\notag
		\ouhkRL(t_n,\x_\ell,\p_m)    &- \ouh(t_n,\x_\ell,\p_m) 
		= \ouhkRL(t_n,\x,\p_m) - \sum_{i=1}^{I}\ouh(t_n,\x_\ell,\pinxl_i(\p))\lambdanxl_i(\p_m)
		\\
		\label{eq:ecor0}
		&\leq \sum_{i=1}^{I}\big(\ouhkRL(t_n,\x_\ell,\pinxl_i(\p_m)) - \ouh(t_n,\x_\ell,\pinxl_i(\p_m))\big)\lambdanxl_i(\p_m).
	\end{align}
	Since $F\mapsto\vexp [F]$ is nonexpansive, it follows for every $\pi\in\{\pinxl_1(\p_m),\ldots,\pinxl_I(\p_m)\}$  that 
	{\rev \begin{align}
		\label{eq:ecor1}
		 \ouhkRL(t_n,\x_\ell,\pi) - \ouh(t_n,\x_\ell,\pi) 
		 =  \ouhkRL(t_n,\x_\ell,\pi) - \oY^{n,\x_\ell,\pi} 
		\leq \oY^{n,\x_\ell,\pi}_{\kappa,\mD,L}- \oY^{n,\x_\ell,\pi}.
	\end{align}}
	Recalling \Cref{eq:fdisc1}, \Cref{eq:fnn2}, 
	{
	we express for $\pi\in\{\pinxl_1(\p_m),\ldots,\pinxl_I(\p_m)\}$
	\begin{align*}
		&\lvert\oY^{n,\x_\ell,\pi}_{\kappa,\mD,L}- \oY^{n,\x_\ell,\pi}\rvert =\Big\vert 
		\min\Big\{\max\Big\{ \EE\big[\Psi^{n+1}_\kk(\oX^{n,\x_\ell}_{n+1};\theta_{n+1}(\pi))\big],\,\pi^\T f(t_n,\x_\ell)\Big\},\,\pi^\T h(t_n,\x_\ell)\Big\}
		\\
		&-\min\Big\{\max\Big\{\EE\big[\ouh(t_{n+1},\oX^{n,\x_\ell}_{n+1},\pi)\big],\,\pi^\T f(t_n,\x_\ell)\Big\},\,\pi^\T h(t_n,\x_\ell)\Big\}\Big\vert.
	\end{align*}
	Using \Cref{eq:lipx8} we get
	\begin{align*}
		&\lvert\oY^{n,\x_\ell,\pi}_{\kappa,\mD,L}- \oY^{n,\x_\ell,\pi}\rvert 
		\\
		&\leq \Big\vert 
		\max\Big\{ \EE\big[\Psi^{n+1}_\kk(\oX^{n,\x_\ell}_{n+1};\theta_{n+1}(\pi))\big],\,\pi^\T f(t_n,\x_\ell)\Big\}-\max\Big\{\EE\big[\ouh(t_{n+1},\oX^{n,\x_\ell}_{n+1},\pi)\big],\,\pi^\T f(t_n,\x_\ell)\Big\}\Big\vert,
	\end{align*}
	and then, by \Cref{eq:lipx7},
	\begin{align*}
		&\lvert\oY^{n,\x_\ell,\pi}_{\kappa,\mD,L}- \oY^{n,\x_\ell,\pi}\rvert \leq 
\big\vert 
		\EE\big[\Psi^{n+1}_\kk(\oX^{n,\x_\ell}_{n+1};\theta_{n+1}(\pi))-\ouh(t_{n+1},\oX^{n,\x_\ell}_{n+1},\pi)\big]\big\vert.
	\end{align*}}
 By the above estimates we deduce from  \Cref{eq:ecor1} that
	\begin{align}
		\label{eq:ecor2}
		 \begin{split}
		 	\ouhkRL(t_n,\x_\ell,\pi) &- \ouh(t_n,\x_\ell,\pi)\leq
		 	\big\vert 
		 	\EE\big[\Psi^{n+1}_\kk(\oX^{n,\x_\ell}_{n+1};\theta_{n+1}(\pi))-\ouh(t_{n+1},\oX^{n,\x_\ell}_{n+1},\pi)\big]\big\vert.
		 \end{split}	
	\end{align}
	
		Below we consider $\mN_\hhx$-valued random variables $\hx_{k}$, $k=n+1,\dots, N$ 
		which are defined as follows.
		Due to the assumption \ref{A3}, see also \Cref{rem:process_D} for $1< k \leq N$, there exists a simplex $\mS \equiv \mS(\omega) \in \mM_\hhx$ such that $\oX^{k-1,\x_\ell}_k(\omega) \in \mS(\omega)$ for $\omega\in\Omega$.
		For $\omega\in \Omega$ we then set
		{\rev
		\begin{equation}\label{maxx}
		\hx_{k}(\omega) = \argmax_{\{v_i:i=0,\dots,d\}} \vert\ouhkRL(t_{k},v_i,\pi)  - \ouh(t_{k},v_i,\pi)\vert,
		\end{equation}
		}
    i.e., $\hx_{k}(\omega)\in {\revv \mathcal{N}_{\Delta x}}$ is the vertex of $\mS(\omega) = \mathrm{conv}\{v_0,\dots,v_d\}$ which realizes the greatest difference between the fully-discrete and semi-discrete numerical solution.
    Below we suppress the explicit dependence of $\hx_{k}$, $\mS$ on $\omega \in \Omega$.

		For any {\revv vertex $v_i\in \mathcal{N}_{\Delta x}$ of $\mS$}
		we express
		\begin{align*}
			&\big[\Psi_\kk(\oX^{n,\x_\ell}_{n+1};\theta_{n+1}(\pi))   - \ouh(t_{n+1},\oX^{n,\x_\ell}_{n+1},\pi)\big] 
			\\
			&= \big[\Psi_\kk(\oX^{n,\x_\ell}_{n+1};\theta_{n+1}(\pi))   - \Psi_\kk({\revv v_i};\theta_{n+1}(\pi))\big] + \big[\Psi_\kk({\revv v_i};\theta_{n+1}(\pi))   - \ouhkRL(t_{n+1},v_i,\pi)\big]
			\\
			&+ \big[ \ouhkRL(t_{n+1},v_i,\pi) - \ouh(t_{n+1},v_i,\pi)\big] + \big[ \ouh(t_{n+1},v_i,\pi) - \ouh(t_{n+1},\oX^{n,\x_\ell}_{n+1},\pi)\big].
		\end{align*}
	{\revv
		We take the expectation in the above expression, note \Cref{maxx} and estimate
		\begin{align*}
			&\big\vert\EE\big[\Psi_\kk(\oX^{n,\x_\ell}_{n+1};\theta_{n+1}(\pi))   - \ouh(t_{n+1},\oX^{n,\x_\ell}_{n+1},\pi)\big] \big\vert
			\\
			&\le \big\vert \EE\big[\Psi_\kk(\oX^{n,\x_\ell}_{n+1};\theta_{n+1}(\pi))   - \Psi_\kk({\revv v_i};\theta_{n+1}(\pi))\big]\big\vert + \big\vert\EE\big[\Psi_\kk({\revv v_i};\theta_{n+1}(\pi))   - \ouhkRL(t_{n+1},v_i,\pi)\big] \big\vert
			\\
			&+ \big\vert\EE\big[ \ouhkRL(t_{n+1},v_i,\pi) - \ouh(t_{n+1},v_i,\pi)\big] \big\vert
			+ \big\vert\EE\big[ \ouh(t_{n+1},v_i,\pi) - \ouh(t_{n+1},\oX^{n,\x_\ell}_{n+1},\pi)\big]\big\vert.
		\end{align*}
	}
	 {\revv Note that $\Psi^{n+1}_\kk\in \mGGn$, cf. \eqref{eq:groupsort}. Since,  we take $\gamma^{n+1} = \gamma = \mLx$, see  \Cref{rem:process_D}~ii)},
   the map $\x\mapsto \Psi^{n+1}_\kk(\x;\theta_{n+1}(\pi))$ is uniformly Lipschitz continuous with a Lipschitz constant {$\gamma=\mLx$}.
	 The map $\x\mapsto \ouh(t_{n+1},\x,\pi)$ is uniformly Lipschitz continuous with a Lipschitz constant $\mLx$ by \Cref{lem:lipx}.
	 Hence, we deduce that
	 \begin{align*}
	 	\notag
	 	&\big\vert\EE\big[\Psi^{n+1}_\kk(\oX^{n,\x_\ell}_{n+1};\theta_{n+1}(\pi))   - \ouh(t_{n+1},\oX^{n,\x_\ell}_{n+1},\pi)\big]\big\vert
	 	\\
	 	&\leq  {\rev \mLx}\EE\big[\vert\oX^{n,\x_\ell}_{n+1}   - {\revv v_i}\vert\big] + {\revv{\EE}\big[\big\vert\Psi^{n+1}_\kk( v_i;\theta_{n+1}(\pi))   - \ouhkRL(t_{n+1},v_i,\pi)\big\vert]}
	 	\\
	 	&+ \EE\big[\big\vert \ouhkRL(t_{n+1},\hx_{n+1},\pi) - \ouh(t_{n+1},\hx_{n+1},\pi)\big\vert\big] + {\revv \mLx\EE\big[\vert v_i - \oX^{n,\x_\ell}_{n+1}\vert\big]}.
	 	\end{align*}
	 	{\revv Since $\oX^{n,\x_\ell}_{n+1},v_i\in\mS$ it holds $\PPas$ that $\vert v_i - \oX^{n,\x_\ell}_{n+1}\vert\leq\hhx$.}
	 	Consequently, we estimate
	 	\begin{align*}
	 		\notag
	 		&\big\vert\EE\big[\Psi^{n+1}_\kk(\oX^{n,\x_\ell}_{n+1};\theta_{n+1}(\pi))   - \ouh(t_{n+1},\oX^{n,\x_\ell}_{n+1},\pi)\big]\big\vert\leq {\rev \mLx}\hhx 
        + {\revv \max_{\x_\ell\in\mathcal{N}_{\Delta x}}\big\vert\Psi^{n+1}_\kk(\x_\ell;\theta_{n+1}(\pi))   - \ouhkRL(t_{n+1},\x_\ell,\pi)\big\vert}
	 		\\
	 		&+ \EE\big[\big\vert \ouhkRL(t_{n+1},\hx_{n+1},\pi) - \ouh(t_{n+1},\hx_{n+1},\pi)\big\vert\big] +  {\rev \mLx}\hhx .
	 	\end{align*}
	 	Using the above inequality  in \Cref{eq:ecor2}, we obtain for $\pi\in\{\pinxl_1(\p_m),\ldots,\pinxl_I(\p_m)\}$ that
	 	\begin{align*}
	 		&\ouhkRL(t_n,\x_\ell,\pi) - \ouh(t_n,\x_\ell,\pi)
	 		\leq {\rev (\mLx + \mLx)\hhx} + \epskrL^{n} + \EE\big[\big\vert \ouhkRL(t_{n+1},\hx_{n+1},\pi) - \ouh(t_{n+1},\hx_{n+1},\pi)\big\vert\big].
	 	\end{align*}

	For {\revv $0\geq\ouhkRL(t_n,\x_\ell,\p) - \ouh(t_n,\x_\ell,\p)$} we get an analogous estimate using \Cref{eq:basisen,eq:fdisc5en}  and the convexity of the map $\p_m\mapsto  \ouh(\ot_n,\x_\ell,\p_m)$.
	Hence, we obtain that
	\begin{align*}
		&\lvert \ouhkRL(t_n,\x_\ell,\p_m)  - \ouh(t_n,\x_\ell,\p_m)\rvert
		\\
		&\leq {\rev 2\mLx\hhx} + \epskrL^{n} + \EE\big[\big\vert \ouhkRL(t_{n+1},\hx_{n+1},\pi) - \ouh(t_{n+1},\hx_{n+1},\pi)\big\vert\big].
	\end{align*}
	Since $\hx_{n+1}\in\mN_{\hhx}$ $\PPas$, we may proceed as above to recursively estimate $\ouhkRL(t_{k},\hx_{k},\p_m) - \ouh(t_{k},\hx_{k},\p_m)$ for $k = n+1,\ldots,N$,
	and conclude that
	\begin{align*}
		\notag
		&\vert \ouhkRL(t_n,\x_\ell,\p_m)  - \ouh(t_n,\x_\ell,\p_m)\vert 
		\\
		&\leq   (N-n){\rev 2\mLx\hhx} + \sum_{n'= n}^{N-1}\epskrL^{n'} +\EE\big[\vert \ouhkRL(t_{N},\hx_{N},\p_m)  - \ouh(t_{N},\hx_{N},\p_m)\vert\big]
		\\
		&\leq (N-n){\rev 2 \mLx\hhx} + \sum_{n'= n}^{N-1}\epskrL^{n'} .
	\end{align*}
\end{proof}

 
\begin{cor}
	Let the assumptions of \Cref{thm:mainRes,thm:error_RL} hold and assume in addition that {\rev $\hhx = \hht^{1+\delta}$ for an arbitrary $\delta>0$}.
	Then the sequence  $\{\ouhkRL\}$ converges uniformly on $  [0,T]\times \mD\times\DI$, i.e.,
	\begin{align*}
		\lim_{\substack{(s,\y,\q)\to (t,\x,\p)\\ \hht, \hhp \rightarrow 0}} \lim_{{ \kappa\to\infty}}  \ouhkRL(s,\y,\q) = u(t,\x,\p),
	\end{align*}
	where $u$ is the unique viscosity solution to \Cref{eq:hjb} in the class of bounded uniformly continuous
	functions, which are uniformly Lipschitz and convex in $\p$.
\end{cor}
\begin{proof}
For a fixed $(t,\x,\p) \in \mD_{T,I}\coloneqq[0,T]\times \mD\times \DI$ 
we consider a sequence $(\ot_k,\ox_k,\op_k)_{k\in\NN}$  
such that $(\ot_k,\ox_k,\op_k)\rightarrow (t,\x,\p)$ for $k\to\infty$
where $\ot_k := n_k\hht_k\in \Pi^{\hht_k}$, $n_k\in \NN$, $\op_k\in\mN_{\hhp_k}$, $\ox_k \in\mN_{\hhx_k}$ 
with $\hht_k :=T/N_k$, $\hhx_k := {\rev \hht_k^{1+\delta}}$, s.t., $\hhp_k \rightarrow 0$, $N_k\rightarrow \infty$ for $k\rightarrow \infty$.

	By \Cref{thm:error_RL} we estimate by the triangle inequality
	\begin{align*}
		&\vert \ouhkRL(\ot_k,\ox_k,\op_k)  - u(t,\x,\p)\vert\leq \lvert \ouhkRL(\ot_k,\ox_k,\op_k)  - \ouh(\ot_k,\ox_k,\op_k)\rvert + \vert\ouh(\ot_k,\ox_k,\op_k) - u(t,\x,\p)\vert
		\\
		&\quad \leq {\rev 2T\mLx\frac{\hhx_k}{\Delta t_k} + \sum_{n'= n_k}^{N_k-1}\epskrL^{n'}}+ \vert\ouh(\ot_k,\ox_k,\op_k) - u(t,\x,\p)\vert,
	\end{align*}
	recall the definition \cref{def_eps} of $\epskrL^{n}$.
	
	{\rev By the neural network approximation property \Cref{eq:univ_approx_groupsort}, for any $\hht_k$ there exists a sufficiently large $\kappa_k = \kappa(\hht_k)$
   such that $\varepsilon_{\kappa_k, \mathcal{D}}^{n}\le \hht_k^2$.}
	Noting that {\rev $\hhx_k = \hht_k^{1+\delta}$, $\delta>0$} we deduce ggg
	\begin{align*}
		&\lim_{\hht_k, \hhp_k \rightarrow 0}\lim_{\kappa\to \infty}\sup_{\mD_{T,I}}\vert \ouhkRL(\ot_k,\ox_k,\op_k)  - u(t,\x,\p)\vert 
		\\
		&\leq \lim_{\hht_k, \hhp_k \rightarrow 0} 2T\mLx{\rev \hht_k^\delta} +
\lim_{\hht_k, \hhp_k \rightarrow 0} \lim_{\kappa_k\to \infty}  {\rev \hht_k}
  + \lim_{\hht_k, \hhp_k \rightarrow 0}\sup_{\mD_{T,I}}\vert\ouh(\ot_k,\ox_k,\op_k) - u(t,\x,\p)\vert  .
	\end{align*}
	
	From \Cref{thm:mainRes} it follows that
	\begin{equation*}
		\lim_{\hht_k, \hhp_k \rightarrow 0}\sup_{\mD_{T,I}}\vert\ouh(\ot_k,\ox_k,\op_k) - u(t,\x,\p)\vert    = 0.
	\end{equation*}
	Consequently
	\begin{align*}
		&\lim_{\hht_k, \hhp_k \rightarrow 0}\lim_{\kappa\to \infty} \sup_{\mD_{T,I}} \vert \ouhkRL(\ot_k,\ox_k,\op_k)  - u(t,\x,\p)\vert= 0 .
	\end{align*}
\end{proof}

\section{Regularity properties of the semi-discrete solution}\label{sec:cont}

{\rev
We recall the following useful property of Lipschitz continuous functions.
\begin{prop}[{\cite[Proposition\,2\,(a)]{bally2003a}}]\label{prop:lipTransit}
Assume that \ref{A1} holds and let $\varphi:\RR^d\mapsto\RR$ be a function that is Lipschitz continuous and bounded. It follows that the map $\x\in\RR^d\mapsto\EE\big[\varphi(\oX^{n,\x}_{n+1})\big]$ is Lipschitz continuous with a Lipschitz coefficient   $\llbracket \varphi\rrbracket Q_\hht$,
where $Q_\hht :=\big(1 + (\llbracket b\rrbracket + \tfrac{1}{2}\llbracket a\rrbracket^2)\hht + \mO(\hht^2)\big)$.
\end{prop}}

%
\begin{lem}[Uniform Lipschitz continuity in $\x$ and uniform boundedness]\label{lem:lipx}
The map $\x\mapsto\ouh(t_n,\x,\p)$ is uniformly Lipschitz continuous, \ie for every $n\in\{0,\ldots,N\}$, it holds that
\[
\forall \x,\y\in\RR^d,\, \p\in\DI, \quad \lvert \ouh(t_n, \x, \p) - \ouh(t_n, \y, \p)\rvert\leq  \mLx\lvert \x - \y\rvert,  
\] 
with $\mLx := {\rev C} \max\big\{\llbracket f\rrbracket,\llbracket g\rrbracket,\llbracket h\rrbracket\big\}\ee^ {(\llbracket b\rrbracket + \tfrac{1}{2}\llbracket a\rrbracket^2)T}$,
{\rev with $C\geq 1$ for sufficiently small $\Delta t$}.
Moreover,   $(t_n,\x,\p)\mapsto\ouh(t_n,\x,\p)$ is uniformly bounded with
\begin{equation*}
\lVert\ouh\rVert_\infty\leq\max\big\{\lVert f\rVert_\infty,\lVert h\rVert_\infty\big\}.
\end{equation*}
\end{lem}
\begin{proof}
We fix $\x,\y\in\RR^d$, $\p\in\DI$ and proceed by induction for $n = N,N-1,\ldots,0$ to show that $\x\mapsto\ouh(t_n,\x,\p)$ is uniformly Lipschitz continuous and uniformly bounded. 
By \Cref{eq:fdis0} and \ref{A1}, \ref{A2}, the base case $n = N$ holds. 
We assume that at time level $t_{n+1}$ there exist a constant $\mLx^{n+1}>0$ such that
\begin{equation}
\label{eq:lipxhypo1}
\lvert \ouh(t_{n+1}, \x, \p) - \ouh(t_{n+1}, \y, \p)\rvert\leq  \mLx^{n+1}\lvert \x - \y\rvert,
\end{equation}
and that $\|\ouh(t_{n+1},\cdot, \cdot)\|_\infty \leq \max\big\{\lVert f\rVert_\infty,\lVert h\rVert_\infty\big\}$.

\textit{Uniform Lipschitz continuity in $\x$}. Suppose that $0\leq  \ouh(t_n,\y,\p)-\ouh(t_n,\x,\p)$. By \Cref{eq:fdisc5} and the convexity of $\p\mapsto \ouh(t_n,\y,\p)$, we have
\begin{align}\notag
\ouh(t_n,\y,\p)-\ouh(t_n,\x,\p) 
&= \ouh(t_n,\y,\p) - \sum_{i=1}^{I}\ouh(t_n,\x,\pinx_i(\p))\lambdanx_i(\p)
\\
\label{eq:lipx0}
&\leq \sum_{i=1}^{I}\big(\ouh(t_n,\y,\pinx_i(\p)) - \ouh(t_n,\x,\pinx_i(\p))\big)\lambdanx_i(\p).
\end{align}
The next step consists to derive an estimate for the summands in the right-hand side of \Cref{eq:lipx0}. 
Note that $\{\pinx_1(\p),\ldots,\pinx_I(\p)\}\subset\{\p_1,\ldots,\p_M\}$ and the map $F\mapsto \vexp F$ is nonexpansive. 
By \Cref{eq:fdisc11}, it follows that
\begin{equation}
\label{eq:lipx1}
\begin{split}
\ouh(t_n,\y,
&\pinx_i(\p)) - \ouh(t_n,\x,\pinx_i(\p))
\leq \max\Big\{\lvert\oY^{n,\y,\pi} - \oY^{n,\x,\pi}\rvert:{\pi\in \{\pinx_1(\p),\ldots,\pinx_I(\p)\}}\Big\}.
\end{split}
\end{equation}
{\rev Next, we estimate the term $\lvert\oY^{n,\y,\pi} - \oY^{n,\x,\pi}\rvert$ on the right-hand side of \Cref{eq:lipx1}.}
Noting that $\oY^{n,\y,\pi}$ is given by \Cref{eq:fdisc1},
using the inequalities \Cref{eq:lipx7,eq:lipx8} and \Cref{prop:lipTransit}, we deduce that
\begin{align}
\notag
\begin{split}
\lvert\oY^{n,\y,\pi} - \oY^{n,\x,\pi}\rvert
&\leq \Big\lvert \min\Big\{\max\Big\{\EE
\big[\ouh(t_{n+1},\oXny_{n+1},\pi)\big],\,\pi^\T f(t_n,\y)\Big\},\,\pi^\T h(t_n,\y)\Big\}
\\
&\hspace{20pt}- \min\Big\{\max\Big\{\EE
\big[\ouh(t_{n+1},\oXnx_{n+1},\pi)\big],\,\pi^\T f(t_n,\x)\Big\},\,\pi^\T h(t_n,\x)\Big\}\Big\rvert
\end{split}
\\
\notag
\begin{split}
&\leq \max\Big\{\EE\big\lvert\ouh(t_{n+1},\oX^{n,\y}_{n+1},\pi)-\ouh(t_{n+1},\oXnx_{n+1},\pi)\big\rvert,
\\
&\hspace{80pt}\lvert f(t_n,\y) - f(t_n,\x)\rvert,\,\lvert h(t_n,\y) - h(t_n,\x)\rvert\Big\}
\end{split}
\\
\label{eq:lipx3}
&\leq\max\big\{\mLx^{n+1} Q_\hht,\,\llbracket f\rrbracket,\,\llbracket h\rrbracket\big\}\lvert \x-\y\rvert,
\end{align}
where we used \ref{A1} to estimate the maps $\x\mapsto f(t_n, \x)$ and $\x\mapsto g(t_n, \x)$, as well as \Cref{prop:lipTransit} along with \Cref{eq:lipxhypo1}, to estimate  the maps $\x\mapsto \ouh(t_{n+1},\x,\pi)$.

We insert \Cref{eq:lipx3} into the right-hand side of \Cref{eq:lipx1}, which shows that
\begin{equation*}
\begin{split}
\ouh(t_n,\y,\pinx_\ell(\p)) &- \ouh(t_n,\x,\pinx_\ell(\p))\leq \max\big\{\mLx^{n+1} Q_\hht,\,\llbracket f\rrbracket,\,\llbracket h\rrbracket\big\}\lvert \x-\y\rvert.
\end{split}
\end{equation*}
If $0\leq  \ouh(t_n,\x,\p)-\ouh(t_n,\y,\p)$, we commute the role of $\x$ and $\y$ in the previous steps to get
\begin{equation*}
\begin{split}
\lvert \ouh(t_n,\y,p) &- \ouh(t_n,\x,\p)\rvert\leq \max\big\{\mLx^{n+1} 	Q_\hht,\,\llbracket f\rrbracket,\,\llbracket h\rrbracket\big\}\lvert \x-\y\rvert,
\end{split}
\end{equation*}
which implies that the map $\x\mapsto \ouh(t_n,\x,\p)$ is Lipschitz continuous with a Lipschitz coefficient $\mLx^n= \max\big\{\mLx^{n+1} Q_\hht,\,\llbracket f\rrbracket,\,\llbracket h\rrbracket\big\}$. A recursion  implies
\begin{align}
\notag
\mLx^n
&\leq \max\big\{\mLx^{n+1}Q_\hht,\,\llbracket f\rrbracket,\,\llbracket h\rrbracket\big\}\leq \max\big\{ \max\big\{\mLx^{n+2}Q_\hht,\,\llbracket f\rrbracket,\,\llbracket h\rrbracket\big\}Q_\hht,\,\llbracket f\rrbracket,\,\llbracket h\rrbracket\big\}
\\
\notag
&\leq \max\big\{\mLx^{n+2},\,\llbracket f\rrbracket,\,\llbracket h\rrbracket\big\} Q_\hht^2
\\
\notag
&\;\;\vdots
\\
\label{eq:lip6}
&\leq \max\big\{\mLx^N,\,\llbracket f\rrbracket,\,\llbracket h\rrbracket\big\}Q_\hht^{N-n}\leq
{\rev C} \max\big\{\llbracket f\rrbracket,\,\llbracket g\rrbracket,\,\llbracket h\rrbracket\big\}\ee^ {(\llbracket b\rrbracket + \tfrac{1}{2}\llbracket a\rrbracket^2)T}\eqqcolon \mLx,
\end{align}
{\rev with a constant $C\geq 1$, for sufficiently small $\Delta t$.}
Hence, the map $\x\mapsto\ouh(t_n,\x,\p)$ is uniformly Lipschitz continuous with a Lipschitz coefficient $\mLx$ defined in \Cref{eq:lip6}.

\textit{Uniform boundedness.} It follows immediately by \Cref{eq:fdisc11,eq:fdisc2,eq:fdisc5}  that
\begin{align*}
\lvert\ouh(t_n,\x,\p)\rvert&\leq \sum_{i = 1}^{I}\Big\lvert \min\Big\{\max\Big\{\EE
\big[\ouh(t_{n+1},\oXnx_{n+1},\pinx_i(\p))\big],
\\
&\hspace{40pt}(\pinx_i(\p))^\T f(t_n,\x)\Big\},\,(\pinx_i(\p))^\T h(t_n,\x)\Big\}\Big\rvert \lambdanx_i(\p)
\\
&\leq\max\big\{\lVert\ouh(t_{n+1},\cdot,\cdot)\rVert_\infty,\,\lVert f\rVert_\infty,\,\lVert h\rVert_\infty\big\}.
\end{align*}
A recursion implies
\begin{align*} 
\lvert\ouh(t_n,\x,\p)\rvert
&\leq \max\big\{\lVert\ouh(t_{n+1},\cdot,\cdot)\rVert_\infty,\,\lVert f\rVert_\infty,\,\lVert h\rVert_\infty\big\}
\\
&\leq \max\big\{\lVert\ouh(t_{n+2},\cdot,\cdot)\rVert_\infty,\,\lVert f\rVert_\infty,\,\lVert h\rVert_\infty\big\}
\\
\notag
&\;\;\vdots
\\
&\leq\max\big\{\lVert\ouh(t_{N}\hspace{10pt},\cdot,\cdot)\rVert_\infty,\,\lVert f\rVert_\infty,\,\lVert h\rVert_\infty\big\} =\max\big\{\lVert g\rVert_\infty,\lVert f\rVert_\infty,\lVert h\rVert_\infty\big\}.
\end{align*}
Hence, the map $(t_n,\x,\p)\mapsto\ouh(t_n,\x,\p)$ is uniformly bounded.
\end{proof}
%
\begin{lem}[Uniform Lipschitz continuity in $\p$]\label{lem:lipp}
Assuming that \Cref{lem:lipx} holds, then the map $\p\mapsto\ouh(t_n,\x,\p)$ is uniformly Lipschitz continuous, \ie for every $n\in\{0,\ldots,N\}$, it holds that 
\[
\forall \x\in\RR^d,\,\p,\q\in\DI, \quad \lvert \ouh(t_n, \x, \p) - \ouh(t_n, \x, \q)\rvert\leq  \mLp\lvert \p - \q\rvert,
\] 
with $\mLp := \max\big\{\lVert f\rVert_\infty,\,\llbracket g\rrbracket,\,\lVert h\rVert_\infty\big\}$.
\end{lem}

\begin{proof}
We fix $\x\in\RR^d$, $\p,\q\in\DI$ and proceed successively for $n = N,N-1,\ldots,0$ to show that $\p\mapsto \ouh(t_n,\x,\p)$ is uniformly Lipschitz continuous. By \ref{A1}, the base case $n = N$ holds. By recursion hypothesis, we suppose that there exist $\mLp^{n+1}>0$ such that
\begin{equation}\label{eq:lipphypo1}
\lvert \ouh(t_{n+1}, \x, \p) - \ouh(t_{n+1}, \x, \q)\rvert\leq  \mLp^{n+1}\lvert \p - \q\rvert.
\end{equation}

We suppose that $0\leq \ouh(t_n, \x, \q) - \ouh(t_n, \x, \p)$. By \Cref{eq:fdisc5} we have
\begin{equation}\label{eq:lipp00}
\ouh(t_n, \x, \q) - \ouh(t_n,\x,\p) 
=\ouh(t_n, \x, \q) - \sum_{ i= 1}^I\ouh(t_n,\x,\pinx_i(\p))\lambdanx_i(\p),
\end{equation}
and by \cite[Lemma~8.2.]{laraki2004on}, there exists $\{\varpinx_1(\p),\ldots,\varpinx_I(\p)\} \subset\DI$ such that
\begin{equation}\label{eq:lipp0}
\q = \sum_{i = 1}^{I}\varpinx_i(\q)\lambda^n_i(\x,\q),\quad\mbox{ and }\quad\lvert\p-\q\rvert =\sum_{i=1}^{I}\lvert \pinx_i(\p)-\varpinx_i(\q)\rvert\lambdanx_i(\p).
\end{equation}
Because the map $\p\mapsto \ouh(t_n, \x, \p)$ is convex, it follows from \Cref{eq:lipp00,eq:lipp0,eq:fdisc11} that
\begin{align}
\notag
\ouh(t_n, \x, \q) - \ouh(t_n,\x,\p) &\leq  \sum_{i=1}^{I}\big(\ouh(t_n, \x, \varpinx_i(\q)) -\ouh(t_n,\x,\pinx_i(\p))\big)\lambdanx_i(\p)
\\
\label{eq:lipp9}
&\leq \sum_{i=1}^{I}\lvert\oY^{n,\x,\varpinx_i(\q)}  -  \oY^{n,\x,\pinx_i(\p)}  \rvert\lambdanx_i(\p),
\end{align}
where $\p\mapsto  \oY^{n,\x,\p}$ is defined by
\begin{equation*}
\oY^{n,\x,\p} := \min\Big\{\max\Big\{\EE \big[\ouh(t_{n+1},\oXnx_{n+1},\p)\big],\,\p^\T f(t_n,\x)\Big\},\,\p^\T h(t_n,\x)\Big\}.
\end{equation*}
We apply the inequality \Cref{eq:lipx7} to the right-hand side of \Cref{eq:lipp9}, then by  \Cref{eq:lipphypo1}, \ref{A1}, and \Cref{eq:lipp0} we obtain that
\begin{align*}
\notag
\begin{split}
\ouh(t_n, \x, \q) - \ouh(t_n,\x,\p)
&\leq \max\big\{\mLp^{n+1},\,\lVert f\rVert_\infty,\,\lVert h\rVert_\infty\big\}\sum_{i=1}^{I}\lvert \pinx_i(\p)- \varpinx_i(\q)\rvert\lambdanx_i(\p)
\end{split}
\\
&= \max\big\{\mLp^{n+1},\,\lVert f\rVert_\infty,\,\lVert h\rVert_\infty\big\}\lvert \p - \q\rvert.
\end{align*}

If $0\leq \ouh(t_n, \x, \p) - \ouh(t_n, \x, \q)$, we commute the role of $\p$ and $\q$ in the above steps to get
\[
\lvert \ouh(t_n, \x, \p) - \ouh(t_n, \x, \q)\rvert\leq \mLp^{n}\lvert \p - \q\rvert
\]
with $\mLp^{n} := \max\big\{\mLp^{n+1},\,\lVert f\rVert_\infty,\,\lVert h\rVert_\infty\big\}$.
It follows immediately by recursion that
\begin{align}
\notag
\mLp^{n}&\leq \max\big\{\mLp^{n+1},\,\lVert f\rVert_\infty,\,\lVert h\rVert_\infty\big\}
\\
\notag
&\leq \max\big\{\mLp^{n+2},\,\lVert f\rVert_\infty,\,\lVert h\rVert_\infty\big\}
\\
\notag
&\;\;\vdots 
\\
\label{eq:lipp8}
&\leq\max\big\{\mLp^{N}\hspace{10pt},\,\lVert f\rVert_\infty,\,\lVert h\rVert_\infty\big\}=\max\big\{\llbracket g\rrbracket,\,\lVert f\rVert_\infty,\,\lVert h\rVert_\infty\big\}\eqqcolon \mLp.
\end{align}
Hence, the map $\p\mapsto\ouh(t_n,\x,\p)$ is uniformly Lipschitz continuous with a Lipschitz coefficient $\mLp$ defined by \Cref{eq:lipp8}.
\end{proof}
%
{\begin{lem}[Uniform almost H\"older continuity in $t$]\label{lem:lipt}
	Assuming that \Cref{lem:lipx} holds, then the map $t\mapsto\ouh(t,\x,\p)$ is uniformly almost H\"older continuous, \ie it holds that
	{\rev \begin{align*}
			\forall s,t\in[0,T],\,(\x,\p)\in\RR^d\times\DI,\,\lvert\ouh(s,\x,\p) - \ouh(t,\x,\p)\rvert\leq \mLt\big(\sqrt{\hht} 
			+  \sqrt{\vert s-t\vert}\big),
		\end{align*}
		with $\mLt := 8\max\{\mLx,\,\llbracket f\rrbracket,\,\llbracket h\rrbracket\}\max\{\Vert b\Vert_{\infty},\Vert a\Vert_{\infty}\}\max\{1,\sqrt{T}\}$.
	}
\end{lem}
\begin{proof}
	We fix $(\x,\p)\in\RR^d\times\DI$. We start the proof with $s,t\in\Pi^\hht$. For every $n\in\{0,\ldots,N-1\}$ and $n'\in\{0,\ldots,N-n\}$, we have that
	\begin{align}
		\label{eq:holdt7}
		\begin{split}
			&\lvert\ouh(t_n,\x,\p) - \ouh(t_{n+n'},\x,\p)\rvert \leq \big\lvert\ouh(t_n,\x,\p) - \EE\big[\ouh(t_{n+n'},\oXnx_{n+n'},\p)\big]\big\rvert
			\\
			&+ \big\lvert\EE\big[\ouh(t_{n+n'},\oXnx_{n+n'},\p)\big] - \ouh(t_{n+n'},\x,\p)\big\rvert,
		\end{split}
	\end{align} 
	where the discrete process $\oXnx_{n+n'}$ is given by \Cref{eq:euler1}. By \Cref{lem:lipx} and \ref{A1} it is straightforward to get
	{\rev
		\begin{align}
			\notag
			\big\lvert\EE\big[\ouh(t_{n+n'},\oXnx_{n+n'},\p)\big] - \ouh(t_{n+n'},\x,\p)\big\rvert
			&\leq \mLx\EE\Big\lvert \sum_{j=n}^{n+n'-1} \Big(b(t_j,\oX^{n,\x}_j)\hht  + a(t_j,\oX^{n,\x}_j) {\xi^j} \sqrt{\hht}\Big)\Big\rvert
			\\
			\label{eq:holdt6}
			&\leq \mLx\max\{\Vert b\Vert_{\infty},\Vert a\Vert_{\infty}\}\big((t_{n+n'} - t_{n})  + (t_{n+n'} - t_{n})^{1/2}\big).
		\end{align}
	}
	It remains to derive an estimate for the first term on the right-hand side of \Cref{eq:holdt7}. We proceed in two steps to show that
	{\rev \begin{equation}\label{eq:holdt006}
			\lvert \ouh(t_n,\x,\p)- \EE \big[ \ouh(t_{n+n'},\oXnx_{n+n'},\p)\rvert \leq \max\{\llbracket f\rrbracket,\,\llbracket h\rrbracket\}\hht^{1/2}.
	\end{equation}}
	
	\textbf{Step 1}: We prove that 
	{\rev
		\begin{align}
			\label{eq:holdt22}
			\ouh(t_n,\x,\p)-\EE\big[\ouh(t_{n+n'},\oXnx_{n+n'},\p)\big]\geq -2\llbracket h\rrbracket\max\{\Vert b\Vert_{\infty},\Vert a\Vert_{\infty}\}\hht^{1/2}.
	\end{align}}
	Let $\p\in\{\p_1,\ldots,\p_M\}\subset\DI$. The formula \Cref{eq:fdisc11} implies that
	{\rev
		\begin{align*}
			\ouh(t_n,\x,\p)
			&\geq \min \Big\{\EE \big[\ouh(t_{n+1},\oXnx_{n+1},\p)\big],\,\p^\T h(t_n,\x)\Big\},
			\\
			&=\min_{\mu\in\mT^{t_n}_{[n..N]}}\EE\big[ \ind_{\{\mu=n\}}\p^\T h(t_n,\oXnx_{n}) + \ind_{\{\mu> n\}}\ouh(t_{n+1},\oXnx_{n+1},\p)\big]
			\\
			&=\min_{\mu\in\mT^{t_n}_{[n..N]}}\EE\Big[ \ind_{\{\mu=n\}}\Big(\p^\T h(t_n,\oXnx_{n}) - \ouh(t_{n+1},\oXnx_{n+1},\p)\Big) + \ouh(t_{n+1},\oXnx_{n+1},\p)\Big].
		\end{align*}
		By induction, we have
		\begin{align*}
			\ouh(t_{n+1},\oXnx_{n+1},\p)
			&\geq \min \bigg\{\EE \big[\ouh(t_{n+2},\oXnx_{n+2},\p)\big],\,\p^\T h(t_{n+1},\oXnx_{n+1})\bigg\},
			\\
			&=\min_{\mu\in\mT^{t_{n+1}}_{[(n+1)..N]}}\EE\Big[ \ind_{\{\mu={n+1}\}}\p^\T h(t_{n+1},\oXnx_{n+1}) + \ind_{\{\mu> n+1\}}\ouh(t_{n+2},\oXnx_{n+2},\p)\Big]
			\\
			&=\min_{\mu\in\mT^{t_{(n+1)}}_{[(n+1)..N]}}\EE\Big[ \ind_{\{\mu=n+1\}}\Big(\p^\T h(t_{n+1},\oXnx_{n+1}) - \ouh(t_{n+2},\oXnx_{n+2},\p)\Big) + \ouh(t_{n+2},\oXnx_{n+2},\p)\Big],
		\end{align*}
		thus
		\begin{align*}
			\ouh(t_n,\x,\p)
			&\geq \min_{\mu\in\mT^{t_n}_{[n..N]}}\EE\Big[ \sum_{k=n}^{n+1}\ind_{\{\mu=k\}}\Big(\p^\T h(t_k,\oXnx_{k}) - \ouh(t_{k+1},\oXnx_{k+1},\p)\Big) + \ouh(t_{n+2},\oXnx_{n+2},\p)\Big].
		\end{align*}
		Recall that $\ouh(t_{k+1},\oXnx_{k+1},\p)\leq \p^\T h(t_{k+1},\oXnx_{k+1})$. We repeat the induction $(n'-2)$ times to get
		\begin{align}
			\notag
			&\ouh(t_n,\x,\p)
			\\
			\notag
			&\geq \min_{\mu\in\mT^{t_n}_{[n..N]}}\EE\Big[ \sum_{k=n}^{n+n'-1}\ind_{\{\mu=k\}}\Big(\p^\T h(t_k,\oXnx_{k}) - \ouh(t_{k+1},\oXnx_{k+1},\p)\Big) + \ouh(t_{n+n'},\oXnx_{n+n'},\p)\Big]
			\\
			\notag
			&\geq \min_{\mu\in\mT^{t_n}_{[n..N]}}\EE\Big[ \sum_{k=n}^{n+n'-1}\ind_{\{\mu=k\}}\Big(\p^\T h(t_k,\oXnx_{k}) - \p^\T h(t_{k+1},\oXnx_{k+1})\Big) + \ouh(t_{n+n'},\oXnx_{n+n'},\p)\Big],
		\end{align}
	}
	and by \ref{A1}, we have for all $\p\in\{\p_1,\ldots,\p_M\}\subset\DI$,
	{\rev
		\begin{equation}
			\label{eq:holdt2}
			\ouh(t_n,\x,\p)\geq -2\llbracket h\rrbracket\max\{\Vert b\Vert_{\infty},\Vert a\Vert_{\infty}\}\hht^{1/2} +  \EE\big[\ouh(t_{n+n'},\oXnx_{n+n'},\p)\big].
	\end{equation}}
	Now let $\p\in\DI\setminus \{\p_1,\ldots,\p_M\}$ and use \Cref{eq:fdisc5}. Since $\{\pinx_1(\p),\ldots,\pinx_I(\p)\}\subset \{\p_1,\ldots,\p_M\}$, we can apply \Cref{eq:holdt2} to get
	{\rev
		\begin{equation}
			\label{eq:holdt3}  
			\begin{split}
				&\ouh(t_n,\x,\pinx_i(\p))
				\geq -2\llbracket h\rrbracket \max\{\Vert b\Vert_{\infty},\Vert a\Vert_{\infty}\}\hht^{1/2} +  \EE\big[\ouh(t_{n+n'},\oXnx_{n+n'},\pinx_i(\p))\big].
			\end{split}
	\end{equation}}
	Next, we multiply both sides of \Cref{eq:holdt3} by $\lambdanx_i(\p)$, then sum for $i = 1,\ldots,I$. It follows by \Cref{eq:fdisc5} and the convexity of $\p\mapsto \ouh(t_{n+n'},\oXnx_{n+n'},\p)$ that
	{\rev
		\begin{align*}
			\notag
			&\ouh(t_n,\x,\p)
			\\
			&\geq -2\llbracket h\rrbracket\max\{\Vert b\Vert_{\infty},\Vert a\Vert_{\infty}\}\hht^{1/2} +  \EE\big[\sum_{ i= 1}^I\ouh(t_{n+n'},\oXnx_{n+n'},\pinx_i(\p))\lambdanx_i(\p)\big]
			\\
			&\geq -2\llbracket h\rrbracket\max\{\Vert b\Vert_{\infty},\Vert a\Vert_{\infty}\}\hht^{1/2} +  \EE\big[\ouh(t_{n+n'},\oXnx_{n+n'},\p)\big].
	\end{align*}}
	This proves \Cref{eq:holdt22}.
	
	\medskip
	
	\textbf{Step 2}: We prove that 
	{\rev \begin{align}
			\label{eq:holdt333}
			\ouh(t_n,\x,\p)-\EE\big[\ouh(t_{n+n'},\oXnx_{n+n'},\p)\big]\leq 2\llbracket f\rrbracket\max\{\Vert b\Vert_{\infty},\Vert a\Vert_{\infty}\}\hht^{1/2}.
	\end{align}}
	
	Let $\p\in\{\p_1,\ldots,\p_M\}\subset\DI$. The formula \Cref{eq:fdisc11} implies that
	{\rev
		\begin{align*}
			\notag
			\ouh(t_n,\x,\p)&\leq \max\Big\{\EE \big[\ouh(t_{n+1},\oXnx_{n+1},\p)\big],\,\p^\T f(t_n,\x)\Big\}
			\\
			\notag
			&= \max_{\nu\in\mT^{t_n}_{[n..N]}} \EE \Big[\ind_{\{\nu=n\}}\p^\T f(t_n,\oXnx_{n}) + \ind_{\{\nu> n\}}\ouh(t_{n+1},\oXnx_{n+1},\p)\Big]
			\\
			&= \max_{\nu\in\mT^{ton}_{[n..N]}} \EE \Big[\ind_{\{\nu=n\}}\Big(\p^\T f(t_n,\oXnx_{n})-\ouh(t_{n+1},\oXnx_{n+1},\p)\Big) + \ouh(t_{n+1},\oXnx_{n+1},\p)\Big].
		\end{align*}
		By induction, we have
		\begin{align*}
			\notag
			\ouh(t_{n+1},\oXnx_{n+1},\p)&\leq \max_{\nu\in\mT^{t_{n+1}}_{[(n+1)..N]}} \EE \Big[\ind_{\{\nu=n+1\}}\Big(\p^\T f(t_{n+1},\oXnx_{n+1})-\ouh(t_{n+2},\oXnx_{n+2},\p)\Big) + \ouh(t_{n+2},\oXnx_{n+2},\p)\Big],
		\end{align*}
		thus
		\begin{equation*}
			\ouh(t_n,\x,\p)\leq \max_{\nu\in\mT^{t_{n}}_{[n..N]}} \EE \Big[\sum_{k=n}^{n+1}\ind_{\{\nu=k\}}\Big(\p^\T f(t_{k},\oXnx_{k})-\ouh(t_{k+1},\oXnx_{k+1},\p)\Big) + \ouh(t_{n+2},\oXnx_{n+2},\p)\Big].
		\end{equation*}
		We recall that  $\p^\T f(t_{k},\oXnx_{k})\leq \ouh(t_{k},\oXnx_{k},\p)$. We repeat the induction $(n'-2)$ times to get
		\begin{align*}
			\ouh(t_n,\x,\p)&\leq \max_{\nu\in\mT^{t_{n}}_{[n..N]}} \EE \Big[\sum_{k=n}^{n+n'-1}\ind_{\{\nu=k\}}\Big(\p^\T f(t_{k},\oXnx_{k})-\ouh(t_{k+1},\oXnx_{k+1},\p)\Big) + \ouh(t_{n+n'},\oXnx_{n+n'},\p)\Big]
			\\
			&\leq \max_{\nu\in\mT^{t_{n}}_{[n..N]}} \EE \Big[\sum_{k=n}^{n+n'-1}\ind_{\{\nu=k\}}\Big(\p^\T  f(t_{k},\oXnx_{k})-\p^\T f(t_{k+1},\oXnx_{k+1})\Big) + \ouh(t_{n+n'},\oXnx_{n+n'},\p)\Big],
		\end{align*}
		and  by \ref{A1}, we have for all $\p\in\{\p_1,\ldots,\p_M\}\subset\DI$
		\begin{equation}
			\label{eq:holdt33}
			\ouh(t_n,\x,\p)\leq 2\llbracket f\rrbracket\max\{\Vert b\Vert_{\infty},\Vert a\Vert_{\infty}\}\hht^{1/2} + \EE \big[ \ouh(t_{n+n'},\oXnx_{n+n'},\p)\big].
	\end{equation}}
	Now, let $\p\in\DI\setminus \{\p_1,\ldots,\p_M\}$. By \Cref{eq:fdisc5} and because $\big\{\pinX_1(\p),\ldots,\pinX_I(\p)\big\}\subset \{\p_1,\ldots,\p_M\}$, we can apply \Cref{eq:holdt33} to get
	{\rev
		\begin{equation}
			\label{eq:lipt04}  
			\begin{split}
				&\ouh(t_n,\x,\pinX_i(\p))
				\leq {\rev 2\llbracket f\rrbracket}\max\{\Vert b\Vert_{\infty},\Vert a\Vert_{\infty}\}\hht^{1/2} +  \EE\big[\ouh(t_{n+n'},\oXnx_{n+n'},\pinX_i(\p))\big].
			\end{split}
	\end{equation}}
	Next, we multiply both sides of \Cref{eq:lipt04} by $\lambdanX_i(\p)$, then sum for $i = 1,\ldots,I$. It follows by \Cref{eq:fdisc5} and the convexity of $\p\mapsto \ouh(t_{n+n'},\x,\p)$ that
	{\rev
		\begin{align}
			\notag
			&2\llbracket f\rrbracket\max\{\Vert b\Vert_{\infty},\Vert a\Vert_{\infty}\}\hht^{1/2} +  \EE\big[\ouh(t_{n+n'},\oXnx_{n+n'},\p)\big]
			\\\notag
			&\geq \sum_{i=1}^{I}\ouh(t_n,\x,\pinX_i(\p))\lambdanX_i(\p)\geq  \ouh(t_n,\x,\p).
	\end{align}}
	This proves \Cref{eq:holdt333}.
	
	\medskip
	
	Now we combine the estimates \Cref{eq:holdt6,eq:holdt006}, which provides that
	{\rev
		\begin{equation}
			\label{eq:holdt8}
			\begin{split}
				&\lvert\ouh(t_n,\x,\p) - \ouh(t_{n+n'},\x,\p)\rvert
				\\
				&\leq\max\{\mLx,\,\llbracket f\rrbracket,\,\llbracket h\rrbracket\}\max\{\Vert b\Vert_{\infty},\Vert a\Vert_{\infty}\}\big(2\hht^{1/2} + (t_{n+n'} - t_{n})  + (t_{n+n'} - t_{n})^{1/2}\big).
			\end{split}
	\end{equation}}

	We conclude the proof by considering the case where $t\in[t_n,t_{n+1}]$ and $s\in[t_{n+n'},t_{n+n'+1}]$. We estimate the term $\sqrt{t_{n + n'} -t_n}$ appearing in the right hand side of \Cref{eq:holdt8} as follows. We use the inequality $\sqrt{\x + \y + \z}\leq \sqrt{\x} + \sqrt{\y}+ \sqrt{\z}$ for any positive number $\x, \y, \z$; and since $t\in[t_{n},t_{n +1}]$ and $s\in[t_{n +n'},t_{n+n'+1}]$, we get
	\begin{align*}
		\sqrt{t_{n+n'} -t_n} =  \sqrt{(t-t_n)+  (s-t) + (t_{n+n'}-s)} \leq 2\sqrt{\hht} +   \sqrt{s-t}.
	\end{align*}
	{\rev
		We have also
		\begin{align*}
			(t_{n+n'} - t_{n}) = (t - t_{n}) + (s-t) + (t_{n+n'} - s)\leq 2{\hht} +   {s-t}\leq 2{\hht} +   \sqrt{T}\sqrt{s-t}.
	\end{align*}}
	It follows that
	{\rev
		\begin{equation}
			\label{eq:holdt9}
			\begin{split}
				&\lvert\ouh(t_n,\x,\p) - \ouh(t_{n+n'},\x,\p)\rvert
				\\
				&\leq \max\{\mLx,\,\llbracket f\rrbracket,\,\llbracket h\rrbracket\}\max\{\Vert b\Vert_{\infty},\Vert a\Vert_{\infty}\}\big(2{\hht} + 4\sqrt{\hht} + \max\{1,\sqrt{T}\}\sqrt{s-t}\big).
			\end{split}
	\end{equation}}
	Finally, by \Cref{eq:fdisc5,eq:holdt9}, it implies that
	{\rev
		\begin{equation*}
			\lvert\ouh(s,\x,\p) - \ouh(t,\x,\p)\rvert\leq 8\max\{\mLx,\,\llbracket f\rrbracket,\,\llbracket h\rrbracket\}\max\{\Vert b\Vert_{\infty},\Vert a\Vert_{\infty}\}\max\{1,\sqrt{T}\}\big(\sqrt{\hht} 
			+\sqrt{s-t}\big).
	\end{equation*}}
\end{proof}
}

\section{Viscosity solution property}\label{sec:viscosityPty} 

In this section, we show the second part of \Cref{thm:mainRes}.
Here, $w$ denotes the limit of a subsequence of $\{\ouh\}_{\hh}$ on an arbitrary compact subsets of $[0,T]\times\RR^d\times \DI$ obtained by the Arzel\'a-Ascoli theorem.
As a limit of $\ouh$, which is uniformly continuous and bounded function  due to \Cref{lem:lipx}, \Cref{lem:lipp}, \Cref{lem:lipt}, $w$ is bounded and uniformly continuous.
As a limit of $\ouh$, which is convex in $\p$ by construction due to \cref{eq:fdisc5,eq:basis}, $w$ is convex in $\p$.
\subsection{Viscosity subsolution property of \texorpdfstring{$w$}{Lg}}
\begin{prop}\label{prop:subvis}
The limit $w$ is a viscosity subsolution of \Cref{eq:hjb} on $[0,T]\times\RR^d\times\Int(\DI)$.
\end{prop}
\begin{proof}
Let $\varphi:[0,T]\times\RR^d\times\DI\rightarrow \RR$ be a smooth test function such that $w - \varphi$ has a strict global maximum at $(\ot,\ox,\op)$. We have to show that $w$ satisfies \Cref{def:viscosity}(i) at $(t,\x,\p) = (\ot,\ox,\op)$, which equivalently consists to show that
{\rev
\begin{alignat}{3}
\label{eq:ass1}\tag{P1}
&\lambda(\op,\DD_\p^2\varphi(\ot, \ox, \op))\geq 0,&&\mbox{ and }&&  
\\
\label{eq:ass2}\tag{P2}
&\,\varphi(\ot, \ox, \op)\leq \op^\T h(\ot,\ox),&&\mbox{ and }&& 
\\
\label{eq:ass3}\tag{P3}
&\varphi(\ot,\ox,\op) > \op^\T f(\ot,\ox)&&\mbox{ implies }&&\mL \varphi(\ot,\ox,\op)\leq0. 
\end{alignat}}

{\rev As a limit of a sequence of convex functions in $\p$, $w$ is also a convex function in $\p$. Moreover, by \cite[Theorem~1]{oberman2007the}, if $w$ is a convex function in $\p$, then it is the viscosity solution of the obstacle problem {\revv$-\lambda(\p,\DD_\p^2w(t, \x, \p)) \leq 0$}. It means that the test function $\varphi$ satisfies {\revv$-\lambda(\p,\DD_\p^2\varphi(t, \x, \p)) \leq 0$} at $(t, \x, \p) = (\ot, \ox, \op)$. Hence, \Cref{eq:ass1} holds.
}

By \Cref{eq:fdisc11}, $\forall\,(n,\x,\p)\in\{0,\ldots,N\}\times\RR^d\times\DI$, it holds that $\ouh(t_n,\x,\p)\leq \p^\T h(t_n,\x)$. Thus by continuity, we obtain \Cref{eq:ass2}.

It remains to show \Cref{eq:ass3}. 
For $\hht_k=T/N_k$, 
$\hhp_k$, s.t., $N_k\rightarrow \infty$ (i.e. $\hht_k \rightarrow 0$), $\hhp_k \rightarrow 0$ for $k\rightarrow \infty$
we consider a sequence $(\ot_k,\ox_k,\op_k)_{k\in\NN}$  with $\ot_k= n_k\hht_k\in \Pi^{\hht_k}$, $n_k\in \NN$, $\op_k\in\mN_{\hhp_k}$ 
such that $(\ot_k,\ox_k,\op_k)\rightarrow (\ot,\ox,\op)$ for $k\to\infty$ and s.t. $\ouh - \varphi$ admits a global maximum at $(\ot_k,\ox_k,\op_k)$ with maximum value equal to $0$, cf.\ \cite[Lemma~2.4]{bardi2009optimal}.
Define $\varphi^\hh := \varphi + (\ouh - \varphi)(\ot_k,\ox_k,\op_k)$. Then, for all $(\x,\p)\in\RR\times\DI$ we have
\begin{equation} \label{eq:subvis3}
(\ouh - \varphi^\hh)(\ot_k + \hht_k,\x,\p)\leq (\ouh - \varphi^\hh)(\ot_k,\ox_k,\op_k) = 0.
\end{equation}

To prove \Cref{eq:ass3} we use arguments similar to \cite[Chapter 3]{touzi2013optimal}, i.e.
we assume in contrario that
{\rev
\begin{equation*}
\varphi(\ot,\ox,\op) > \op^\T f(\ot,\ox)\quad\mbox{and}\quad \mL \varphi(\ot,\ox,\op)>0,
\end{equation*}}
and work toward a contradiction on the dynamic programming principle, which implies that
\begin{equation}\label{eq:subvis12}
\begin{split}
\ouh&(\ot_k,\ox_k,\op_k)
\leq \max_{\nu\in\mT^{\ot_k}_{[ n_k..N]}}\EE\big[\ouh(\ot_k   + \hht_k,\oX^{n_k,\ox_k}_{n_k+ 1},\op_k)\ind_{\{\nu\geq n_k   + 1\}}   + \op_k^\T f(t_\nu,\ox_k)\ind_{\{\nu< n_k   + 1\}}\big].
\end{split}
\end{equation}

By \Cref{eq:subvis3} and the continuity of $\varphi$ and $f$, we can find $\delta>0$ and $k$ large enough (equivalently $\hht_k$ small enough) such that
\begin{equation}\label{eq:subvis5}
\varphi^\hh(\ot_k,\ox_k,\op_k) \geq \op_k^\T f(\ot_k,\ox_k)+\delta\quad\mbox{and}\quad \mL\varphi^\hh(\ot_k,\ox_k,\op_k)\geq\delta,
\end{equation}
for $(\ot_k,\ox_k,\op_k)\in\hht_k\mB $, where $\mB $ is the unit ball of $[0,T]\times\RR^d\times\DI$ centered at $(\ot,\ox,\op)$. We choose an arbitrary stopping rule $\nu\in\mT^{\ot_k}_{[n_k..N]}$. By \Cref{eq:subvis3}, we have that
\begin{equation}
\label{eq:subvis4}
\begin{split}
\EE&\big[\ouh((\ot_k + \hht_k)\wedge t_\nu,\oX^{n_k,\ox_k}_{(n_k +1)\wedge \nu},\op_k)\big] -\ouh (\ot_k,\ox_k,\op_k) 
\\
&= \EE\big[(\ouh - \varphi^\hh)((\ot_k + \hht_k)\wedge t_\nu,\oX^{n_k,\ox_k}_{(n_k +1)\wedge \nu},\op_k) \big]
\\
&\qquad +\EE\big[\varphi^\hh((\ot_k + \hht_k)\wedge t_\nu,\oX^{n_k,\ox_k}_{(n_k +1)\wedge \nu},\op_k)\big]  - \varphi^\hh(\ot_k,\ox_k, \op_k).
\end{split}
\end{equation}
Applying the Taylor formula to the term $\varphi^\hh((\ot_k + \hht_k)\wedge t_\nu,\oX^{n_k,\ox_k}_{(n_k\wedge \nu)+ 1},\op_k)$ appearing in the right-hand side of \Cref{eq:subvis4}, we get
\begin{equation}
\label{eq:subvis6}
\begin{split}
\EE&\big[\ouh((\ot_k + \hht_k)\wedge t_\nu,\oX^{n_k,\ox_k}_{(n_k +1)\wedge \nu},\op_k)\big] -\ouh (\ot_k,\ox_k,\op_k) 
\\
&=\EE\big[(\ouh - \varphi^\hh)((\ot_k + \hht_k)\wedge t_\nu,\oX^{n_k,\ox_k}_{(n_k +1)\wedge \nu},\op_k) 
\\
&\qquad+(\dt + \mathcal{L})\varphi^\hh(\ot_k,\ox_k,\op_k)(\hht_k\wedge(t_\nu - \ot_k))\big] + C\hht_k\mO({\hht_k}^{1/2}).
\end{split}
\end{equation}
For $k$ large enough (equivalently $\hht_k$ small enough), $(\ot_k,\ox_k,\op_k)\in\hht_k\mB$. Thus the assertion \Cref{eq:subvis5} applies to $\mL\varphi^\hh$. It follows directly from \Cref{eq:subvis6}  that 
\begin{align}
\notag
\EE&\big[\ouh((\ot_k + \hht_k)\wedge t_\nu,\oX^{n_k,\ox_k}_{(n_k +1)\wedge \nu},\op_k)\big] -\ouh (\ot_k,\ox_k,\op_k) 
\\
\notag
&\leq \EE\big[(\ouh - \varphi^\hh)((\ot_k + \hht_k)\wedge t_\nu,\oX^{n_k,\ox_k}_{(n_k +1)\wedge \nu},\op_k) \big] + C\hht_k\mO({\hht_k}^{1/2})
\\
\label{eq:subvis7}
\begin{split}
&\leq \EE\big[(\ouh - \varphi^\hh)(\ot_k  + \hht_k,\oX^{n_k,\ox_k}_{n_k+ 1},\op_k)\ind_{\{\nu\geq n_k  + 1 \}} 
\\
&\qquad+ (\ouh - \varphi^\hh)(t_\nu,\oX^{n_k,\ox_k}_{\nu},\op_k)\ind_{\{\nu<n_k  + 1\}}\big] + C\hht_k\mO({\hht_k}^{1/2}).
\end{split}
\end{align}
Note that if $\nu<n_k  + 1$, then $(t_\nu,\oX^{n_k,\ox_k}_{\nu},\op_k)\in\hht_k\mB$. Therefore, on one hand, by \Cref{eq:subvis3} and the continuity of $(\ouh - \varphi^\hh)$ it implies for the second term in the right-hand side of \Cref{eq:subvis7} that $(\ouh-\varphi^\hh)(t_\nu,\oX^{n_k,\ox_k}_{\nu},\op_k) = 0$. Thus we have 
\begin{equation}\label{eq:subvis9}
\begin{split}
	&\EE\big[\ouh((\ot_k + \hht_k)\wedge t_\nu,\oX^{n_k,\ox_k}_{(n_k +1)\wedge \nu},\op_k)\big] -\ouh (\ot_k,\ox_k,\op_k)
\leq -C\hht_k\mO({\hht_k}^{1/2})-\gamma\PP[\nu\geq n_k  + 1], 
\end{split}
\end{equation}
with 
\begin{equation*}
-\gamma= \max_{\hht_k\partial\mB} (\ouh-\varphi)<0.
\end{equation*}
On the other hand, the assertion \Cref{eq:subvis5} applies to $\varphi^\hh$. Thus we have
\begin{align}
\notag
\EE&\big[\ouh((\ot_k + \hht_k)\wedge t_\nu,\oX^{n_k,\ox_k}_{(n_k +1)\wedge \nu},\op_k)\big] 
\\
\notag
&= \EE\big[\ouh(\ot_k + \hht_k,\oX^{n_k,\ox_k}_{n_k+ 1},\op_k)\ind_{\{\nu\geq n_k   + 1\}}  + \ouh(t_\nu,\oX^{n_k,\ox_k}_{\nu},\op_k)\ind_{\{\nu<n_k   + 1\}}\big]
\\
\label{eq:subvis8}
\begin{split}
&\geq \EE\big[\ouh(\ot_k + \hht_k,\oX^{n_k,\ox_k}_{n_k+ 1},\op_k)\ind_{\{\nu\geq n_k   + 1\}} 
+ \op_k^\T f(t_\nu,\ox_k)\ind_{\{\nu<n_k   + 1\}}\big] + \delta\PP[\nu< n_k  + 1].
\end{split}
\end{align}
Hence, combining \Cref{eq:subvis9,eq:subvis8}, we arrive at
\begin{equation*}
\begin{split}
&\ouh(\ot_k,\ox_k, \op_k) 
\\
&\geq \EE\big[\ouh(\ot_k  + \hht_k,\oX^{n_k,\ox_k}_{n_k+ 1},\op_k)\ind_{\{\nu\geq n_k   + 1\}}   
+ \op_k^\T f(t_\nu,\ox_k)\ind_{\{\nu<n_k   + 1\}}\big]
 +C\hht_k\mO({\hht_k}^{1/2})+ (\gamma \wedge \delta).
\end{split}
\end{equation*}
Since $\nu$ is arbitrary, the above inequality provides the desired contradiction to \Cref{eq:subvis12}.

\end{proof}
\subsection{Viscosity supersolution property of \texorpdfstring{$w$}{Lg}}

To establish the viscosity supersolution property of the candidate limit $w$, we construct martingale processes that satisfy a one-step-ahead dynamic programming principle.

We note that \Cref{eq:fdisc5} implies
\begin{equation} \label{eq:vexpu}
\begin{split}
\ouh(t_n,\x,\p) 	=  \sum_{i=1}^{I}\min\Big\{&\max\Big\{\EE \big[\ouh(t_{n+1},\oXnx_{n+1},\pinx_i(\p))\big],
\\
&\,(\pinx_i(\p))^\T f(t_n,\x)\Big\},\,(\pinx_i(\p))^\T h(t_n,\x)\Big\}\lambdanx_i(\p).
\end{split}
\end{equation}

Similar to \cite{banas2020numerical,cardaliaguet2009on,gruen2012aprobabilistic}, it is now possible to construct the so called \textit{one-step a posteriori martingales}, which start at $\p$ and jump then to one of the support points of the convex hull $ \pinx_1(\p),\ldots, \pinx_I(\p)$, $(n,\x,\p)\in\{0,\ldots,N\}\times\RR^d\times\{\p_1,\ldots,\p_M\}$.

\begin{defn}\label{def:feedBack}
For all $i = 1,\ldots,I$ and $(n,\x,\p)\in \{0,\ldots,N\}\times\RR^d\times\{\p_1,\ldots,\p_M\}$ we define the one-step feedbacks $\bp^{i,\x,\p}_{k+1}$ as $\big\{\pinx_1(\p),\ldots, \pinx_I(\p)\big\}$-valued random variables that are independent of $\sigma\{B_s: 0\leq s\leq T\}$, such that
\begin{itemize}
\item for $n = 0,\ldots,N-1$
\begin{enumerate}[label = $\circ$]
\item if $p_i = 0$, set $\bp_{n+1}^{i,\x,\p} = \p$.
\item if $p_i > 0$, choose $\bp_{n+1}^{i,\x,\p}$ among $\big\{\pinx_1(\p),\ldots, \pinx_I(\p)\big\}$ with probability
\begin{equation*}
\PP\Big[\bp_{n+1}^{i,\x,\p} = \pinx_\ell(\p)\big| \big(\bp_{n}^{j,\y,\q}\big)_{j\in\{1,\ldots,I\}, (n,\y,\q)\in \{1,\ldots,k\}\times\RR^d\times\{\p_1,\ldots,\p_M\}}\Big] = \frac{(\pinx_\ell(\p))_i}{p_i}\lambdanx_\ell(\p).
\end{equation*}
\end{enumerate}
\item for $k=N$, set $\bp_{N+1}^{i,\x,\p} = e_i$, where $\{e_i:i=1,\ldots,I\}$ is the canonical basis of $\RR^I$.
\end{itemize}

Moreover, we define $\bp_{n+1}^{\x,\p}= \bp_{n+1}^{\bi,\x,\p}$, where the index $\bi$ is a random variable with law $\PP[\bi = i] = p_i$, independent of the processes $B$ and  $\big(\bp_{n}^{j,\x,\p}\big)_{j\in\{1,\ldots,I\}, (n,\x,\p)\in\{1,\ldots,N\}\times\RR^d\times\{\p_1,\ldots,\p_M\}}$.
\end{defn}

For all $(n,\x,\p)\in\{0,\ldots,N\}\times\RR^d\times\{\p_1,\ldots,\p_M\}$, the process $\bp_{n+1}^{\x,\p}$ defined by \Cref{def:feedBack} is a martingale. Furthermore, it satisfies the following one-step dynamic programming {principle}.

\begin{lem}\label{lem:dynamicProg}
For all $(n,\x,\p)\in\{0,\ldots,N\}\times\RR^d\times\{\p_1,\ldots,\p_M\}$ it holds that
\begin{equation*} 
\ouh(t_n,\x,\p) =   \EE\Big[\min\Big\{\max\Big\{\ouh(t_{n+1},\oXnx_{n+1},\bp_{n+1}^{\x,\p}),\,(\bp_{n+1}^{\x,\p})^\T f(t_n,\x)\Big\},\,(\bp_{n+1}^{\x,\p})^\T h(t_n,\x)\Big\}\Big].
\end{equation*}
\end{lem}
\begin{proof}
The proof is similar to the one of \cite[Lemma\,3.11]{gruen2012aprobabilistic}.

We fix $(n,\x)\in\{0,\ldots,N\}\times\RR^d$, and define the map $\p\in\{\p_1,\ldots,\p_M\}\mapsto F(\p)$ by 
\begin{equation*} 
F(\p) := \EE\Big[\min\Big\{\max\Big\{\ouh(t_{n+1},\oXnx_{n+1},\bp_{n+1}^{\x,\p}),\,(\bp_{n+1}^{\x,\p})^\T f(t_n,\x)\Big\},\,(\bp_{n+1}^{\x,\p})^\T h(t_n,\x)\Big\}\Big].
\end{equation*}
Assume $(\p)_i>0$ for all $i =1,\ldots,I$. By \Cref{def:feedBack} it holds that
\begin{align*}
\EE\big[F(\bp_{n+1}^{\x,\p})\big] &= \sum_{i=1}^{I}\EE\big[\ind_{\{\bi=i\}}F(\bp_{n+1}^{i,\x,\p})\big]= \sum_{i=1}^{I}\PP[{\bi=i}]\EE\big[F(\bp_{n+1}^{i,\x,\p})\big]
\\
&=\sum_{i=1}^{I}p_i\sum_{\ell = 1}^{I}\frac{(\pinx_\ell(\p))_i}{p_i}F(\pinx_\ell(\p))\lambdanx_\ell(\p) 
\\
&=\sum_{\ell = 1}^{I}\sum_{i=1}^{I}p_i\frac{(\pinx_\ell(\p))_i}{p_i}F(\pinx_\ell(\p))\lambdanx_\ell(\p) 
\\
&=\sum_{\ell = 1}^{I} F(\pinx_\ell(\p))\lambdanx_\ell(\p)\Big(\sum_{i=1}^{I}{(\pinx_\ell(\p))_i}\Big)=\sum_{\ell = 1}^{I} F(\pinx_\ell(\p))\lambdanx_\ell(\p),
\end{align*} 
since $\sum_{i=1}^{I}{(\pinx_\ell(\p))_i} = 1$. On noting \Cref{eq:vexpu}, the statement follows immediately.
\end{proof}

\begin{prop}\label{prop:supvis}
The limit $w$  is a viscosity supersolution of \Cref{eq:hjb} on $[0,T]\times\RR^d\times\DI$.
\end{prop}
\begin{proof}
Let $\varphi:[0,T]\times \RR^d\times\DI\rightarrow \RR$ be a smooth test function, such that $(w - \varphi)$ has a strict global minimum at $(\ot,\ox,\op)$ with $(w - \varphi)(\ot,\ox,\op) = 0$ and such that its derivatives are uniformly Lipschitz continuous in $\p$. We have that $w$ satisfies the viscosity supersolution property \Cref{def:viscosity}(ii), which equivalently consists to show that
\begin{alignat}{3}
\label{eq:ass4}\tag{P4}
&\lambda(\op,\DD_\p^2\varphi(\ot, \ox, \op))\leq 0,&&\mbox{ or }&&
\\
\label{eq:ass5}\tag{P5}
&\varphi(\ot, \ox, \op)\geq \op^\T f(\ot,\ox),&&\mbox{ and }&&
\\
\label{eq:ass6}\tag{P6}
&\varphi(\ot,\ox,\op) < \op^\T h(\ot,\ox) && \mbox{ implies }&& \mL\varphi(\ot,\ox,\op) \geq 0.
\end{alignat}

If \Cref{eq:ass4} holds, then \Cref{eq:hjbsup} follows immediately. Let us assume that 
\begin{equation}
\label{eq:supvis1}
\lambda(\op,\DD_\p^2\varphi(\ot, \ox, \op))> 0.
\end{equation}

By \Cref{eq:fdisc11}, for every $(n,\x,\p)\in\{0,\ldots,N\}\times\RR^d\times\DI$, it holds that $\ouh(t_n,\x,\p)\geq \p^\T f(t_n,\x)$. Thus by continuity, we obtain \Cref{eq:ass5}.

It remains to prove \Cref{eq:ass6}. 
For $\hht_k=T/N_k$, $\hhp_k$, s.t., $\hht_k,\,\hhp_k \rightarrow \infty$ for $k \to\infty$
we consider a sequence $(\ot_k,\ox_k,\op_k)_{k\in\NN}$  with $\ot_k= n_k\hht_k\in \Pi^{\hht_k}$, $n_k\in \NN$, $\op_k\in\mN_{\hhp_k}$ 
such that $(\ot_k,\ox_k,\op_k)\rightarrow (\ot,\ox,\op)$ for $k\to\infty$ and s.t. $\ouh - \varphi$ has a global minimum at $(\ot_k,\ox_k,\op_k)$. 
Define $\varphi^\hh := \varphi + (\ouh - \varphi)(\ot_k,\ox_k,\op_k)$. Then, for all $(\x,\p)\in\RR\times\DI$ we have
\begin{equation} \label{eq:supvis2}
(\ouh - \varphi^\hh)(\ot_k + \hht_k,\x,\p)\geq (\ouh - \varphi^\hh)(\ot_k,\ox_k,\op_k) = 0.
\end{equation}

By \Cref{eq:supvis1}, there exist $\delta,\eta>0$ such that for all $k$ large enough we have
\begin{equation}\label{eq:supvis3}
\forall (t,\x,\p)\in\eta\widebar{\mB},\, \z\in T_{\DI}(\op_k),\, \z^\T\DD_\p^2\varphi(t, \x,\p)\z > 4\delta\lvert 	\z\rvert ^2,
\end{equation}
where $\widebar{\mB}$ is the unit ball of $[0,T]\times\RR^d\times\DI$ centered at $(\ot_k,\ox_k,\op_k)$.

Furthermore, we assume without loss of generality that outside of $\eta\widebar{\mB}$, $\varphi^\hh$ is still convex on $\DI$. Thus for any $(s,\x,\p)\in [\ot_k,T]\times\RR^d\times\DI$ it holds that
\begin{align}
\label{eq:supvis4}
\ouh(s, \x, \p)\geq \varphi^\hh(s,\x,\p)\geq \varphi^\hh(s, \x,\op_k) + (p-\op_k)^\T 	\DD_\p\varphi^\hh(s,\x,\op_k).
\end{align}

To prove \Cref{eq:ass6}, we use also arguments similar to \cite[Chapter 3]{touzi2013optimal}. 
We proceed by contradiction and assume that
{\rev
\begin{equation*}
\varphi(\ot,\ox,\op)<\op^\T h(\ot,\ox)\quad\mbox{and}\quad \mL \varphi(\ot,\ox,\op)<0.
\end{equation*}}
and work toward a contradiction on \Cref{lem:dynamicProg}, which by \Cref{eq:fdisc11} implies for all $(n,\x,\p)\in\{0,\ldots,N\}\times\RR^d\times\{\p_1,\ldots,\p_M\}$ that
\begin{equation} \label{eq:sup32}
\ouh(t_n,\x,\p)\geq\min_{\mu\in\mT^{\ot_k}_{[n_k..N]}}\EE\Big[ \ouh(t_{\mu+1},\oX^{n,\x}_{n+1},\bp_{n_k + 	1})\ind_{\{\mu\geq n_k + 1\}} +\bp_{n}^\T h(t_\mu,\x)\ind_{\{\mu< n_k + 1\}}\Big],
\end{equation}
where $\bp_{n_k+1}\coloneqq\bp_{n_k+1}^{\ox_k,\op_k}$.

By \Cref{eq:supvis2} and the continuity of $\varphi$ and $h$, we can find $\delta>0$ and $k$ large enough (equivalently $\hht_k$ small enough) such that
\begin{equation}\label{eq:sup30}
\varphi^\hh(\ot_k,\ox_k,\op_k) \leq \op_k^\T 	h(\ot_k,\ox_k)-\delta\quad\mbox{and}\quad \mL\varphi^\hh(\ot_k,\ox_k,\op_k)\leq 0,
\end{equation}
for $(\ot_k,\ox_k,\op_k)\in\hht_k\widebar{\mB} $, where $\widebar{\mB} $ is now the unit ball of $[0,T]\times\RR^d\times\DI$ centered at $(\ot,\ox,\op)$.  We choose an arbitrary stopping rule $\mu\in\mT^{\ot_k}_{[n_k..N]}$. By \Cref{eq:supvis2,eq:supvis4}, we have that
\begin{align}
\notag
\EE\big[\ouh&((\ot_k + \hht_k)\wedge t_\mu,\oX^{n_k,\ox_k}_{(n_k +1)\wedge \mu},\bp_{(n_k +1)\wedge \mu})\big]-\ouh 	(\ot_k,\ox_k,\op_k) 
\\
\notag
&\geq \EE\big[\varphi^\hh((\ot_k + \hht_k)\wedge t_\mu,\oX^{n_k,\ox_k}_{(n_k +1)\wedge \mu}, \op_k) \big]-\ouh 	(\ot_k,\ox_k,\op_k) 
\\
\notag
&\qquad+ \EE\big[(\bp_{(n_k +1)\wedge \mu}-\op_k)^\T\big] \EE\big[\DD_\p\varphi^\hh((\ot_k + \hht_k)\wedge 	t_\mu,\oX^{n_k,\ox_k}_{(n_k +1)\wedge \mu}, \op_k)\big]
\\
\label{eq:sup31}
& = \EE\big[\varphi^\hh((\ot_k + \hht_k)\wedge t_\mu,\oX^{n_k,\ox_k}_{(n_k +1)\wedge \mu}, \op_k) \big] -\ouh 	(\ot_k,\ox_k,\op_k).
\end{align}
Applying the Taylor formula to the right-hand side of \Cref{eq:sup31}, we get
\begin{equation}
\label{eq:sup33}
\begin{split}
\EE&\big[\ouh((\ot_k + \hht_k)\wedge t_\mu,\oX^{n_k,\ox_k}_{(n_k +1)\wedge \mu},\op_k)\big] -\ouh 	(\ot_k,\ox_k,\op_k) 
\\
&=\EE\big[(\ouh - \varphi^\hh)((\ot_k + \hht_k)\wedge t_\mu,\oX^{n_k,\ox_k}_{(n_k +1)\wedge \mu},\op_k) 
\\
&\quad+(\dt+\mathcal{L})\varphi^\hh(\ot_k,\ox_k,\op_k)(\hht_k\wedge(t_\mu - \ot_k))\big] + C\hht_k\mO(\hht_k^{1/2}).
\end{split}
\end{equation}

For $k$ large enough (equivalently $\hh$ small enough), $(\ot_k,\ox_k,\op_k)\in\hht_k\widebar{\mB}$. Thus the assertion \Cref{eq:sup30} applies to $\mL\varphi^\hh$. It follows immediately from \Cref{eq:sup33} that 
\begin{align}
\notag
\EE&\big[\ouh((\ot_k + \hht_k)\wedge t_\mu,\oX^{n_k,\ox_k}_{(n_k +1)\wedge \mu},\op_k)\big] -\ouh 	(\ot_k,\ox_k,\op_k) 
\\
\notag
&\geq \EE\big[(\ouh - \varphi^\hh)((\ot_k + \hht_k)\wedge t_\mu,\oX^{n_k,\ox_k}_{(n_k +1)\wedge \mu},\op_k) \big]
\\
\label{eq:sup34}
\begin{split}
&\geq \EE\big[(\ouh - \varphi^\hh)(\ot_k  + \hht_k,\oX^{n_k,\ox_k}_{n_k+ 1},\op_k)\ind_{\{\mu\geq n_k  + 1 	\}} 
+ (\ouh - \varphi^\hh)(t_\mu,\oX^{n_k,\ox_k}_{\mu},\op_k)\ind_{\{\mu<n_k  + 1\}}\big].
\end{split}
\end{align}
Note that if $\mu<n_k  + 1$, then $(t_\mu,\oX^{n_k,\ox_k}_{\mu},\op_k)\in\hht_k\widebar{\mB}$. Therefore, on one hand, by \Cref{eq:supvis2} and the continuity of $(\ouh - \varphi^\hh)$ it implies for the second term in the right-hand side of \Cref{eq:sup34} that $(\ouh - \varphi^\hh)(t_\mu,\oX^{n_k,\ox_k}_{\mu},\op_k) = 0$. Thus we have 
\begin{equation}\label{eq:sup35}
\EE\big[\ouh((\ot_k + \hht_k)\wedge t_\mu,\oX^{n_k,\ox_k}_{(n_k +1)\wedge \mu},\op_k)\big] -\ouh 	(\ot_k,\ox_k,\op_k)\geq \gamma\PP[n\geq n_k  + 1], 
\end{equation}
with 
\begin{equation*}
\gamma= \min_{\hht_k\partial\widebar{\mB}} (\ouh-\varphi^\hh)>0.
\end{equation*}
On the other hand, the assertion \Cref{eq:sup30} applies to $\varphi^\hh$. Thus we have
\begin{align}
\notag
&\EE\big[\ouh ((\ot_k + \hht_k)\wedge t_\mu,\oX^{n_k,\ox_k}_{(n_k +1)\wedge \mu},\bp_{(n_k +1)\wedge \mu})\big] 
\\
\notag
&= \EE\big[\ouh(\ot_k   + \hht_k,\oX^{n_k,\ox_k}_{n_k+ 1},\bp_{n_k +1})\ind_{\{\mu\geq n_k   + 1\}}  + 	\ouh(t_\mu,\oX^{n_k,\ox_k}_{\mu},\bp_{ \mu})\ind_{\{\mu<n_k   + 1\}}\big]
\\
\label{eq:sup36}
\begin{split}
&\leq \EE\big[\ouh(\ot_k   + \hht_k,\oX^{n_k,\ox_k}_{n_k+ 1},\bp_{n_k +1})\ind_{\{\mu\geq n_k   + 1\}} 
+ \bp_{\mu}^\T h(t_\mu,\ox_k)\ind_{\{\mu<n_k   + 1\}}\big] - \delta\PP[\mu< n_k  + 1].
\end{split}
\end{align}
Hence, combining \Cref{eq:sup35,eq:sup36}, we arrive at
\begin{equation*}
\begin{split}
\ouh(\ot_k,\ox_k, \op_k) \leq \EE\big[\ouh(\ot_k   + \hht_k,\oX^{n_k,\ox_k}_{n_k+ 1},\bp_{n_k 	+1})\ind_{\{\mu\geq n_k   + 1\}}
+ \bp_n^\T f(t_\mu,\ox_k)\ind_{\{\mu<n_k   + 1\}}\big] - (\gamma \wedge \delta).
\end{split}
\end{equation*}
Since $\mu$ is arbitrary, the above inequality provides the desired contradiction to \Cref{eq:sup32}.
\end{proof}
{\rev
\section{Regularity properties of the fully-discrete solution}\label{sec_full}
\begin{lem}[Uniform Lipschitz continuity in $\x$ and uniform boundedness]\label{lem:lipx_nn}
	The map $\x\mapsto\ouhkRL(t_n,\x,\p)$ is uniformly Lipschitz continuous, \ie for every $n\in\{0,\ldots,N\}$, it holds that
	\[
	\forall \x,\y\in\RR^d,\, \p\in\DI, \quad \lvert \ouhkRL(t_n, \x, \p) - \ouhkRL(t_n, \y, \p)\rvert\leq  \mLx\lvert \x - \y\rvert,  
	\] 
	where $\mLx := C \max\big\{\llbracket f\rrbracket,\llbracket g\rrbracket,\llbracket h\rrbracket\big\}\ee^ {(\llbracket b\rrbracket + \tfrac{1}{2}\llbracket a\rrbracket^2)T}$
  and $C\geq 1$ for sufficiently small $\Delta t$.
 Moreover,   $\ouhkRL$ is uniformly bounded with
	\begin{equation*}
		\lVert\ouhkRL\rVert_\infty\leq\max\Big\{\beta\mLx \big({\revv M_\rho } + \zeta\big),\,\lVert f\rVert_\infty,\,\lVert h\rVert_\infty\Big\},
	\end{equation*}
	where {\revv $M_{\rho} := \max_{n=0,\ldots,N-1}{\revv\Vert\rho^{s,n+1}_{\lambda-2}\Vert}_\infty$ (with {\revv $\rho^{s,n+1}_{\lambda-2}$} being the activation function at the layer {\revv$\lambda-2$}
 in the GroupSort Neural Network function $\Psi^{n+1}_\kappa$ used to approximate $\ouhkRL(t_{n+1}, \cdot, \cdot)$)
 and the ($n$-independent) constants $\beta$, $\zeta$ are given in \cref{eq:gs_constraints}, \Cref{eq:groupsort}}.
\end{lem}
\begin{proof}
	We fix $\x,\y\in\mD$, $\p\in\DI$ and proceed by induction for $n = N,N-1,\ldots,0$ to show that $\x\mapsto\ouhkRL(t_n,\x,\p)$ is uniformly Lipschitz continuous and uniformly bounded. 
	By \Cref{eq:dnn0} and \ref{A1}, \ref{A2}, the base case $n = N$ holds. 
	We assume that at time level $t_{n+1}$ there exist a constant $\gamma^{n+1}>0$ such that
	\begin{equation}\label{eq:lipx_nn_h1}
		\lvert \ouhkRL(t_{n+1}, \x, \p) - \ouhkRL(t_{n+1}, \y, \p)\rvert\leq  \gamma^{n+1}\lvert \x - \y\rvert,
	\end{equation}
	and that $\|\ouhkRL(t_{n+1},\cdot, \cdot)\|_\infty \leq \max\big\{\lVert f\rVert_\infty,\lVert h\rVert_\infty\big\}$.
	Thanks to \cite[Proposition~2.1]{maximilien2022approximation}, we can find a Groupsort Neural Network $\Psi^{n+1}_\kk\in \mGGn$ verifying for all $\x\in [-R,R]^d$ and $\p\in \DI$
	\begin{equation*}
		\vert \ouhkRL(t_{n+1}, \x, \p)  - \Psi^{n+1}_\kk(\x;\theta_{n+1}(\p))\vert\leq 2R\gamma^{n+1}\varepsilon.
	\end{equation*}
	Note that $\Psi^{n+1}_\kk\in \mGGn$ means that $\x\mapsto \Psi^{n+1}_\kk(\x;\theta_{n+1}(\p))$ is Lipschitz continuous with a Lipschitz constant $\gamma^{n+1}$.
	
	\textit{Uniform Lipschitz continuity in $\x$}. Suppose that $0\leq  \ouhkRL(t_n,\y,\p)-\ouhkRL(t_n,\x,\p)$. By \Cref{eq:fdisc5en} and the convexity of $\p\mapsto \ouhkRL(t_n,\y,\p)$, we have
	\begin{align}\notag
		\ouhkRL(t_n,\y,\p)-\ouhkRL(t_n,\x,\p) 
		&= \ouhkRL(t_n,\y,\p) - \sum_{i=1}^{I}\ouhkRL(t_n,\x,\tpinxk_i(\p))\tlambdanxk_i(\p)
		\\
		\label{eq:lipx_nn_1}
		&\leq \sum_{i=1}^{I}\big(\ouhkRL(t_n,\y,\tpinxk_i(\p)) - \ouhkRL(t_n,\x,\tpinxk_i(\p))\big)\tlambdanxk_i(\p).
	\end{align}
	The next step consists to derive an estimate for the summands in the right-hand side of \Cref{eq:lipx_nn_1}. 
	Note that $\{\tpinxk_1(\p),\ldots,\tpinxk_I(\p)\}\subset\{\p_1,\ldots,\p_M\}$ and the map $F\mapsto \vexp F$ is nonexpansive. 
	By \Cref{eq:fnn3}, it follows that
	\begin{equation}
		\label{eq:lipx_nn_2}
		\begin{split}
			&\ouhkRL(t_n,\y,\tpinxk_i(\p)) - \ouhkRL(t_n,\x,\tpinxk_i(\p))
			\\
			&\leq \max\Big\{\lvert\oY^{n,\y,\pi}_\kappa - \oY^{n,\x,\pi}_\kappa\rvert:{\pi\in \{\tpinxk_1(\p),\ldots,\tpinxk_I(\p)\}}\Big\}.
		\end{split}
	\end{equation}
	Let us derive an estimate for the term $\lvert\oY^{n,\y,\pi}_\kappa - \oY^{n,\x,\pi}_\kappa\rvert$ appearing in the right-hand side of \Cref{eq:lipx_nn_2}. 
	Using \Cref{eq:fnn2} we obtain
	\begin{align}
		\notag
		\begin{split}
			\lvert\oY^{n,\y,\pi}_\kappa - \oY^{n,\x,\pi}_\kappa\rvert
			&\leq \Big\lvert \min\Big\{\max\Big\{\EE
			\big[\Psi^{n+1}_\kk(\oXny_{n+1};\theta_{n}(\pi))\big],\,\pi^\T f(t_n,\y)\Big\},\,\pi^\T h(t_n,\y)\Big\}
			\\
			&\hspace{20pt}- \min\Big\{\max\Big\{\EE
			\big[\Psi^{n+1}_\kk(\oXnx_{n+1};\theta_{n}(\pi))\big],\,\pi^\T f(t_n,\x)\Big\},\,\pi^\T h(t_n,\x)\Big\}\Big\rvert
		\end{split}
		\\
		\notag
		\begin{split}
			&\leq \max\Big\{\EE\big\lvert\Psi^{n+1}_\kk(\oXny_{n+1};\theta_{n}(\pi))-\Psi^{n+1}_\kk(\oXnx_{n+1};\theta_{n}(\pi))\big\rvert,
			\\
			&\hspace{80pt}\lvert f(t_n,\y) - f(t_n,\x)\rvert,\,\lvert h(t_n,\y) - h(t_n,\x)\rvert\Big\}.
		\end{split}
	\end{align}
	By the hypothesis \cref{eq:lipx_nn_h1}, $\x\mapsto \ouhkRL(t_{n+1}, \x, \p) $ is Lipschitz continuous with a Lipschitz constant $\gamma^{n+1}$.
	Thus, using \ref{A1} and \Cref{prop:lipTransit}, we obtain
	\begin{equation}\label{eq:lipx_nn_3}
		\lvert\oY^{n,\y,\pi}_\kappa - \oY^{n,\x,\pi}_\kappa\rvert \leq\max\big\{\gamma^{n+1} Q_\hht,\,\llbracket f\rrbracket,\,\llbracket h\rrbracket\big\}\lvert \x-\y\rvert
	\end{equation}
	We insert \Cref{eq:lipx_nn_3} into the right-hand side of \Cref{eq:lipx_nn_2}, which gives that
	\begin{equation*}
		\begin{split}
			\ouhkRL(t_n,\y,\pinx_\ell(\p)) &- \ouhkRL(t_n,\x,\pinx_\ell(\p))\leq \max\big\{\gamma^{n+1} Q_\hht,\,\llbracket f\rrbracket,\,\llbracket h\rrbracket\big\}\lvert \x-\y\rvert.
		\end{split}
	\end{equation*}
	If $0\leq  \ouhkRL(t_n,\x,\p)-\ouhkRL(t_n,\y,\p)$, we commute the role of $\x$ and $\y$ in the previous steps to get
	\begin{equation*}
		\begin{split}
			\lvert \ouhkRL(t_n,\y,p) &- \ouhkRL(t_n,\x,\p)\rvert\leq \max\big\{\gamma^{n+1} 	Q_\hht,\,\llbracket f\rrbracket,\,\llbracket h\rrbracket\big\}\lvert \x-\y\rvert,
		\end{split}
	\end{equation*}
	which implies that the map $\x\mapsto \ouhkRL(t_n,\x,\p)$ is Lipschitz continuous with a Lipschitz coefficient $\gamma^n= \max\big\{\gamma^{n+1} Q_\hht,\,\llbracket f\rrbracket,\,\llbracket h\rrbracket\big\}$. We obtain recursively that
	\begin{align}
		\notag
		\gamma^n
		&\leq \max\big\{\gamma^{n+1},\,\llbracket f\rrbracket,\,\llbracket h\rrbracket\big\}Q_\hht
		\\
		\notag
		&\leq \max\big\{\gamma^{n+2},\,\llbracket f\rrbracket,\,\llbracket h\rrbracket\big\} Q_\hht^2
		\\
		\notag
		&\;\;\vdots
		\\
		\label{eq:lipx_nn_4}
		&\leq \max\big\{\gamma^N,\,\llbracket f\rrbracket,\,\llbracket h\rrbracket\big\}Q_\hht^{N-n}\leq C \max\big\{\llbracket f\rrbracket,\,\llbracket g\rrbracket,\,\llbracket h\rrbracket\big\}\ee^ {(\llbracket b\rrbracket + \tfrac{1}{2}\llbracket a\rrbracket^2)T}\eqqcolon \gamma,
	\end{align}
{\rev with $C\geq1$ for sufficiently small $\Delta t$.}
	Hence, the map $\x\mapsto\ouhkRL(t_n,\x,\p)$ is uniformly Lipschitz continuous with a Lipschitz coefficient $\gamma = \mLx$ defined in \Cref{eq:lip6}
and noting \Cref{lem:lipx} we conclude that the Lipschitz constants of $\x\mapsto\ouhkRL(t_n,\x,\p)$ and $\x\mapsto\ouh(t_n,\x,\p)$ coincide.
	
	\textit{Uniform boundedness.} It follows by \cref{eq:fnn2}, \cref{eq:fnn3}, and \cref{eq:fdisc5en}  that
	\begin{align*}
		\lvert\ouhkRL(t_n,\x,\p)\rvert&\leq \sum_{i = 1}^{I}\Big\lvert \min\Big\{\max\Big\{\EE
		\big[\Psi^{n+1}_\kk(\oX^{\x,n}_{n+1};\theta_{n}(\tpinxk_i))\big],
		\\
		&\hspace{40pt}(\tpinxk_i(\p))^\T f(t_n,\x)\Big\},\,(\tpinxk_i(\p))^\T h(t_n,\x)\Big\}\Big\rvert \tlambdanxk_i(\p)
		\\
		&\leq\max\big\{\Vert\Psi^{n+1}_\kk(\;\cdot\;;\theta_{n+1})\Vert_\infty,\,\lVert f\rVert_\infty,\,\lVert h\rVert_\infty\big\}.
	\end{align*}
	Noting that $\Psi^{n+1}_\kk(\;\cdot\;;\theta_{n+1})\in \mGGn$ and \Cref{eq:groupsort} we estimate
	\begin{equation*}
		\Vert\Psi^{n+1}_\kk(\;\cdot\;;\theta_{n+1})\Vert_\infty\leq \gamma^{n+1}\beta \Vert\Phi^{n+1}_\kk(\;\cdot\;;\theta_{n+1})\Vert_\infty,
	\end{equation*}
	where $\Phi_\kappa(\;\cdot\;;\theta_{n+1})\in \mSSn$.
  Recalling the representation of $\Phi_\kappa$, cf. \Cref{eq:groupsort0}, we can further estimate
	\begin{equation*}\revv
		\Vert\Psi^{n+1}_\kk(\;\cdot\;;\theta_{n+1})\Vert_\infty \leq \gamma^{n+1}\beta \big(\vert w^{(\lambda-1)}\vert_{\infty}\Vert\rho^{s,n+1}_{\lambda-2}\Vert_\infty + \vert b^{(\lambda-1)}\vert_{\infty}\big).
	\end{equation*}
	The first part of the proof implies that $ \gamma^{n+1} = \gamma$ and by \cref{eq:gs_constraints} we get that
	\begin{equation*}\revv
		\Vert\Psi^{n+1}_\kk(\;\cdot\;;\theta_{n+1})\Vert_\infty\leq \gamma\beta \big(\Vert\rho^{s,n+1}_{\lambda-2}\Vert_\infty + \zeta\big).
	\end{equation*}
	Consequently, we conclude that
	\begin{equation*}\revv
		\lvert\ouhkRL(t_n,\x,\p)\rvert\leq \max\big\{\gamma\beta \big(\Vert\rho^{s,n+1}_{\lambda-2}\Vert_\infty + \zeta\big),\,\lVert f\rVert_\infty,\,\lVert h\rVert_\infty\big\},
	\end{equation*}
and the desired estimate follows after taking the maximum over $n$ and $\x$, $\p$ in the above expression.
\end{proof}
}
\section{Numerical results}
\label{sec:num_res}
In the numerical experiments below we take $d = 1$, $T = 1$, $\mD = [0,1]$. 
We use the Quickhull algorithm for the computation of the discrete convex envelope in all experiments below.
{\revv For the ease of implementation we compute the fully-discrete solution $\ouhkRL$ using a Feedforward Neural Network approximation instead of the GroupSort network function.
For comparison we also compute a solution $\tuhRL$ using  a semi-Lagrangian scheme.}

\subsection*{Numerical experiment 1} 
In this experiment we determine experimental convergence rates of the numerical approximation for the linear PDE
\begin{equation}
\label{eq:pde_test1}
\partial_t u +  a^2\frac12\frac{\partial^2 u}{\partial x^2} + H = 0,\mbox{ in } [0,T]\times \mD,
\end{equation}
with a terminal condition $g(\x)= u(T, \x)$, where $a(\x)= 0.2 \x(1-\x)$,
and $H(t, \x)= 3\pi\sin(3\pi t)\cos(3\pi \x)+\frac12(3\pi a(\x))^2\cos(3\pi t)\cos(3\pi \x)$.
For the given data, \Cref{eq:pde_test1} has the solution $u(t,\x)= \cos(3\pi t)\cos(3\pi \x)$.

Due to the choice of the diffusion $a$, we may restrict the spatial domain to the interval $[0,1]$, which is partitioned into uniform line segments $\mT^{\hhx}=\{(\x_{\ell-1}, \x_{\ell})\}_{\ell=1}^L$, $\x_\ell = \ell\hhx$ with the mesh size $\hh\x = 1/{L}$. 
The time interval $[0,T]$ is partitioned with time-step $\hht= 1/N$. 
We consider the Euler approximation of the diffusion process associated with the \eqref{eq:pde_test1}:
\begin{equation*}
\oX^{n,\x_\ell}_{n+1} = \x_\ell+ a(\x_\ell)\xi^n_\ell\sqrt{\hht} \qquad n=0,\ldots,N,\,\,\ell=0,\ldots,L,
\end{equation*}
where $\{\xi^n_\ell\}$ are i.i.d random variables that take values $\{-1, 1\}$ with equal probability (i.e., the expectation $\mathbb{E}[\oX^{n,\x_\ell}_{n+1}]$ is an average of the two possibilities $\xi^n_\ell=\pm 1$).

For the computations we employ a Feedforward Neural Network (FFN) algorithm \Cref{algo:ffn}  and a semi-Lagrangian (SL) algorithm \Cref{algo:sl} (cf., e.g. \cite{banas2020numerical}) with piecewise linear interpolation in space.
We use a Feedforward Neural Network that contains one hidden layer with $10$ neurons, 
take ${\tt tanh}$ as the activation function for the hidden layers and choose identity function as activation function for the output layer.
Furthremore, we use the Levenberg--Marquardt algorithm in \Cref{algo:ffn} to solve the least-square problem at each time-step. 

\begin{algo}[FFN - Experiment 1]\label{algo:ffn}
Set $\ouhkRL(t_N, \x_\ell)= g(\x_\ell)$ for $\ell=0,\dots, L$ and proceed for $n = N-1,\ldots,0$ as follows:
\begin{enumerate}[label = $\circ$]
\item Compute:
\begin{equation*}
\theta_n=\argmin_{\theta\in\Theta_\kk}\frac{1}{L+1}\sum_{\ell=0}^L\big\lvert \ouhkRL(t_{n+1},\x_\ell) - \Psi_\kk(\x_\ell;\theta)\big\rvert^2,
\end{equation*}
\item For $\ell = 0,\dots,L$ set:
\begin{equation*}
\ouhkRL(t_n, \x_\ell) = \mathbb{E}[ \Psi_\kk(\oX^{n,\x_\ell}_{n+1};\theta_n)]  + \hht H(t_{n}, \x_\ell).
\end{equation*}
\end{enumerate}
\end{algo}

\begin{algo}[SL algorithm - Experiment 1]\label{algo:sl}
Set $\tuhRL(t_N, \x)= g(\x)$ for $\x\in \{\x_0,\ldots,\x_L\}$ and proceed for $k = N-1,\ldots,0$ as follows:
\begin{enumerate}[label = $\circ$]
\item Define $\x\mapsto\tuhRL(t_{n+1},\x)$ as the piecewise linear interpolant of $\big\{\tuhRL(t_{n+1},\x_\ell)\big\}_{\ell=0}^L$, i.e. for $\x\in[\x_{\ell}, \x_{\ell-1}]$:
\[
\tilde{u}^\triangle(t_{n+1}, \x) = \tuhRL(t_{n+1}, \x_{\ell-1} ) + \frac{(\tuhRL(t_{n+1}, \x_{\ell} ) - \tuhRL(t_{n+1}, \x_{\ell-1} ))(\x - \x_{\ell-1})}{\hhx}.
\]
\item For $\ell=0,\dots,L$ set:
\begin{equation*}
\tuhRL(t_n, \x_\ell) = \mathbb{E}[\tilde{u}^\triangle(\oX^{n,\x_\ell}_{n+1})]   + \hht H(t_{n}, \x_\ell).
\end{equation*}
\end{enumerate}
\end{algo}

We  display the respective solutions computed by \Cref{algo:ffn} and \Cref{algo:sl} in \Cref{fig:utx} computed for $\hhx = \hht =  1.56\times 10^{-2}$.
With $\hhx = \hht =  1.56\times 10^{-2}$, both numerical solutions are graphically similar, with noticeable differences at extremal points, i.e., around $\x = 0$, between $(0.6,0.7)$, and around $\x = 1$.
We notice similar differences at $t = 0, 0.3$.
\begin{figure}[t!]
\subfloat{\includegraphics[scale = 0.25]{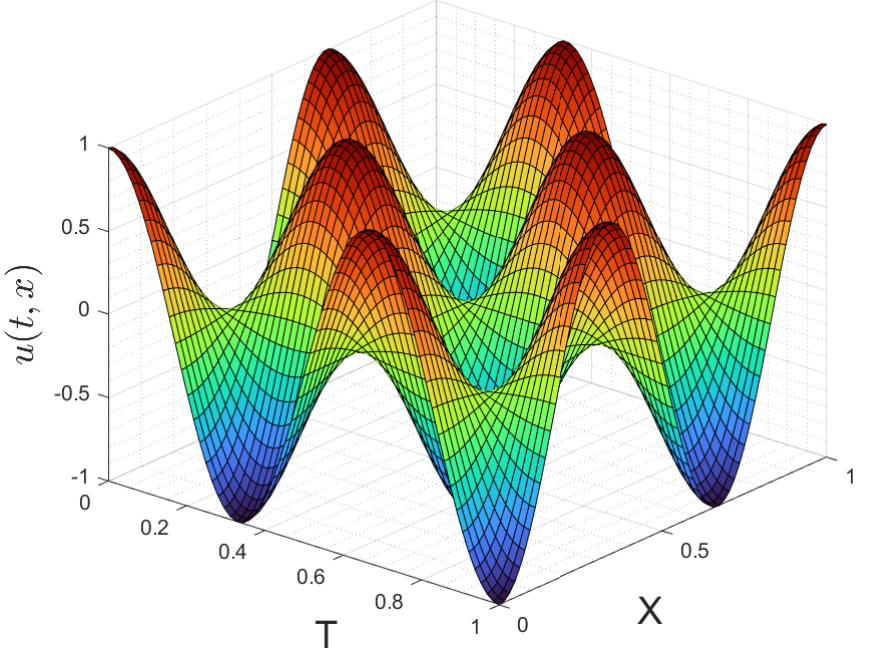}}
\subfloat{\includegraphics[scale = 0.25]{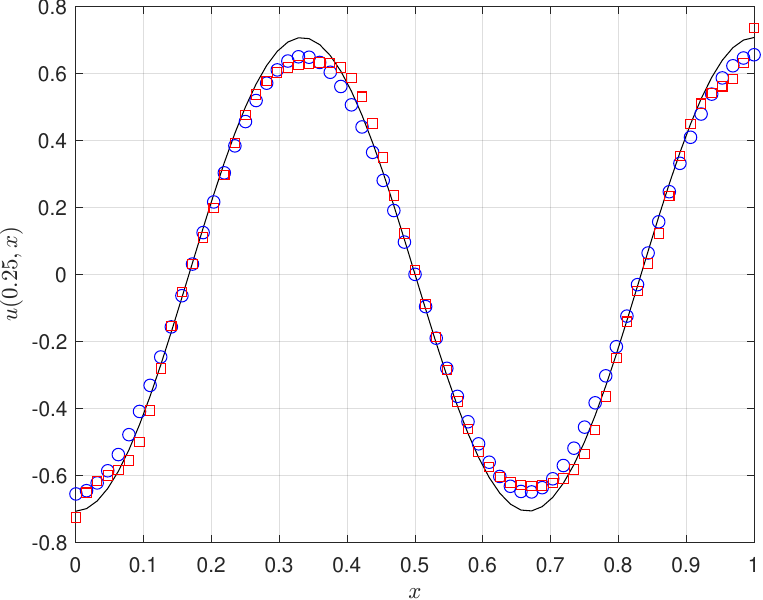}}
\subfloat{\includegraphics[scale = 0.25]{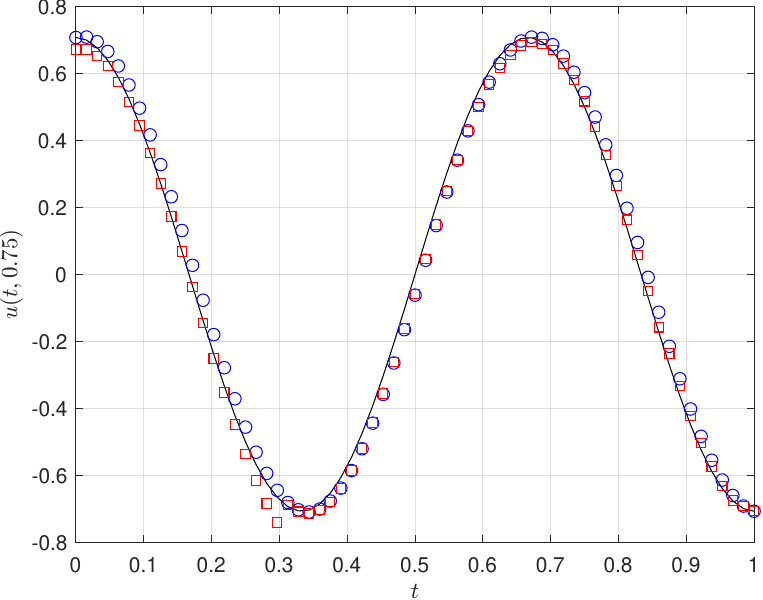}}
\caption{Left: plot of the function $u(t,\x)= \cos(3\pi t)\cos(3\pi x)$.
Middle: Graph of the solution at $t = 0.25$.
Right: time evolution of the solution at $\x = 0.75$. 
(($-$) exact solution, ({\rev  $\Box$}) neural network approximation, ({\color{blue} $\bigcirc$}) SL algorithm.}
\label{fig:utx}
\end{figure}

In \Cref{tab:convergence1} we examine convergence order of the error of respective algorithms;
we measure the error in the maximum norm over the discrete space-time grid (MAX) and in the root mean square norm (RMS) for decreasing discretization parameters $\hhx$ and $ \hht$;
we observe first order of convergence with respect to both discretization parameters.
\begin{table}[b!]
\caption{Experiment 1 -- convergence rate.}\label{tab:convergence1}
{
\centering
\resizebox{.75\textwidth}{!}{
\begin{tabular}{c|c|cc|cc|cc|cc|cc|cc|}
&\multirow{2}{*}{$\triangle$}  & \multicolumn{4}{c|}{FFN} & \multicolumn{4}{c|}{SL-Scheme} 
\\
\hhline{~|~|----|----|}
&&&&&&&&&
\\[-1em]
&&   MAX-error & rate & RMS-error&  rate& MAX-error&  rate & RMS-error &  rate
\\ 
\hline
\hline
&&&&&&&&&
\\[-1em]
\multirow{5}{*}{\rotatebox{90}{Conv. $t$}}
&1.56$\times 10^{-2}$& 7.47$\times 10^{-2}$& -- &3.20$\times 10^{-2}$& -- &6.21$\times 10^{-2}$  & -- &  4.31$\times 10^{-2}$  & -- 
\\
&7.81$\times 10^{-3}$ & 2.76$\times 10^{-2}$& 1.44 & 1.75$\times 10^{-2}$& 0.87 & 2.87$\times 10^{-2}$  & 1.12 &  1.99$\times 10^{-2}$  & 1.12
\\
&3.91$\times 10^{-3}$ & 1.31$\times 10^{-2}$& 1.07 &9.03$\times 10^{-3}$& 0.95 & 1.64$\times 10^{-2}$ & 0.80 & 1.15$\times 10^{-2}$  &  0.79
\\
&1.95$\times 10^{-3}$ &6.67$\times 10^{-3}$& 0.97 &4.44$\times 10^{-3}$& 1.02 & 7.58$\times 10^{-3}$ & 1.11 &  5.26$\times 10^{-3}$  & 1.12
\\
&9.77$\times 10^{-4}$ &4.17$\times 10^{-3}$& 0.68 &2.38$\times 10^{-3}$& 0.90 &3.78$\times 10^{-3}$  & 1.00 &  2.62$\times 10^{-3}$  & 1.00
\\
\hline
\hline
&&&&&&&&&
\\[-1em]
\multirow{5}{*}{\rotatebox{90}{Conv. $x$}}
&1.56$\times 10^{-2}$& 9.31$\times 10^{-2}$& -- &4.26$\times 10^{-2}$& -- &5.67$\times 10^{-2}$ & -- &   4.07$\times 10^{-2}$ & --
\\
&7.81$\times 10^{-3}$ & 2.28$\times 10^{-2}$& 1.44 & 1.54$\times 10^{-2}$& 1.46 & 3.09$\times 10^{-2}$ & 0.88 &  2.14$\times 10^{-2}$  & 0.93
\\
&3.91$\times 10^{-3}$ & 1.32$\times 10^{-2}$& 1.07 &8.33$\times 10^{-3}$& 0.89 &1.54$\times 10^{-2}$  & 1.00 &  1.07$\times 10^{-2}$ & 1.00
\\
&1.95$\times 10^{-3}$ &1.18$\times 10^{-2}$& 0.97 &4.75$\times 10^{-3}$& 0.81 & 8.16$\times 10^{-3}$ & 0.91 &  5.47$\times 10^{-3}$  & 0.96
\\
&9.77$\times 10^{-4}$ &1.29$\times 10^{-2}$& 0.68 &2.82$\times 10^{-3}$& 0.75 &  4.09$\times 10^{-3}$& 1.00 &  2.70$\times 10^{-3}$  & 1.02
\\
\hline
\end{tabular}
}
}
\end{table}

\subsection*{Numerical experiment 2} 
In this experiment we consider a problem with incomplete information. We take $I=2$
and eliminate one probability variable from the solution by parametrizing $\Delta(2) = (\p, 1-\p)$ for $p\in [0,1]$. Hence, we consider the generalized obstacle problem for the transformed solution
\begin{equation}
\label{eq:pde_test2}
\max\Big\{\partial_t u +  a^2\frac12\frac{\partial^2 u}{\partial x^2} + H,\, -\lambda(\cdot,\DD^2_\p u)\Big\} = 0 \mbox{ in } [0,T]\times \mD\times[0,1],
\end{equation}
with the terminal condition $u(T, \x)= 0 $; we take $H(t, \x, \p)= \sin(\pi t)\cos(\pi \x)\sin(3\pi \p)$ and $a(\x)= \x(1-\x)$.

As in the previous experiments we restrict the spatial domain to the interval $[0,1]$ and
approximate \eqref{eq:pde_test2} on a space-time grid with (uniform) mesh sizes $\hhx=1/L$, $\hht=1/N$
and define the discrete process $\oX^{n,\x_\ell}_{n+1}$ analogously as in the previous experiment.
The approximation in the probability variable $\p\in [0,1]$ is constructed on a uniform mesh with mesh size $\hhp=1/M$.
We perform the computations using a Feedforward Neural Network (FFN)  \Cref{algo:exp2_ffn} and a semi-Lagrangian (SL)  \Cref{algo:exp2_sl} which employs piecewise linear interpolation in the spatial variable.

In \Cref{fig:u_test2a} we display the  solution 
computed with $\hht = \hhx = \hhp =  1.56\times 10^{-2}$ at time $t=0$.
To illustrate the effect of the convexity constraint we also display in \Cref{fig:u_test2a} the (nonconvex) solution computed with \Cref{algo:exp2_sl} without the convexification step
(this corresponds to neglecting the obstacle term $\lambda(\cdot,\DD^2_\p u)$ in \eqref{eq:pde_test2} and solving the corresponding unconstrained counterpart of \eqref{eq:pde_test2}).
In \Cref{fig:u_test2b} we display the two dimensional profiles of the numerical solution crossing the point $(\overline{t}, \overline{\x}, \overline{\p}) = (0,25, 0.75, 0.25)$ in the direction of each variable and 
we observe good agreement between the respective algorithms.
In \Cref{tab:exp2_convrate} we display the convergence of the errors decreasing discretization parameters.
{\rev Since to explicit solution to \Cref{eq:pde_test3} is known, we determine the order of convergence using a reference solution computed by \Cref{algo:exp2_ffn} with $\hht=\hhx=\hhp = 4.88\times 10^{-4}$.
} We observe linear rate of convergence with respect to the discretization in all variables.

\begin{algo}[FFN algorithm - Experiment 2]\label{algo:exp2_ffn}
For $\ell = 0,\ldots,L$, $m=0,\dots, M$ set $\ouhkRL(t_N, \x_\ell, p_m)= 0$ and proceed for $n = N-1,\ldots,0$ as follows:
\begin{enumerate}[label = $\circ$]
\item For $m=1,\dots,M$ compute:
\begin{equation*}
\theta_n(\p_m)=\argmin_{\theta\in\Theta_n^\gamma}\frac{1}{L+1}\sum_{\ell = 0}^{L}\big\lvert \ouhkRL(t_{n+1},\x_\ell,\p_m) - \Psi_\kk(\x_\ell;\theta(\p))\big\rvert^2.
\end{equation*}
\item For $\ell = 0,\dots,L$, $m=0,\dots, M$ set:
\begin{equation*}
y(t_n, \x_\ell,\p_m) = \mathbb{E}\big[\Psi_\kk(\oX^{n,\x_\ell}_{n+1};\theta_n(\p_m))\big]   + \hht H(t_{n}, \x_\ell,\p_m).
\end{equation*}
\item For $\ell = 0,\dots,L$, $m=0,\dots, M$ set:
\[
\ouhkRL(t_n, \x_\ell, \p_m) = \vexp\big[y(t_n, \x_\ell, \p_0),\ldots,y(t_n, \x_\ell, \p_M)\big](\p_m).
\]
\end{enumerate}
\end{algo}

\begin{algo}[SL algorithm - Experiment 2]\label{algo:exp2_sl}
For $\ell = 0,\ldots,L$, $m=0,\dots, M$ set $\tuhRL(t_N, \x_\ell, p_m)= 0$ and proceed for $n = N-1,\ldots,0$ as follows:
\begin{enumerate}[label = $\circ$]
\item For $m=1,\dots,M$ define the map $\x\mapsto\tilde{u}^\triangle(t_{n+1},\x,\p_m)$ as the piecewise linear interpolant of $\big\{\tuhRL(t_{n+1},\x_\ell,\p_m)\big\}_{\ell=0}^L$, i.e., for $\x\in[\x_{\ell}, \x_{\ell-1}]$ set
\[
\tilde{u}^\triangle(t_{n+1}, \x, \p_m) = \tuhRL(t_{n+1}, \x_{\ell-1},\p ) + \frac{(\tuhRL(t_{n+1}, \x_{\ell}, \p ) - \tuhRL(t_{n+1}, \x_{\ell-1}, \p))(\x - \x_{\ell-1})}{\hhx}.
\]
\item For $\ell = 0,\dots,L$, $m=0,\dots, M$ set:
\begin{equation*}
y(t_n, \x_\ell,\p_m) = \mathbb{E}[\tilde{u}^\triangle(t_{n+1},\oX^{n,\x_\ell}_{n+1},\p_m)]   + \hht H(t_{n}, \x_\ell,\p_m),
\end{equation*}
\item For $\ell = 0,\dots,L$, $m=0,\dots, M$ set:
\[
\tuhRL(t_n, \x_\ell, \p_m) = \vexp\big[y(t_n, \x_\ell, \p_0),\ldots,y(t_n, \x_\ell, \p_M)\big](\p_m).
\]
\end{enumerate}
\end{algo}

\begin{figure}[t!]
\subfloat{\includegraphics[scale = 0.2]{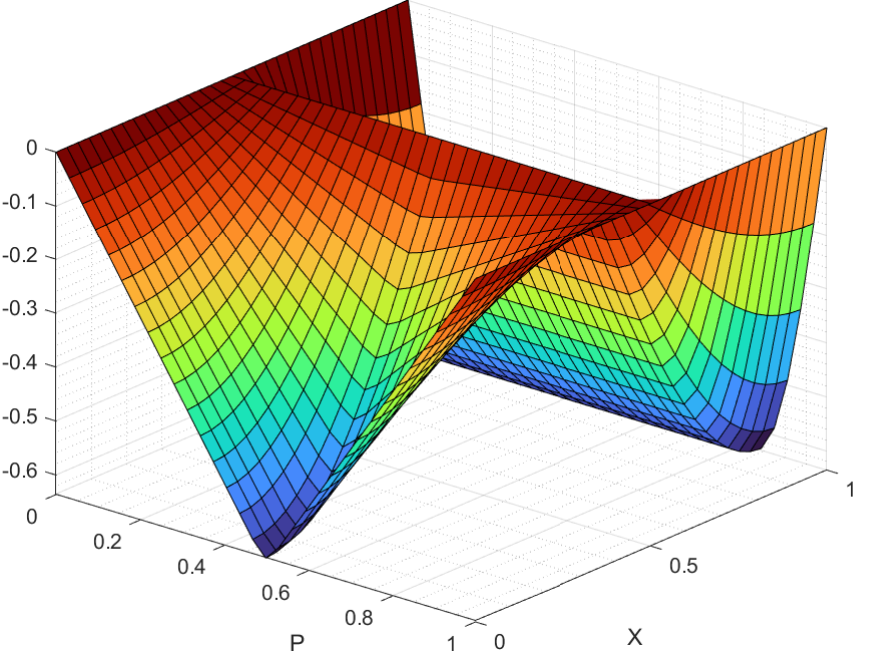}}
\quad
\subfloat{\includegraphics[scale = 0.2]{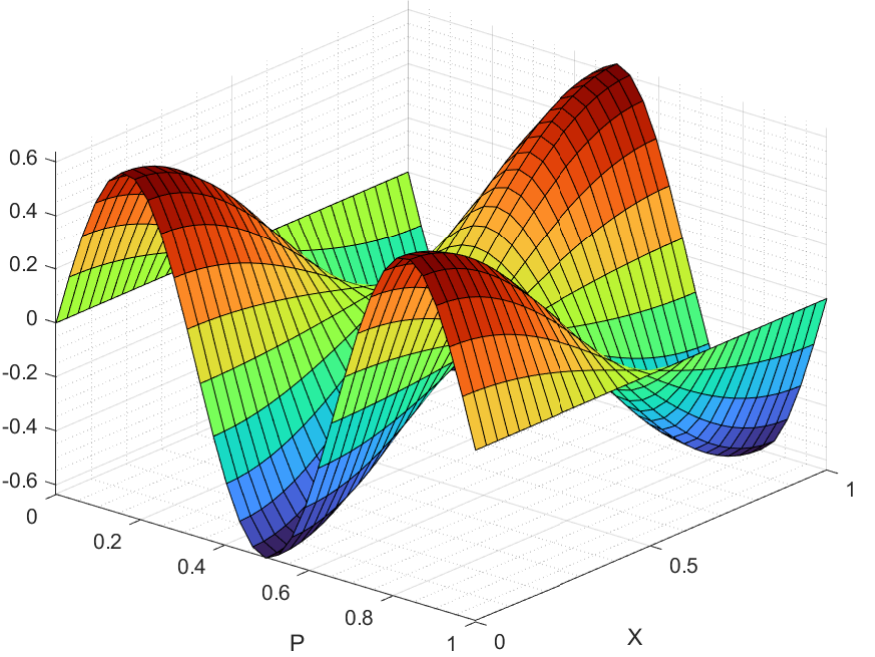}}
\caption{Solution at $t = 0$ computed with the SL algorithm (left) and solution computed without the obstacle term $\lambda(\p,\DD^2_\p u)$ (right).}
\label{fig:u_test2a}
\end{figure}

\begin{figure}[t!]
\subfloat{\includegraphics[scale = 0.2]{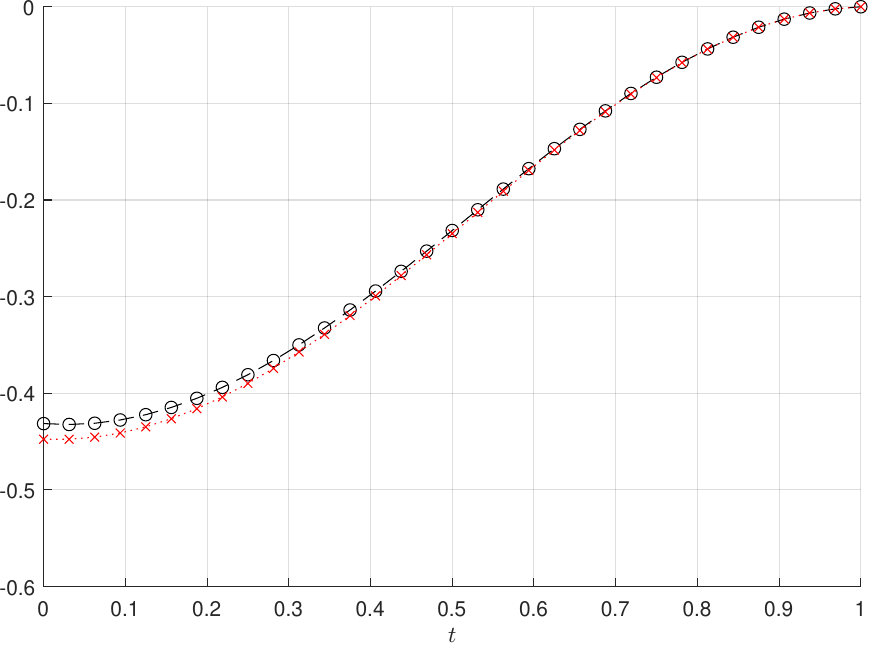}}
\subfloat{\includegraphics[scale = 0.2]{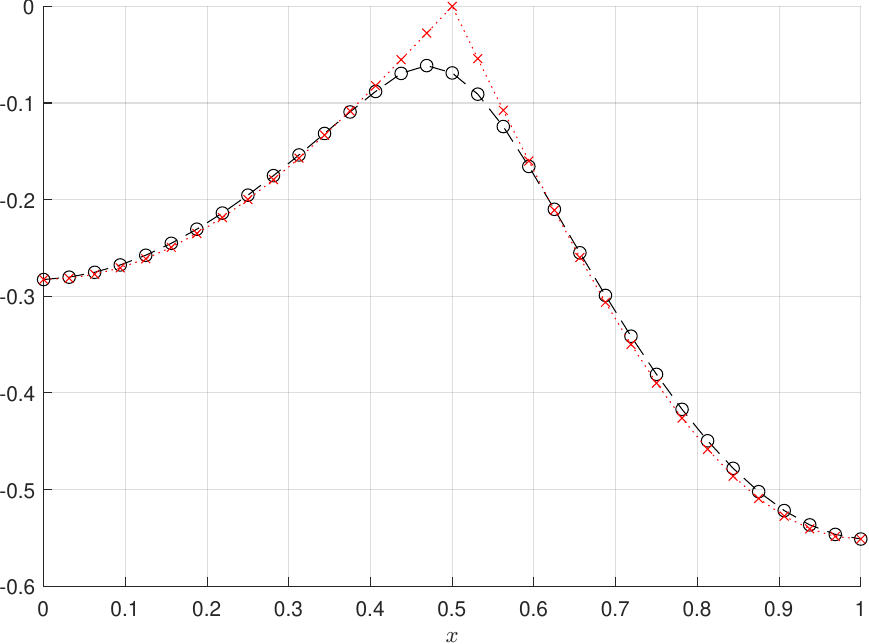}}
\subfloat{\includegraphics[scale = 0.2]{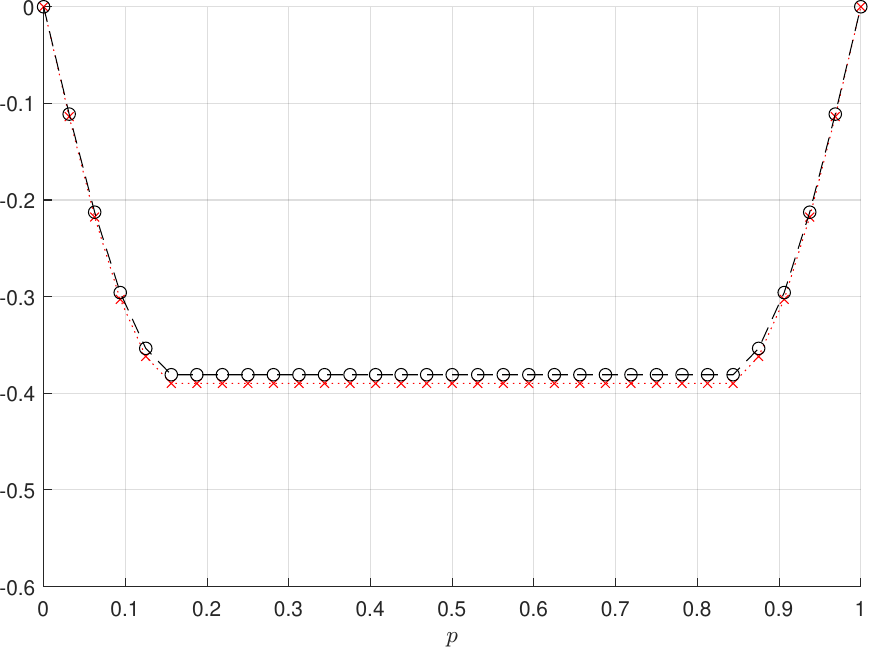}}
\caption{Profiles of the solution computed using the FFN algorithm ({\rev $ +$}) and the SL algorithm ($\circ$) at $(\x,\p) = (0.75, 0.25)$ (left), $(t,\p) = (0.25, 0.25)$ (middle), $(t,\x) = (0.25, 0.75)$ (right)}
\label{fig:u_test2b}
\end{figure}

\begin{table}[b!]
\caption{Experiment 2 -- convergence rate.}\label{tab:exp2_convrate}
{
\centering
\resizebox{.7\textwidth}{!}{
\begin{tabular}{c|c|cc|cc|cc|cc|}
&\multirow{2}{*}{$\triangle$}  & \multicolumn{4}{c|}{FFN} & \multicolumn{4}{c|}{SL-Scheme} 
\\
\hhline{~|~|----|----|}
&&&&&&&&&   
\\[-1em]
&&   MAX-error & rate & RMS-error&  rate& MAX-error&  rate & RMS-error &  rate
\\ 
\hline
\hline
&&&&&&&&&   
\\[-1em]
\multirow{5}{*}{\rotatebox{90}{Conv. $t$}}
&3.13$\times 10^{-2}$& 9.63$\times 10^{-3}$& -- &6.51$\times 10^{-3}$& -- &9.47$\times 10^{-3}$  & -- &  6.05$\times 10^{-3}$  & -- 
\\
&1.56$\times 10^{-2}$ & 5.04$\times 10^{-3}$& 0.94 & 3.48$\times 10^{-3}$& 0.90 & 4.60$\times 10^{-3}$  & 1.04 &  3.10$\times 10^{-3}$  & 0.96
\\
&7.81$\times 10^{-3}$ & 2.51$\times 10^{-3}$& 1.00 & 1.75$\times 10^{-3}$& 0.99 & 2.43$\times 10^{-3}$ & 0.92 & 1.65$\times 10^{-3}$  &  0.91
\\
&3.91$\times 10^{-3}$ &1.19$\times 10^{-3}$& 1.08 & 8.35$\times 10^{-4}$& 1.07 & 1.06$\times 10^{-3}$ & 1.20 &  7.13$\times 10^{-4}$  & 1.21
\\
&1.95$\times 10^{-3}$ &5.13$\times 10^{-4}$& 1.21 & 3.61$\times 10^{-4}$& 1.20 &  4.86$\times 10^{-4}$ & 1.21  & 3.29$\times 10^{-4}$ & 1.11
\\
\hline
\hline
&&&&&&&&&   
\\[-1em]
\multirow{5}{*}{\rotatebox{90}{Conv. $x$}}
&3.13$\times 10^{-2}$& 7.77$\times 10^{-3}$& -- & 4.56$\times 10^{-3}$& -- &1.13$\times 10^{-2}$  & -- &  4.59$\times 10^{-3}$  & -- 
\\
&1.56$\times 10^{-2}$ & 4.58$\times 10^{-3}$& 0.76 & 2.51$\times 10^{-3}$& 0.86 & 4.57$\times 10^{-3}$  & 1.30 &  1.46$\times 10^{-3}$  & 1.65
\\
&7.81$\times 10^{-3}$ & 2.40$\times 10^{-3}$& 0.93 & 1.37$\times 10^{-3}$& 0.87 & 3.33$\times 10^{-3}$ & 0.46 & 1.43$\times 10^{-3}$  &  0.03
\\
&3.91$\times 10^{-3}$ & 1.61$\times 10^{-3}$& 1.05 & 6.58$\times 10^{-4}$& 1.06 & 1.16$\times 10^{-3}$ & 1.52 &  4.11$\times 10^{-4}$  & 1.80
\\
&1.95$\times 10^{-3}$ & 5.07$\times 10^{-4}$& 1.20 & 2.86$\times 10^{-4}$& 1.20  & 5.05$\times 10^{-4}$ &  1.20  & 2.24$\times 10^{-4}$ & 0.88
\\
\hline
\hline
&&&&&&&&&   
\\[-1em]
\multirow{5}{*}{\rotatebox{90}{Conv. $p$}}
&3.13$\times 10^{-2}$& 7.35$\times 10^{-3}$& -- & 5.35$\times 10^{-3}$& -- &3.34$\times 10^{-3}$  & -- &  1.84$\times 10^{-3}$  & -- 
\\
&1.56$\times 10^{-2}$ & 3.69$\times 10^{-3}$& 0.99 & 3.02$\times 10^{-3}$& 0.82 & 2.10$\times 10^{-3}$  & 0.67 &  1.59$\times 10^{-3}$  & 0.21
\\
&7.81$\times 10^{-3}$ & 1.81$\times 10^{-3}$& 1.03 & 1.56$\times 10^{-3}$& 0.95 & 1.52$\times 10^{-3}$ & 0.47 & 1.30$\times 10^{-3}$  &  0.30
\\
&3.91$\times 10^{-3}$ & 8.50$\times 10^{-4}$& 1.09 & 7.53$\times 10^{-4}$& 1.05 & 4.18$\times 10^{-4}$ & 1.85 &  3.61$\times 10^{-4}$  & 1.85
\\
&1.95$\times 10^{-3}$ &3.65$\times 10^{-4}$& 1.22 & 3.28$\times 10^{-4}$& 1.20  & 2.69$\times 10^{-4}$ &  0.64 & 2.41$\times 10^{-4}$ & 0.58
\\
\hline
\end{tabular}
}
}
\end{table}

\subsection*{Numerical experiment 3} 
We consider the full problem \eqref{eq:hjb} with $I=2$.
As in the previous experiment we eliminate one probability variable from the solution by parametrizing $\Delta(2) = (\p, 1-\p)$ for $\p\in (0,1)$ and consider the following obstacle problem
\begin{equation}
\label{eq:pde_test3}
\begin{split}
\max\Big\{\max\Big\{\min\Big\{&\partial_t u +  a^2\frac12\frac{\partial^2 u}{\partial x^2} +   b\frac{\partial u}{\partial x},\,
\\
&u - \p f_1 - (1- \p)f_2\Big\},\,
u - \p h_1 - (1- \p)h_2\Big\},\, -\lambda(\cdot,\DD^2_\p u)\Big\}\Big\} = 0,
\end{split}
\end{equation}
in $[0,T]\times\RR\times [0,1]$ with a terminal condition $u(T, \x,\p)= \p^\T g(\x) $, $g:=(g_1, g_2)$ with $g_i(\x) = \tilde{g}(\x) :=  \max\{2 - \ee^\x,\, 0\}$, $i=1,2$.
The obstacles are chosen as $h_i := g_i(\x) +\delta_i$, $i = 1, 2$ with $\delta_1 = 1.25\times 10^{-1}$, $\delta_2 = 6.5\times 10^{-2}$ and $f_i(\x) = g_i(\x)$, $i = 1, 2$, i.e., the lower obstacle in \cref{eq:pde_test3} does not depend on $\p$.
Furthermore, we take $b= (r-a^2/2)$,  $T=1$, $a = 2\times 10^{-1}$, $r = 3\times 10^{-2}$.
The above problem setup is motivated by the Dynkin game problem (without asymmetric information), considered in \cite{KuhnKyprianou}, which arises in the pricing of Israeli $\delta$-penalty put options. 
In this example, asymmetric information arises because the holder of the option does not know exactly how much the writer of the option is going to pay her upon termination of the contract, but she has only knowledge on the a priori probability according to which the additional payment $\delta_i$ is initially chosen. 
Our numerical study provides the structure of the game's state space and depict the equilibrium stopping boundaries, which is of importance for understanding the financial meaning of the equilibrium.

The discretization is performed as in the previous experiment.
We restrict the spatial domain to the interval $[0,1]$ and partition the solution domain $[0,T]\times[0,1]\times [0,1]$ uniformly in each variable with the respective mesh sizes $\hht = 1/N$, $\hh\x = 1/L$, $\hhp = 1/M$.
In the algorithm below we consider the following Euler approximation of the diffusion process associated with the problem \eqref{eq:pde_test3}:
\begin{equation*}
\oX^{n,\x_\ell}_{n+1} = \x_\ell+ b{\hht} + a\xi^n_\ell\sqrt{\hht} \qquad n=0,\ldots,N,\,\,\ell=0,\ldots,L,
\end{equation*}
where $\{\xi^n_\ell\}$ are i.i.d random variables that take values $\{-1, 1\}$ with equal probability.
The computations in this section are performed using the freedforward neural network algorithm \Cref{algo:exp3_ffn} where
we employ a fully connected Feedforward Neural Network that consists of one hidden layer with  $50$ neurons.
We choose ${\tt tanh}$, i.e.
\begin{equation*}
	{\tt tanh}:\x\mapsto
	\frac{\ee^{\x}- \ee^{-\x}}{\ee^{\x} + \ee^{-\x}}.
\end{equation*} as the activation function for the hidden layers and identity function as activation function for the output layer.
We use the  limited-memory (BFGS) quasi-Newton algorithm \cite{gill1989a,liu1989on} to solve the nonlinear least-squares problem at each time-step of the algorithm.

\begin{algo}[FFN - Experiment 3]\label{algo:exp3_ffn}
For $\ell = 0,\ldots,L$, $m=1,\dots,M$, we initialize $\ouhkRL(t_N, \x_\ell, \p_m)=  \p_m^\T g(\x_\ell)$ and proceed for $n = N-1,\ldots,0$  as follows:
\begin{enumerate}[label = $\circ$]
\item For $m=1,\dots,M$ compute:
\begin{equation*}
\theta_n(\p_m)=\argmin_{\theta\in\Theta_n^\gamma}\frac{1}{L+1}\sum_{\ell = 0}^{L}\big\lvert \ouhkRL(t_{n+1},\x_\ell,\p_m) -  \Psi_\kk(\x_\ell;\theta)\big\rvert^2.
\end{equation*}
\item For $\ell = 0,\dots,L$, $m=0,\dots, M$ set:
\begin{equation*}
y(t_n, \x_\ell, \p_m) = \mathbb{E}[\Psi_\kk(\oX^{n,\x_\ell}_{n+1};\theta_n(\p_m))].
\end{equation*}
\item For $\ell = 0,\dots,L$, $m=0,\dots, M$ set:
\begin{equation*}
\begin{split}
\ovhkRL(t_n, \x_\ell, \p_m) = \min\big\{\max\big\{y(t_n, \x_\ell, \p_m),\, &\p_m f_1(\x_\ell) + (1- \p_m)f_2(\x_\ell)\big\},\, 
\p_m h_1(\x_\ell) + (1- \p_m)h_2(\x_\ell)\big\}.
\end{split}
\end{equation*}
\item For $\ell = 0,\dots,L$, $m=0,\dots, M$ set:
\[
\ouhkRL(t_n, \x_\ell, \p_m) = \vexp\big[\ovhkRL(t_n,\x_\ell,\p_0),\ldots,\ovhkRL(t_n,\x_\ell,\p_M)\big](\p_m).
\]
\end{enumerate}
\end{algo}

To illustrate the effect of the convexity constraint in \cref{eq:pde_test3} we compute the solution
$v= v(t, \x, \p)$ for $\p=0,0.5,1$ of the non-constrained problem
\begin{equation}
\label{eq:pde_test3_nonconvex}
\begin{split}
\max\Big\{\min\Big\{&\partial_t v +  a^2\frac12\frac{\partial^2 v}{\partial x^2} +   b\frac{\partial v}{\partial x},\,
v - \p f_1 - (1- \p)f_2\Big\},\,
v - \p h_1 - (1- \p)h_2\Big\}\Big\}  = 0.
\end{split}
\end{equation}
In \Cref{fig:up1p2}  we compare the solution $\ovhkRL$ of \cref{eq:pde_test3_nonconvex} to the solution $\ouhkRL$ of \cref{eq:pde_test3}.
In the top row in \Cref{fig:up1p2} we plot the difference $\frac12(\ovhkRL(t,\x,0) + \ovhkRL(t,\x,1)) - \ovhkRL(t,\x,0.5)$ which takes negative values since  $\ovhkRL$ is not necessarily convex in the $\p$ variable,
the difference $\frac12(\ouhkRL(t,\x,0) + \ouhkRL(t,\x,1)) - \ouhkRL(t,\x,0.5)$ in the bottom row remains positive since $\ouhkRL$ is convex in $\p$.
\begin{figure}[t!]
\subfloat{\includegraphics[scale = 0.2]{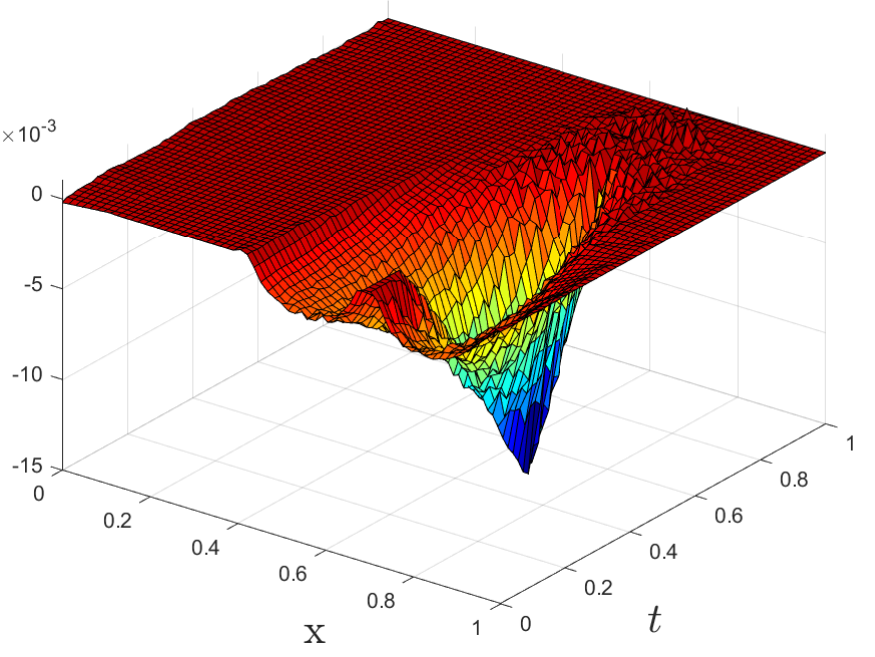}}
\quad
\subfloat{\includegraphics[scale = 0.2]{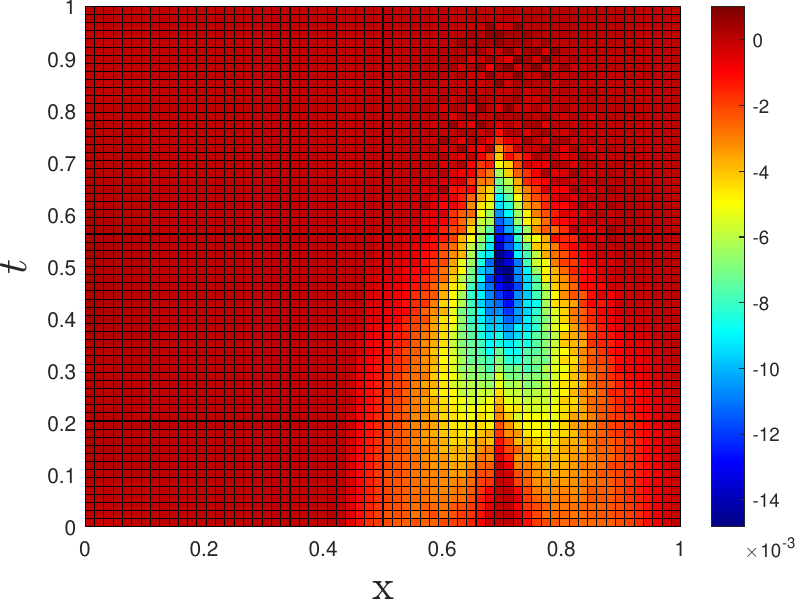}}
\quad\quad
\subfloat{\includegraphics[scale = 0.2]{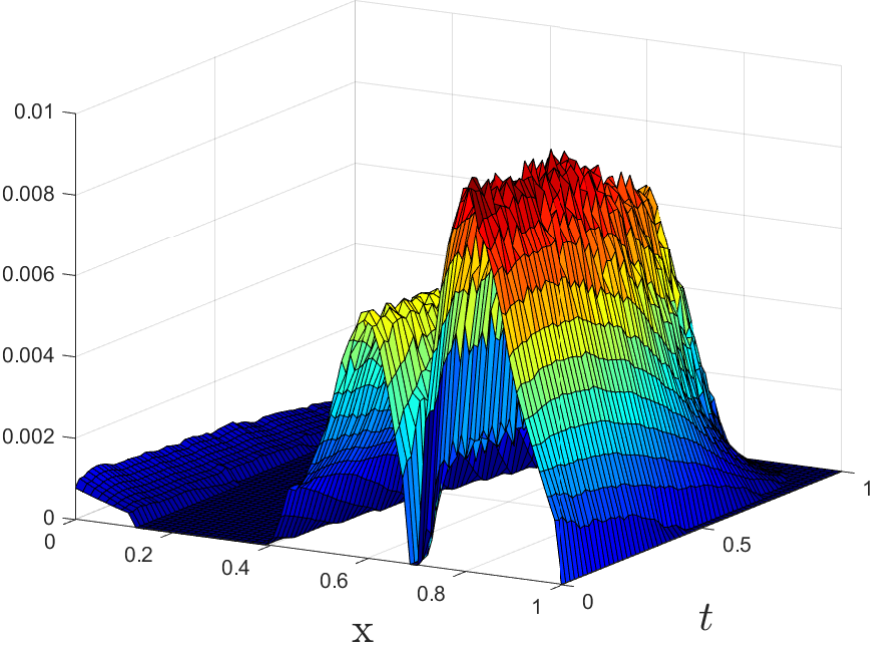}}
\quad
\subfloat{\includegraphics[scale = 0.2]{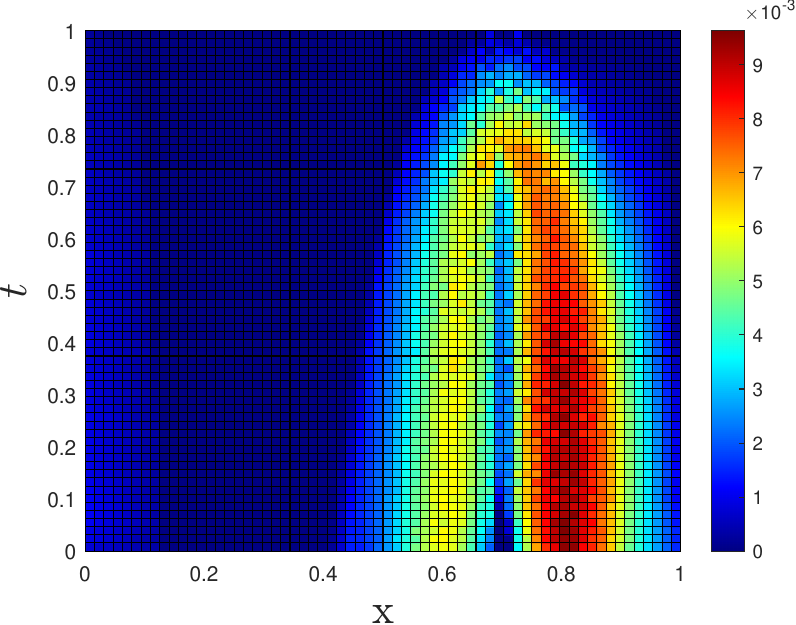}}
\caption{From left to right:
plot of $\tfrac12(\ovhkRL(t,\x,0) + \ovhkRL(t,\x,1)) - \ovhkRL(t,\x,0.5)$ (3d and top view) and of $\tfrac12(\ouhkRL(t,\x,0) + \ouhkRL(t,\x,1)) - \ouhkRL(t,\x,0.5)$ (3d and top view).}
\label{fig:up1p2}
\end{figure}

In \Cref{fig:freeBoundary}  we display space-time graphs of the numerical solution computed with $\hht = \hhx =\hhp = 1.56\times 10^{-2}$ for different values of $p$ along with the regions where the two obstacles are active; the red color represents the active region of the lower obstacle
$\{(t,\x); \,\, \ouhkRL(t,\x,\p) = \p f_1(\x) + (1- \p)f_2(\x) = \tilde{g}(\x)\}$ and the green color represents the active region of the upper obstacle $\{(t,\x);\,\,\ouhkRL(t,\x,\p) = \p h_1(\x) + (1- \p)h_2(\x)\}$.
In all figures the regions were marked as active if the distance of the numerical solution to the obstacle at the nodes $\x_\ell$ was below a tolerance $2\times 10^{-5}$.
\begin{figure}[hbtpt!]
	\subfloat{\includegraphics[width = 0.24\textwidth]{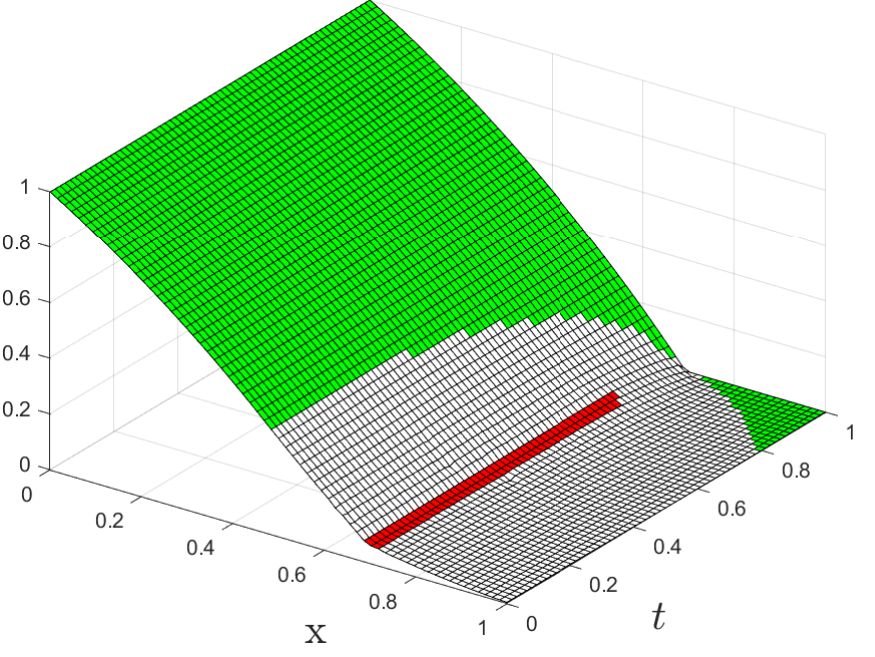}}
	\subfloat{\includegraphics[width = 0.24\textwidth]{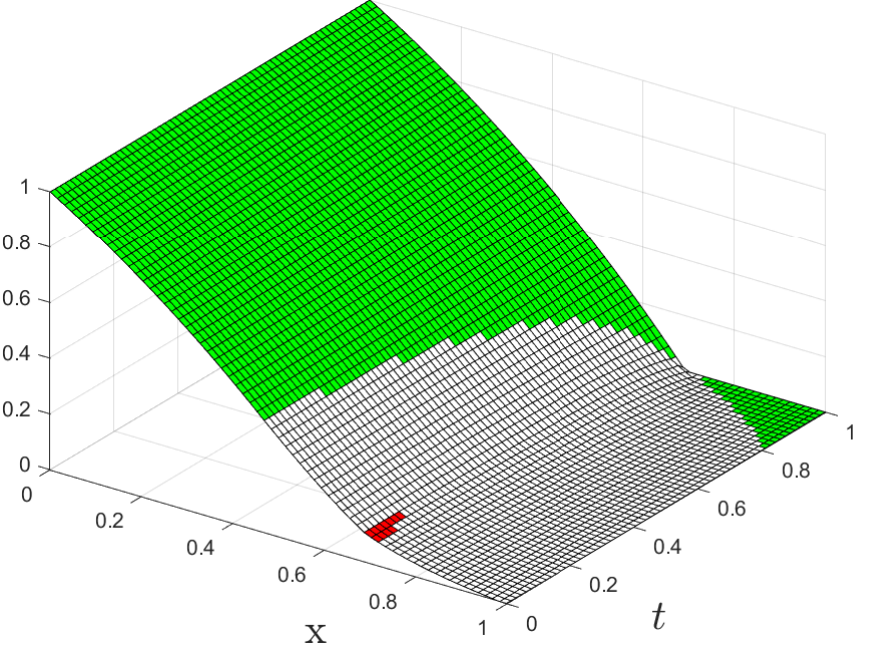}}
	\subfloat{\includegraphics[width = 0.24\textwidth]{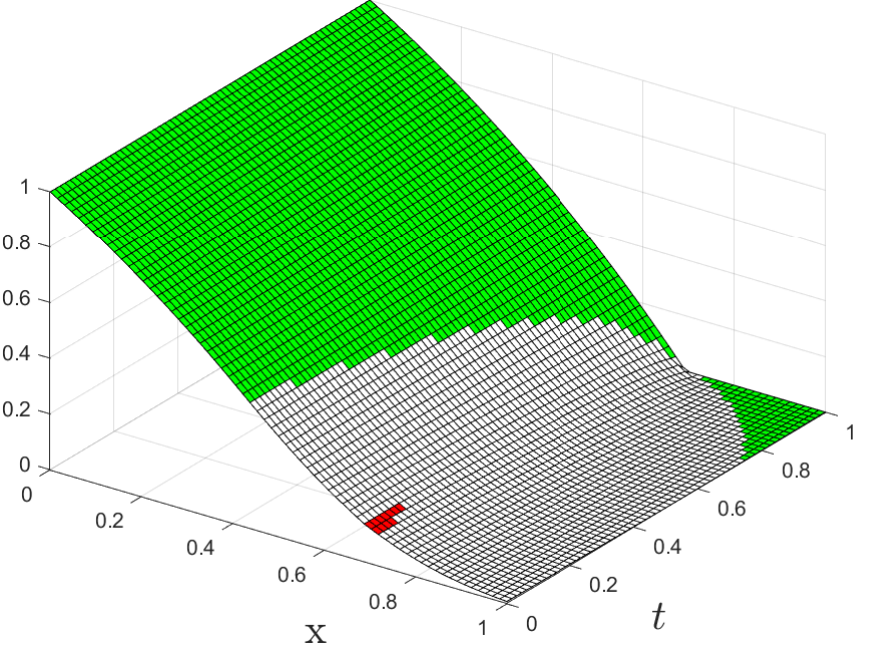}}
	\subfloat{\includegraphics[width = 0.24\textwidth]{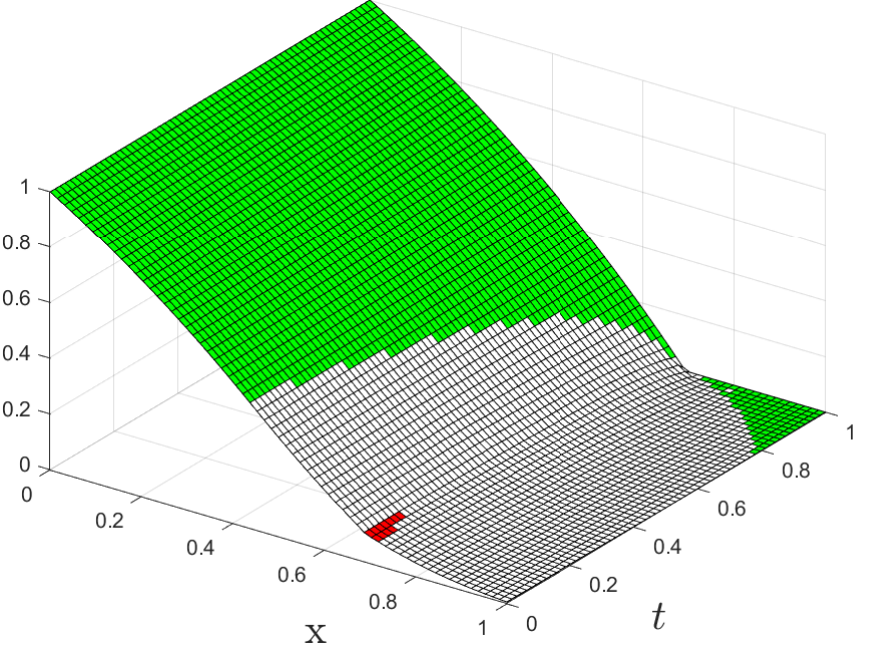}}
\\
	\subfloat{\includegraphics[width = 0.24\textwidth]{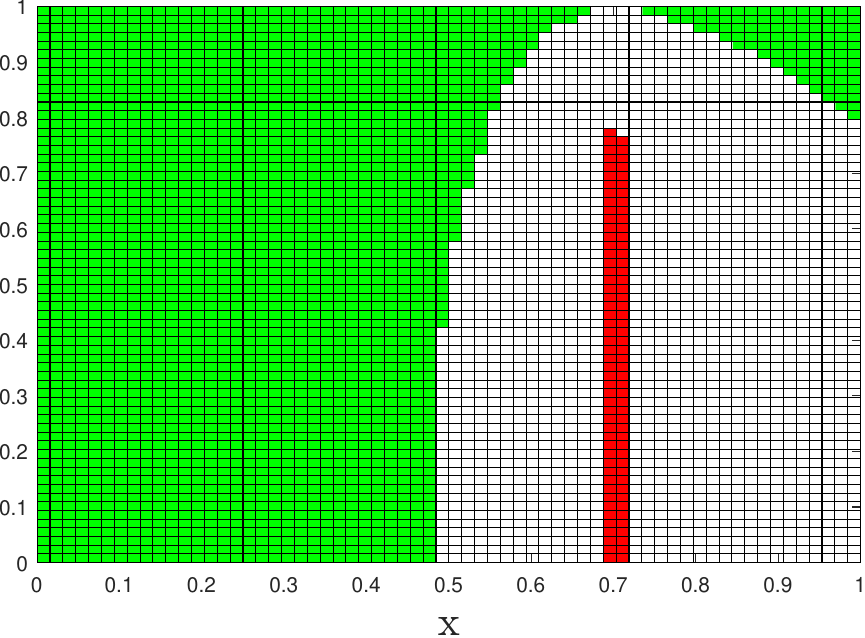}}
	\subfloat{\includegraphics[width = 0.24\textwidth]{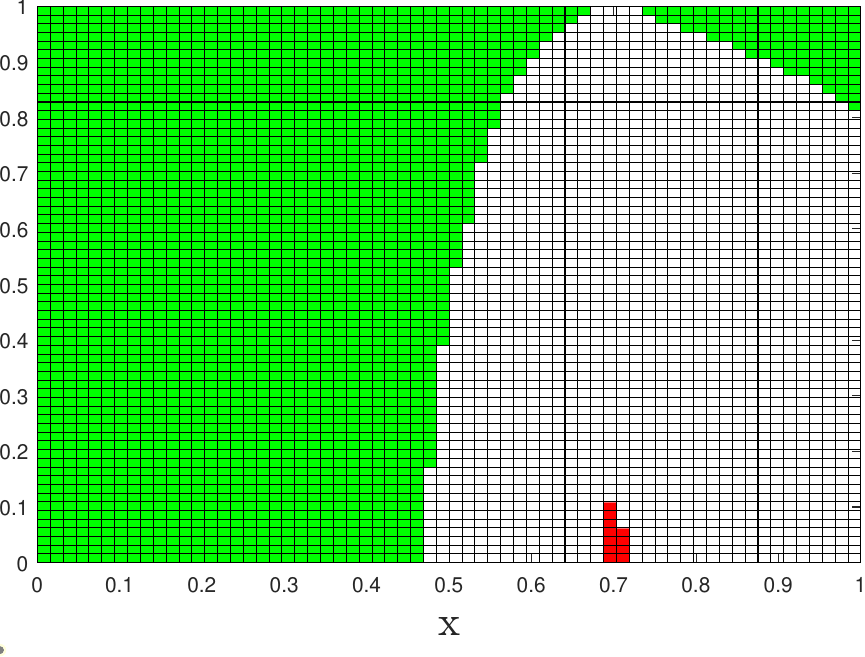}}
	\subfloat{\includegraphics[width = 0.24\textwidth]{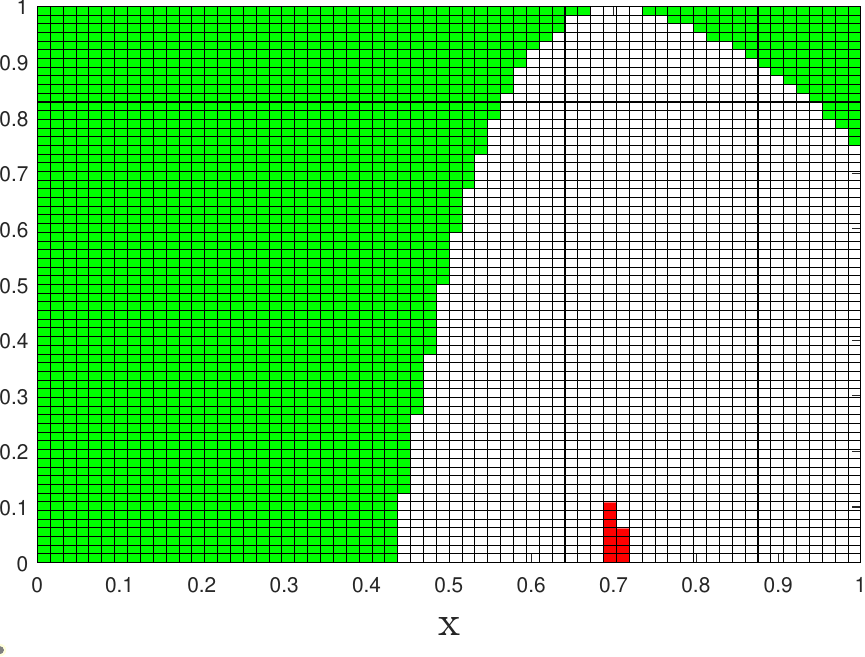}}
	\subfloat{\includegraphics[width = 0.24\textwidth]{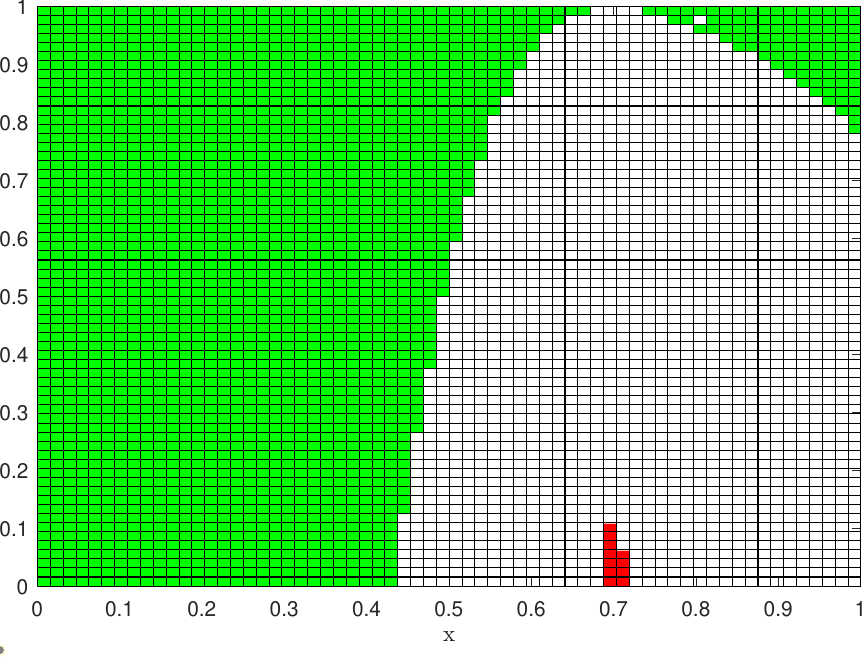}}
	\caption{Results for the BR solver: 3d view (top row) and top view (bottom row) of the numerical solution $(t,\x) \rightarrow \ouhkRL(t,\x,\p)$  (from left to right) for $p=0,0.5,1$
		and of the average $(t,\x)\rightarrow \frac12 (\ouhkRL(t,\x,0) + \ouhkRL(t,\x,1))$ with obstacles respectively represented by  $\frac12 f_1(\x) + \frac12f_2(\x)$, $\frac12 h_1(\x) + \frac12h_2(\x)$.
	}
	\label{fig:test4_freeBoundary}
\end{figure}

\begin{figure}[hbtpt!]
\subfloat{\includegraphics[width = 0.24\textwidth]{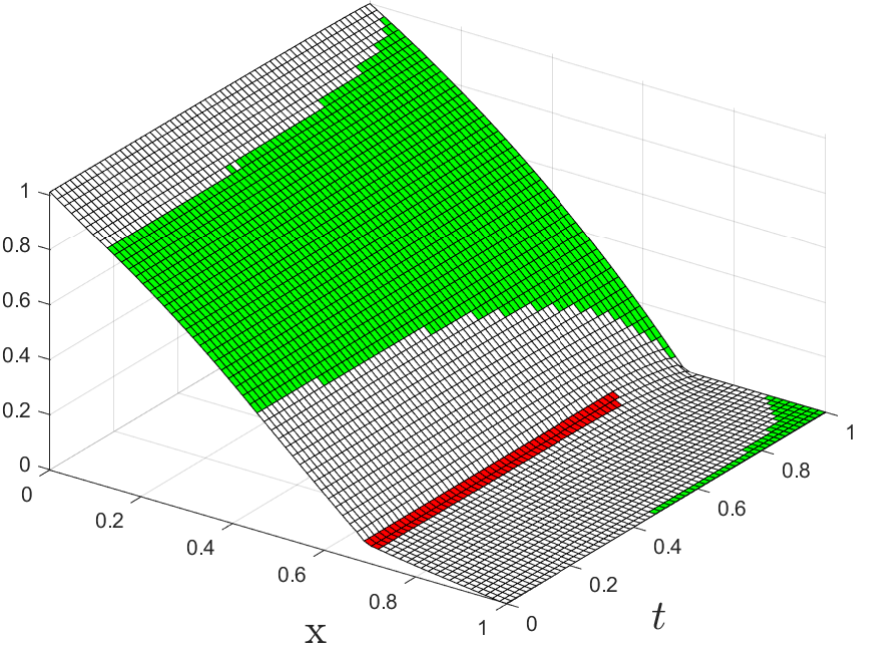}}
\subfloat{\includegraphics[width = 0.24\textwidth]{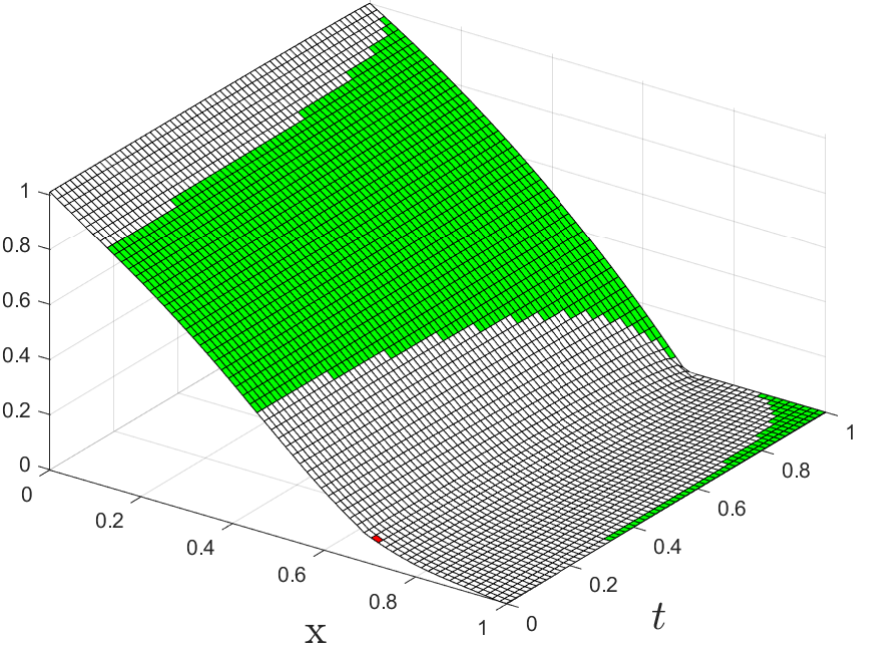}}
\subfloat{\includegraphics[width = 0.24\textwidth]{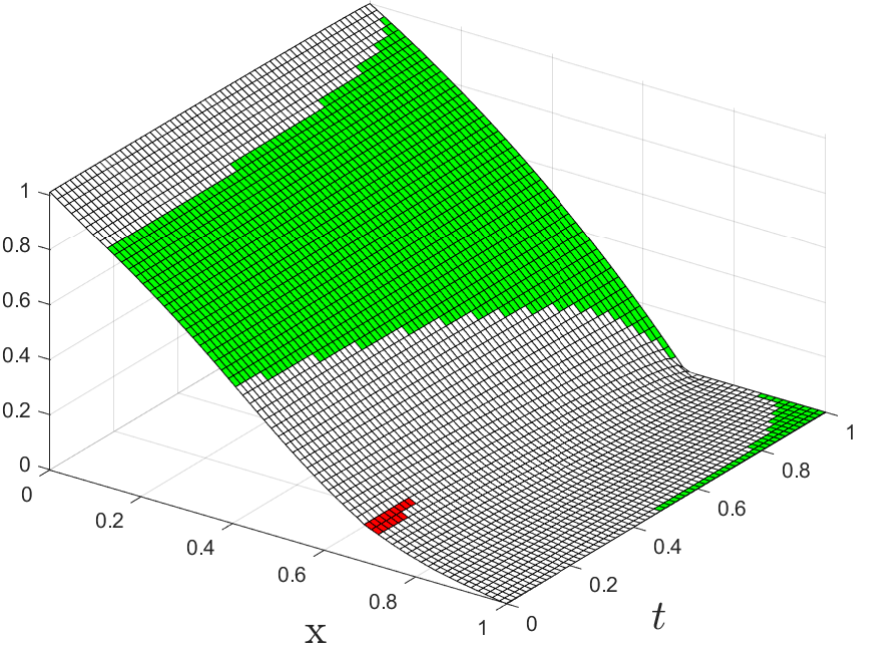}}
\subfloat{\includegraphics[width = 0.24\textwidth]{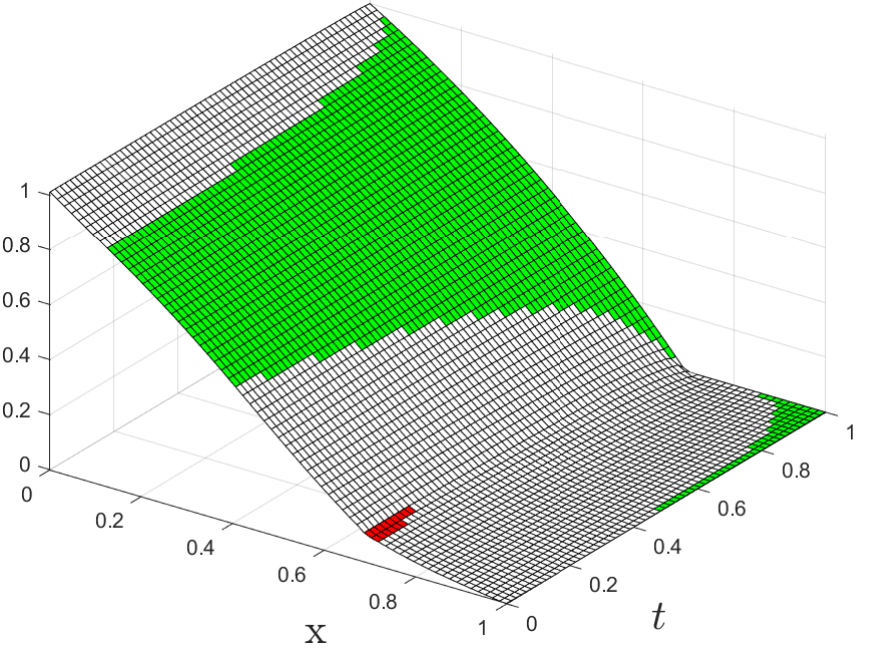}}
\\
\subfloat{\includegraphics[width = 0.24\textwidth]{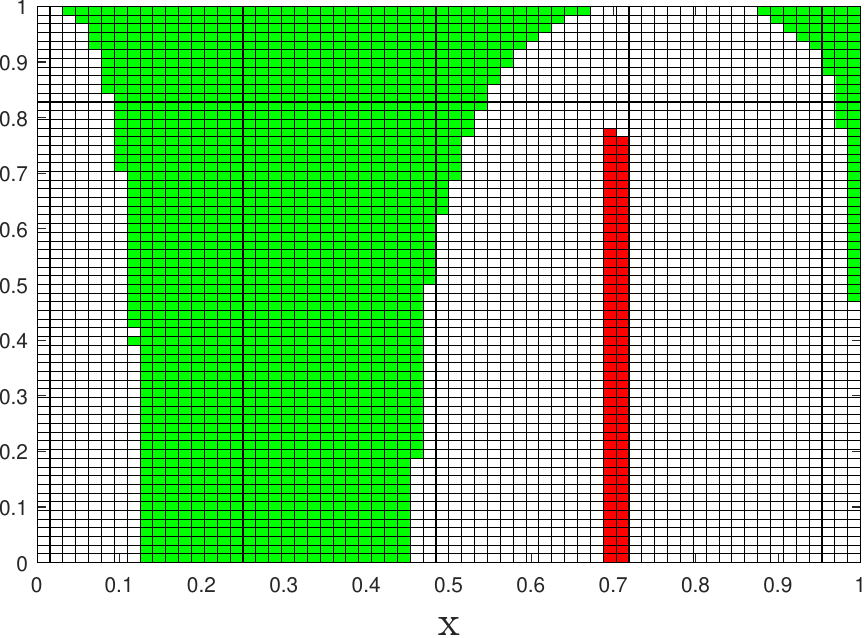}}
\subfloat{\includegraphics[width = 0.24\textwidth]{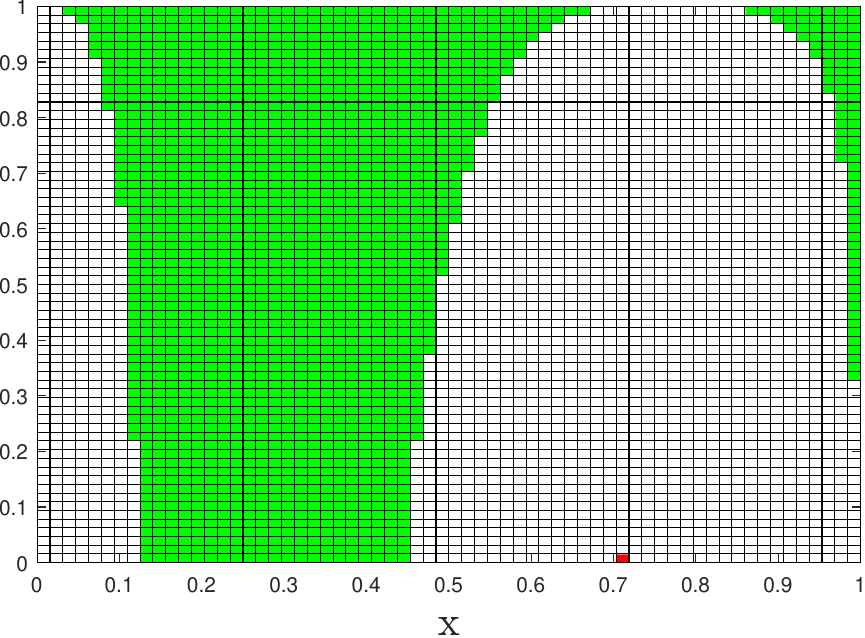}}
\subfloat{\includegraphics[width = 0.24\textwidth]{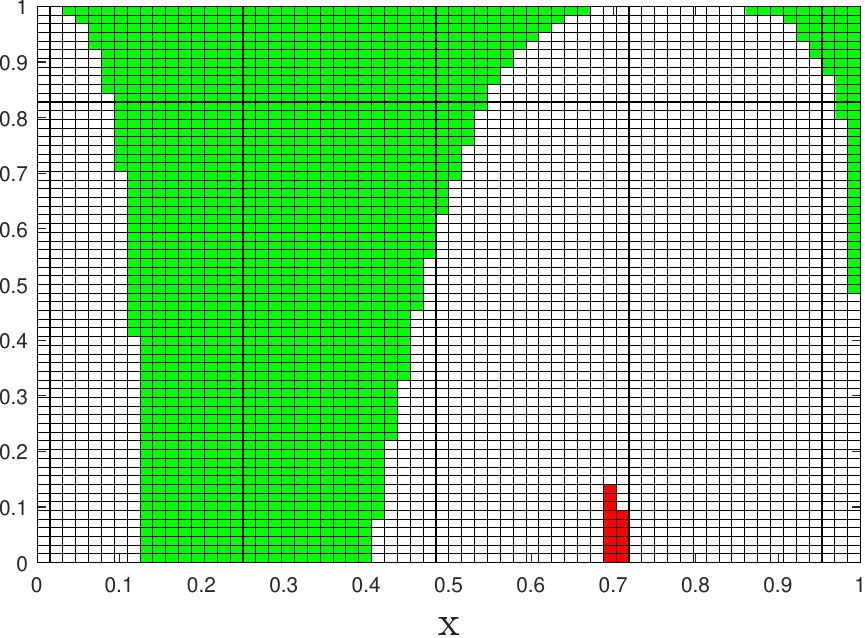}}
\subfloat{\includegraphics[width = 0.24\textwidth]{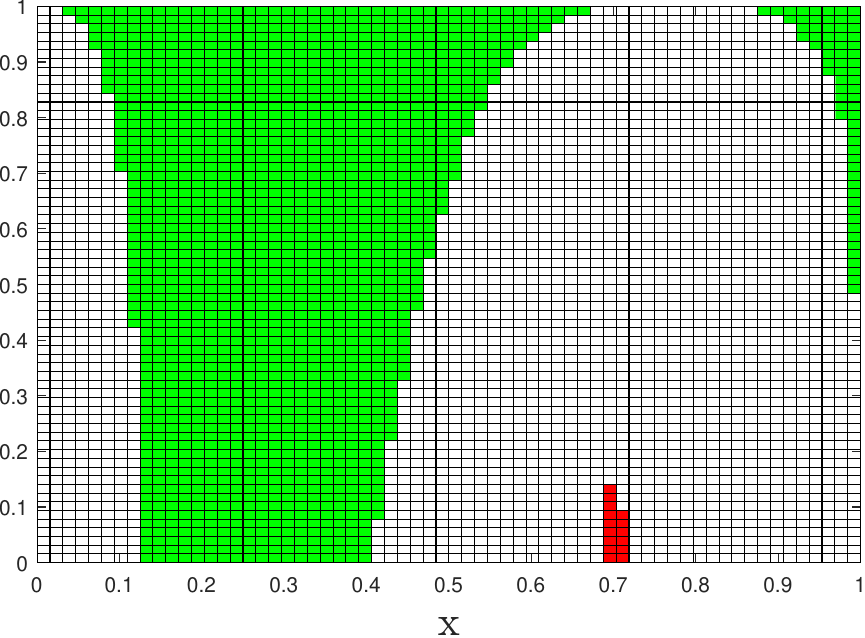}}
\caption{Results for the BFGS solver: 3d view (top row) and top view (bottom row) of the numerical solution $(t,\x) \rightarrow \ouhkRL(t,\x,\p)$ (from left to right) for $p=0,0.5,1$ and of the average $(t,\x)\rightarrow \frac12 (\ouhkRL(t,\x,0) + \ouhkRL(t,\x,1))$ with obstacles respectively represented by  $\frac12 f_1(\x) + \frac12f_2(\x)$, $\frac12 h_1(\x) + \frac12h_2(\x)$.
}
\label{fig:freeBoundary}
\end{figure}

The results in \Cref{fig:freeBoundary} show that the writer of the contract would exercise when the price of the underlying stock assumes value $\ln(K) \sim 0.7$. This is consistent with the results in \cite{KuhnKyprianou}. Also, the buyer would stop when the price of the stock is sufficiently low. This is also consistent to the findings of \cite{KuhnKyprianou}. Interestingly, we also see from \Cref{fig:freeBoundary} that the waiting region (gray area) is not connected, a feature which is not observed in the symmetric information case of \cite{KuhnKyprianou}. In particular, such a feature of the waiting region does not disappear when $p \downarrow 0$ or $p\uparrow 1$, i.e.\ in the two symmetric information cases. We repeated the experiment for different combinations of the discretization parameters but the obtained results were qualitatively very similar to those in \Cref{fig:freeBoundary}. We believe that the presence of a portion of the waiting region for small values of $x$ is a numerical artefact, depending on the accuracy of the nonlinear solver used for the neural network approximation. As a matter of fact, the waiting region is observed to be connected (for $p=0$, i.e.\ in the symmetric information case) in \Cref{fig:freeBoundary_exp4}, the latter being produced by using a (BR) nonlinear solver or a semi-Lagrangian scheme, for which we have higher accuracy and better convergence. From the obtained results we conclude that sufficient accuracy of the nonlinear solver used used to determine the neural network approximation of free boundaries is required in order to obtain reliable results.

In \Cref{fig:test4_freeBoundary}, we repeat the simulation with the same discretization parameters as above $\hht = \hhx = \hhp = 1.56\times 10^{-2}$
and employ a Feedforward Neural Network with one hidden layer and $10$ neurons (opposed to the $50$ neurons used to compute \cref{fig:freeBoundary} above) with the difference that we
solve the least-squares problem in \Cref{algo:exp3_ffn} using a Bayesian regularization (BR) algorithm (see \cite{foresee1997foresee,mackay1992bayesian}).
{\rev We remark that the good performance of the BR nonlinear solver is due to the built-in regularization property, where the regularization parameters are determined using the Bayesian rules, see \cite{mackay1992bayesian}.}
The results displayed in \Cref{fig:freeBoundary_exp4} for $\p=0$ reproduce the expected behavior of the free boundary.
For comparison in  \Cref{fig:freeBoundary_exp4} we also display the numerical solution obtained by a semi-Lagrangian algorithm  
with piecewise linear interpolation (i.e., a modification of \Cref{algo:exp3_ffn} which employs linear interpolation instead of the least-squares approximation, cf. \Cref{algo:exp2_sl}) for $\p=0$;
the results are in good qualitative agreement.
{\rev Comparison between the respective solutions of the SL and BR schemes at $\p=0$ is given in \Cref{fig:linear_test3_test4}.}

\begin{figure}[t!]
\subfloat{\includegraphics[width = 0.24\textwidth]{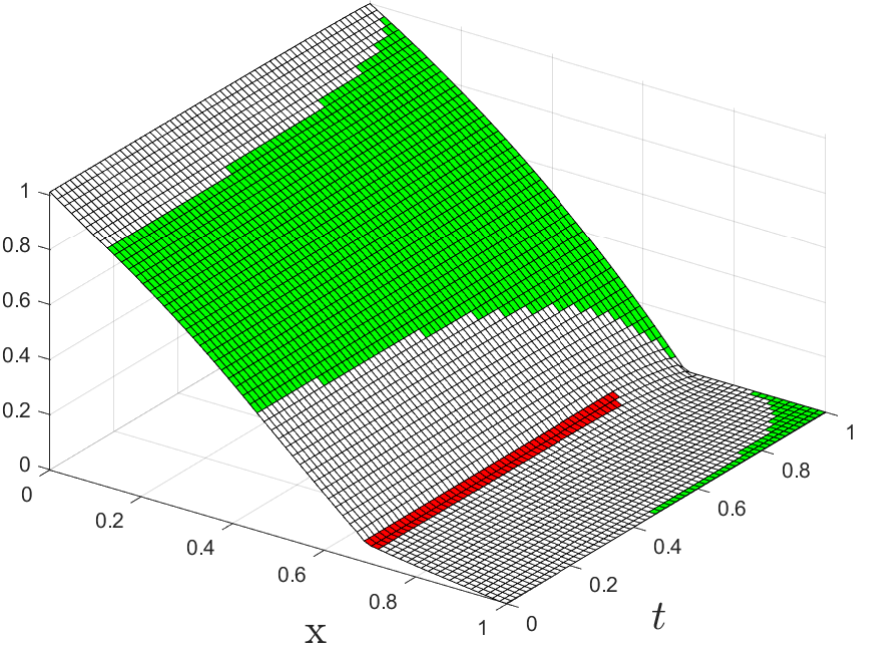}}
\subfloat{\includegraphics[width = 0.24\textwidth]{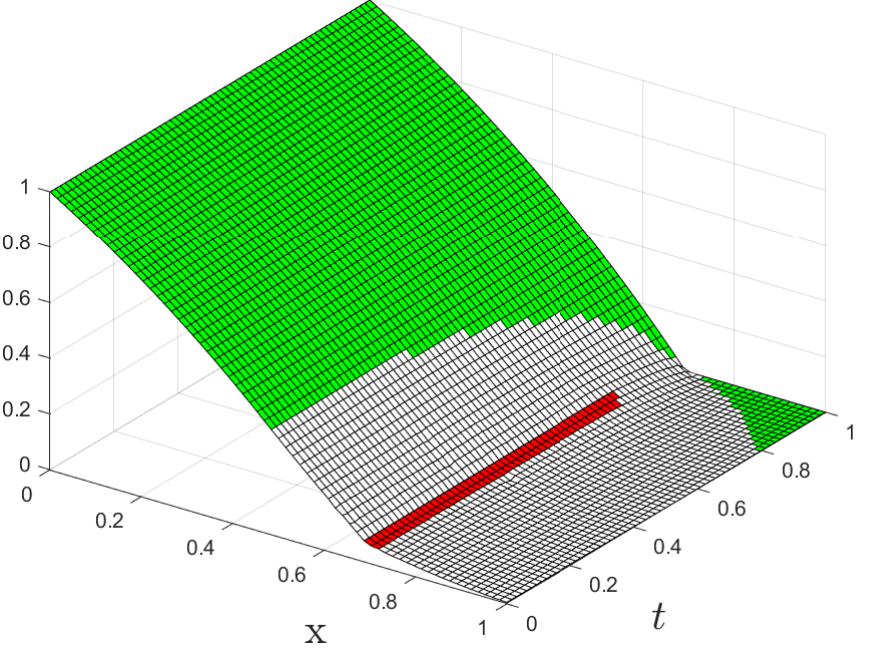}}
\subfloat{\includegraphics[width = 0.24\textwidth]{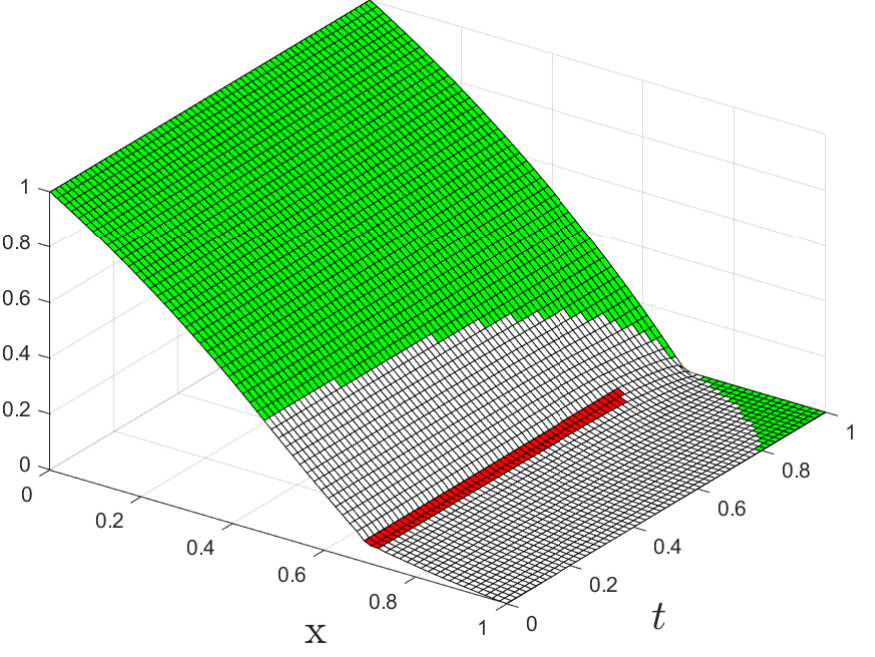}}
\\
\subfloat{\includegraphics[width = 0.24\textwidth]{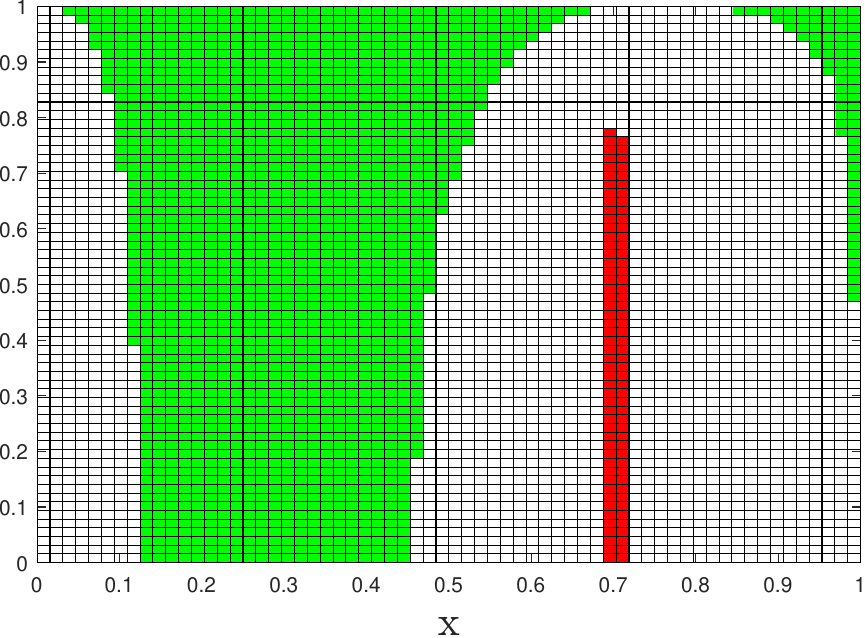}}
\subfloat{\includegraphics[width = 0.24\textwidth]{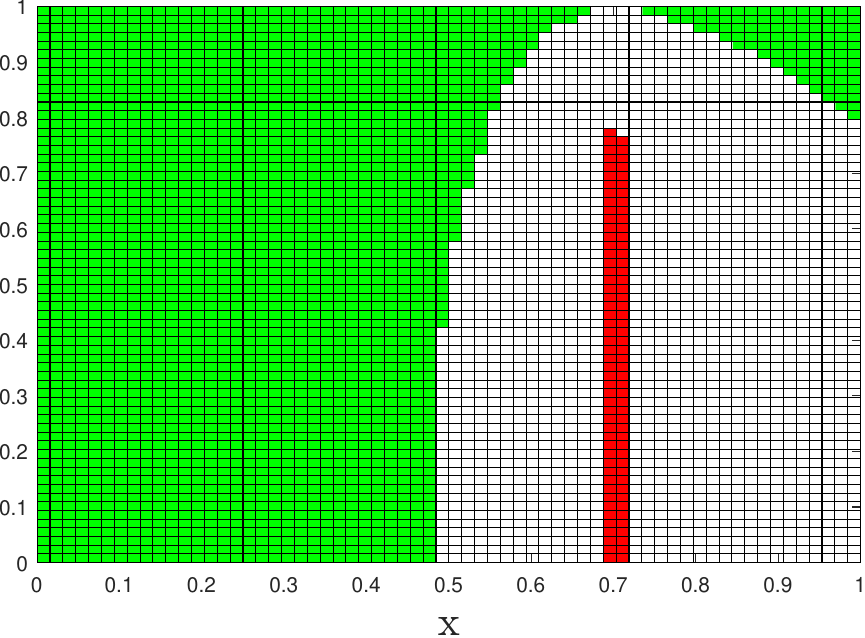}}
\subfloat{\includegraphics[width = 0.24\textwidth]{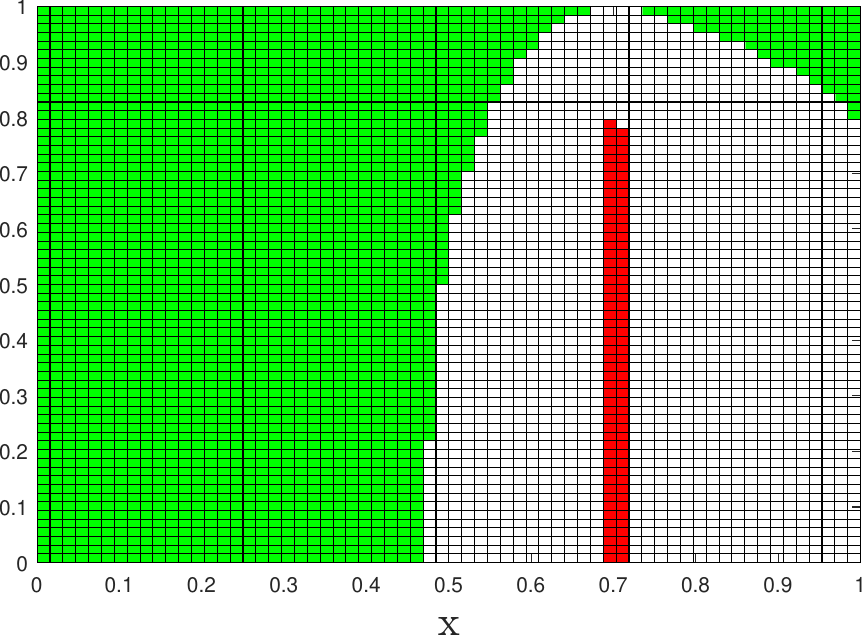}}
\caption{From left to right: solution for the BFGS solver,
solution for the BR solver and
the numerical solution computed with a semi-Lagrangian scheme.}
\label{fig:freeBoundary_exp4}
\end{figure}

For better illustration we display in \Cref{fig:snapshot} the graph of the numerical solution $\x\rightarrow \ouhkRL(t,\x,\p)$ for fixed $t=0.25$, $\p=0$. We observe that in the numerical solution the lower obstacle is active approximately between $\x = 0$  and $\x = 0.55$;
and the upper obstacle is active approximately between $\x = 0.55$ and $\x = 0.8$. 
\begin{figure}[t!]
\subfloat{\includegraphics[width = 0.4\textwidth]{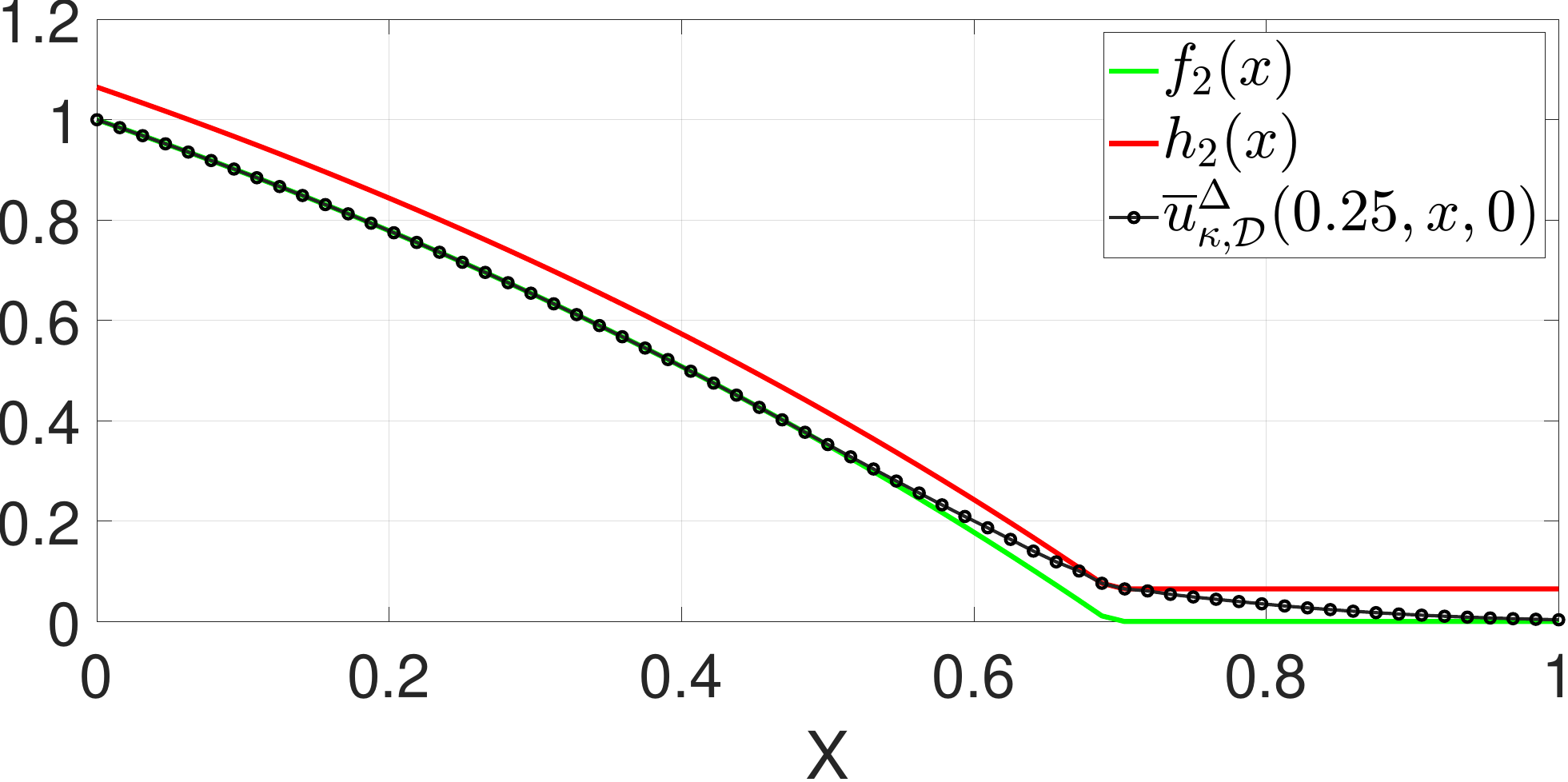}}
\caption{
Graph of the solution $\x \rightarrow \ouhkRL$ at $t=0.25$, $\p = 0$ with (BR) nonlinear solver.}
\label{fig:snapshot}
\end{figure}

\begin{figure}[t!]
	\includegraphics[width = 0.4\textwidth]{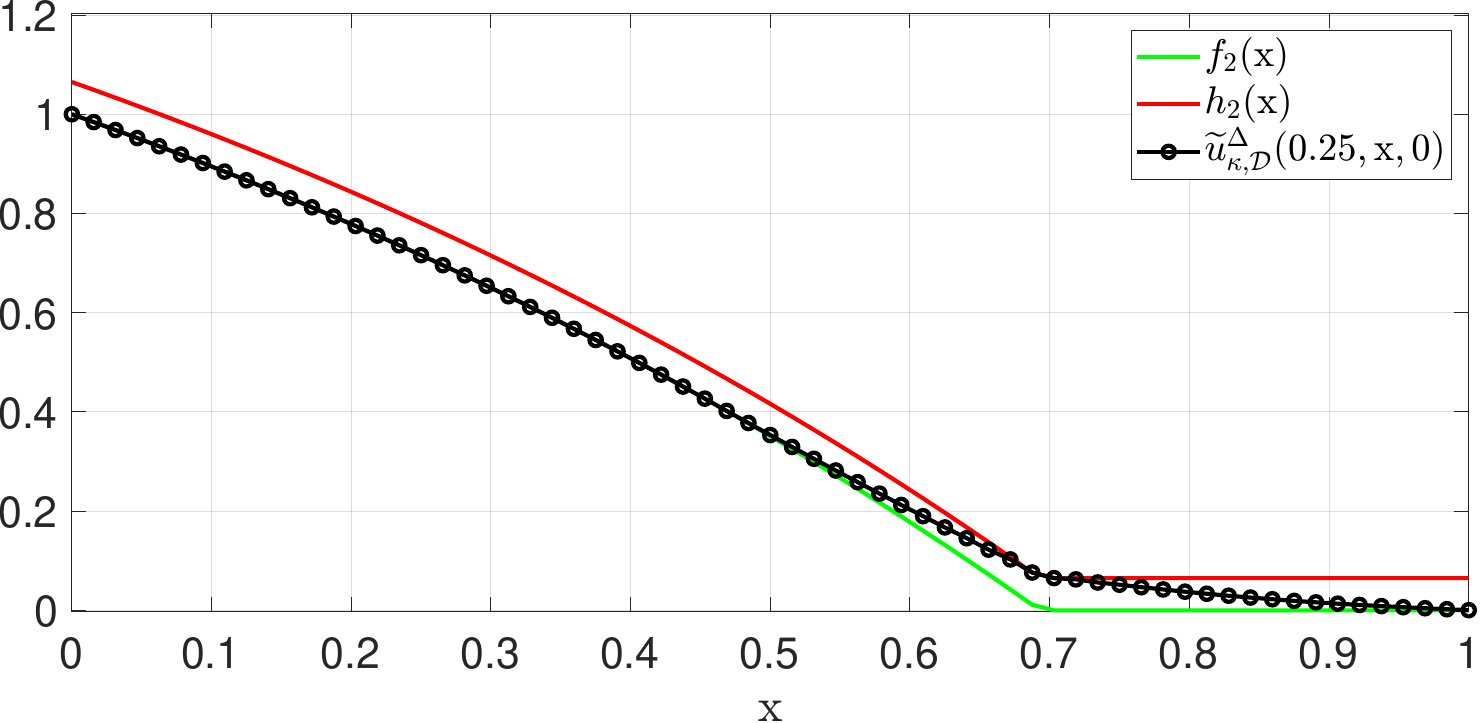}
\quad
	\includegraphics[width = 0.4\textwidth]{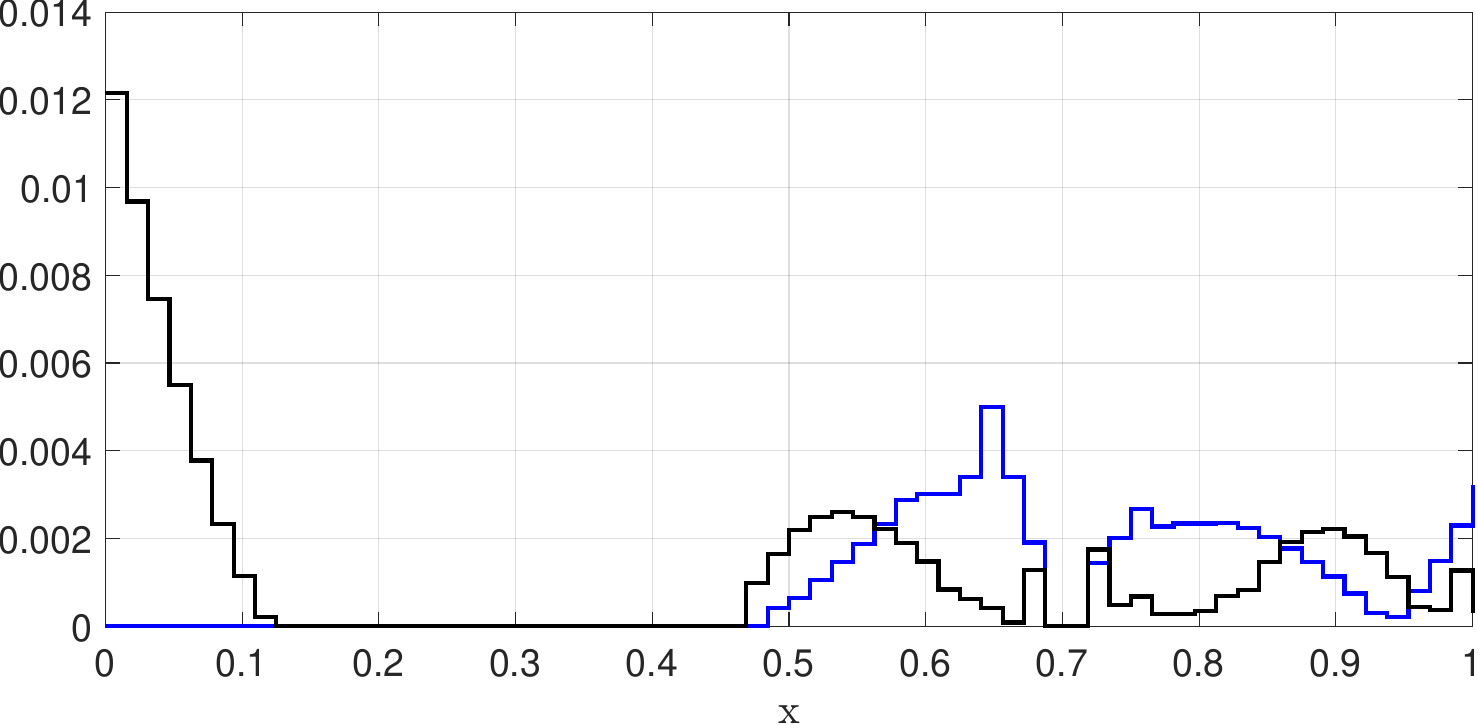}
	\caption{
		\textbf{(Left)} Graph of the solution $\x \rightarrow \tuh_{\kappa,\mD}$ at $t=0.25$, $\p = 0$ with the Semi-Lagrangian scheme.
		\textbf{(Right)} The pointwise error $\vert \tuh_{\kappa,\mD}(0.25,\x,0) - \ouhkRL(0.25,\x,0) \vert$.
		In blue, the solution $\ouhkRL$ were computed with the BR as a nonlinear solver, and in black, with the BFGS. }
	\label{fig:linear_test3_test4}
\end{figure}


\section*{Acknowledgement}
The work of the first and second author was funded by the Deutsche Forschungsgemeinschaft (DFG, German Research Foundation) -- Project-ID 317210226 -- SFB 1283.
The third author was funded by the FWF-Project 10.55776/P34681.

\bibliographystyle{abbrv}
\bibliography{dynkin.bib}

\begin{thebibliography}{10}

\bibitem{anil2019sorting}
C.~Anil, J.~Lucas, and R.~Grosse.
\newblock Sorting out {L}ipschitz function approximation.
\newblock In {\em International Conference on Machine Learning}, pages
  291--301. PMLR, 2019.

\bibitem{bally2003a}
V.~Bally and G.~Pag\`es.
\newblock A quantization algorithm for solving multi-dimensional discrete-time
  optimal stopping problems.
\newblock {\em Bernoulli}, 9(6):1003--1049, 2003.

\bibitem{bardi2009optimal}
M.~Bardi and I.~Capuzzo-Dolcetta.
\newblock {\em Optimal control and viscosity solutions of
  {H}amilton-{J}acobi-{B}ellman equations}.
\newblock Modern Birkh{\"a}user Classics. Birkh{\"a}user Boston, 2009.

\bibitem{banas2020numerical}
v.~Ba\v{n}as, G.~Ferrari, and T.~A. Randrianasolo.
\newblock Numerical approximation of the value of a stochastic differential
  game with asymmetric information.
\newblock {\em SIAM J. Control Optim.}, 59(2):1109--1135, 2021.

\bibitem{bishop2006pattern}
C.~M. Bishop.
\newblock {\em Pattern recognition and machine learning}.
\newblock Information Science and Statistics. Springer, New York, 2006.

\bibitem{cardaliaguet2009on}
P.~Cardaliaguet and C.~Rainer.
\newblock On a continuous-time game with incomplete information.
\newblock {\em Math. Oper. Res.}, 34(4):769--794, 2009.

\bibitem{carnicier1992convexity}
J.~M. Carnicer and W.~Dahmen.
\newblock Convexity preserving interpolation and {P}owell-{S}abin elements.
\newblock {\em Comput. Aided Geom. Design}, 9(4):279--289, 1992.

\bibitem{foresee1997foresee}
F.~Dan~Foresee and M.~Hagan.
\newblock Gauss-newton approximation to bayesian learning.
\newblock In {\em Proceedings of International Conference on Neural Networks
  (ICNN'97)}, volume~3, pages 1930--1935 vol.3, 1997.

\bibitem{dynkin22}
T.~De~Angelis, E.~Ekstr\"{o}m, and K.~Glover.
\newblock Dynkin games with incomplete and asymmetric information.
\newblock {\em Math. Oper. Res.}, 47(1):560--586, 2022.

\bibitem{DeAMePa}
T.~De~Angelis, N.~Merkulov, and J.~Palczewski.
\newblock On the value of non-{M}arkovian {D}ynkin games with partial and
  asymmetric information.
\newblock {\em Ann. Appl. Probab.}, 32(3):1774--1813, 2022.

\bibitem{Dynkin1967}
E.~Dynkin.
\newblock Game variant of a problem on optimal stopping.
\newblock {\em Soviet Math. Dokl.}, 10:270--274, 1969.

\bibitem{dynkin19}
F.~Gensbittel and C.~Gr\"{u}n.
\newblock Zero-sum stopping games with asymmetric information.
\newblock {\em Math. Oper. Res.}, 44(1):277--302, 2019.

\bibitem{maximilien2022approximation}
M.~Germain, H.~Pham, and X.~Warin.
\newblock Approximation error analysis of some deep backward schemes for
  nonlinear {PDE}s.
\newblock {\em SIAM Journal on Scientific Computing}, 44(1):A28--A56, 2022.

\bibitem{gill1989a}
P.~E. Gill, W.~Murray, M.~A. Saunders, and M.~H. Wright.
\newblock A practical anti-cycling procedure for linearly constrained
  optimization.
\newblock {\em Math. Programming}, 45(3, (Ser. B)):437--474, 1989.

\bibitem{gruen2012aprobabilistic}
C.~Gr\"{u}n.
\newblock A probabilistic-numerical approximation for an obstacle problem
  arising in game theory.
\newblock {\em Appl. Math. Optim.}, 66(3):363--385, 2012.

\bibitem{gruen2013on}
C.~Gr\"{u}n.
\newblock On {D}ynkin games with incomplete information.
\newblock {\em SIAM J. Control Optim.}, 51(5):4039--4065, 2013.

\bibitem{hure2020deep}
C.~Hur\'{e}, H.~Pham, and X.~Warin.
\newblock Deep backward schemes for high-dimensional nonlinear {PDE}s.
\newblock {\em Math. Comp.}, 89(324):1547--1579, 2020.

\bibitem{Kifer}
Y.~Kifer.
\newblock Game options.
\newblock {\em Finance Stoch.}, 4(4):443--463, 2000.

\bibitem{KuhnKyprianou}
C.~K\"{u}hn and A.~E. Kyprianou.
\newblock Callable puts as composite exotic options.
\newblock {\em Math. Finance}, 17(4):487--502, 2007.

\bibitem{Kyprianou}
A.~E. Kyprianou.
\newblock Some calculations for {I}sraeli options.
\newblock {\em Finance Stoch.}, 8(1):73--86, 2004.

\bibitem{laraki2004on}
R.~Laraki.
\newblock On the regularity of the convexification operator on a compact set.
\newblock {\em J. Convex Anal.}, 11(1):209--234, 2004.

\bibitem{LempaMatomaki}
J.~Lempa and P.~Matom\"{a}ki.
\newblock A {D}ynkin game with asymmetric information.
\newblock {\em Stochastics}, 85(5):763--788, 2013.

\bibitem{liu1989on}
D.~C. Liu and J.~Nocedal.
\newblock On the limited memory {BFGS} method for large scale optimization.
\newblock {\em Math. Programming}, 45(3, (Ser. B)):503--528, 1989.

\bibitem{mackay1992bayesian}
D.~J.~C. MacKay.
\newblock {Bayesian Interpolation}.
\newblock {\em Neural Computation}, 4(3):415--447, 05 1992.

\bibitem{oberman2007the}
A.~M. Oberman and C.~W. Craig.
\newblock The convex envelope is the solution of a nonlinear obstacle problem.
\newblock {\em Proc. Amer. Math. Soc}, 135(6):1689--1694, 2007.

\bibitem{touzi2013optimal}
N.~Touzi.
\newblock {\em Optimal stochastic control, stochastic target problems, and
  backward {SDE}}, volume~29 of {\em Fields Institute Monographs}.
\newblock Springer, New York; Fields Institute for Research in Mathematical
  Sciences, Toronto, ON, 2013.

\bibitem{yosida2013functional}
K.~Yosida.
\newblock {\em Functional Analysis}.
\newblock Grundlehren der mathematischen Wissenschaften. Springer Berlin
  Heidelberg, 2013.

\end{thebibliography}
\end{document}